\documentclass[11pt,a4paper,reqno]{amsart}
\usepackage{amsthm,amsmath,amsfonts,amssymb,xcolor,amsxtra,appendix,bookmark,bm,mathrsfs}
\usepackage[latin1]{inputenc}
\usepackage{graphicx}
\usepackage{booktabs,microtype,flafter,array,makecell}

\usepackage{placeins}
\graphicspath{{figures/}}
\usepackage{hyperref}
\hypersetup{hidelinks,pdftitle={Residual-based error control for neural Galerkin reconstruction of minimum-norm wave controls},
  pdfauthor={Nguyen Gia Hien}}

\theoremstyle{plain}
\newtheorem{theorem}{Theorem}
\newtheorem{lemma}[theorem]{Lemma}
\newtheorem{corollary}[theorem]{Corollary}
\newtheorem{proposition}[theorem]{Proposition}
\theoremstyle{remark}
\newtheorem{remark}[theorem]{Remark}

\newcommand{\codeDOI}{10.5281/zenodo.23288937}
\def\codeDOIvtwo{10.5281/zenodo.23984763}
\newif\ifcodereleased
\ifx\codeDOIvtwo\empty\codereleasedfalse\else\codereleasedtrue\fi

\allowdisplaybreaks

\makeatletter
\newcommand{\sref}[1]{\@ifundefined{suppref@#1}{\textbf{??}}{\@nameuse{suppref@#1}}}
\newcommand{\seqref}[1]{\textup{(\sref{#1})}}
\makeatother
\expandafter\def\csname suppref@app:numerical-diagnostics\endcsname{S1}
\expandafter\def\csname suppref@tab:appendix-configurations\endcsname{S1}
\expandafter\def\csname suppref@app:code-data\endcsname{S2}
\expandafter\def\csname suppref@app:assembly-diagnostics\endcsname{S3}
\expandafter\def\csname suppref@tab:assembly-settings\endcsname{S2}
\expandafter\def\csname suppref@tab:resolution-examples123\endcsname{S3}
\expandafter\def\csname suppref@app:ritz-consistency-diagnostics\endcsname{S4}
\expandafter\def\csname suppref@tab:ritz-consistency-diagnostics\endcsname{S4}
\expandafter\def\csname suppref@app:example1-diagnostics\endcsname{S5}
\expandafter\def\csname suppref@fig:peanut-geometry\endcsname{S1}
\expandafter\def\csname suppref@fig:peanut-errors\endcsname{S2}
\expandafter\def\csname suppref@fig:peanut-states\endcsname{S3}
\expandafter\def\csname suppref@fig:peanut-modewise\endcsname{S4}
\expandafter\def\csname suppref@fig:peanut-budget\endcsname{S5}
\expandafter\def\csname suppref@tab:peanut-modewise\endcsname{S5}
\expandafter\def\csname suppref@app:example2-diagnostics\endcsname{S6}
\expandafter\def\csname suppref@fig:lens-errors\endcsname{S6}
\expandafter\def\csname suppref@fig:lens-modewise\endcsname{S7}
\expandafter\def\csname suppref@fig:lens-budget\endcsname{S8}
\expandafter\def\csname suppref@tab:lens-modewise\endcsname{S6}
\expandafter\def\csname suppref@tab:lens-finite-diagnostics\endcsname{S7}
\expandafter\def\csname suppref@tab:lens-forward-checks\endcsname{S8}
\expandafter\def\csname suppref@app:example3-diagnostics\endcsname{S7}
\expandafter\def\csname suppref@fig:graded-geometry\endcsname{S9}
\expandafter\def\csname suppref@tab:graded-cutoff\endcsname{S9}
\expandafter\def\csname suppref@fig:graded-errors\endcsname{S10}
\expandafter\def\csname suppref@fig:graded-states\endcsname{S11}
\expandafter\def\csname suppref@fig:graded-modewise\endcsname{S12}
\expandafter\def\csname suppref@fig:graded-budget\endcsname{S13}
\expandafter\def\csname suppref@tab:graded-modewise\endcsname{S10}
\expandafter\def\csname suppref@tab:graded-degree-check\endcsname{S11}
\expandafter\def\csname suppref@tab:graded-forward-checks\endcsname{S12}
\expandafter\def\csname suppref@app:example4\endcsname{S8}
\expandafter\def\csname suppref@app:timing\endcsname{S8}
\expandafter\def\csname suppref@tab:resolution\endcsname{S13}
\expandafter\def\csname suppref@tab:ex4-modewise\endcsname{S14}
\expandafter\def\csname suppref@tab:ex4-sweep\endcsname{S15}
\expandafter\def\csname suppref@app:certificate\endcsname{S9}
\expandafter\def\csname suppref@app:ablation\endcsname{S10}
\expandafter\def\csname suppref@app:extension-reproduction\endcsname{S11}
\expandafter\def\csname suppref@app:truncated\endcsname{S12}
\expandafter\def\csname suppref@fig:truncated-benchmark\endcsname{S14}
\expandafter\def\csname suppref@app:proofs\endcsname{S13}
\expandafter\def\csname suppref@app:proofs-wave\endcsname{S13.1}
\expandafter\def\csname suppref@Wave4\endcsname{S1}
\expandafter\def\csname suppref@Wave5\endcsname{S2}
\expandafter\def\csname suppref@app:X-injective\endcsname{S13.1}
\expandafter\def\csname suppref@app:square-density\endcsname{S13.1}
\expandafter\def\csname suppref@app:density-proof\endcsname{S13.2}
\expandafter\def\csname suppref@app:galerkin-algebra\endcsname{S13.3}
\expandafter\def\csname suppref@app:proofs-perturbation\endcsname{S13.4}
\expandafter\def\csname suppref@eq:exact-coefficient-uniform-bound\endcsname{S3}
\expandafter\def\csname suppref@eq:coefficient-error-bound\endcsname{S4}
\expandafter\def\csname suppref@eq:reconstruction-constant\endcsname{S5}
\expandafter\def\csname suppref@eq:fully-discrete-exact-integral-spectrum\endcsname{S6}
\expandafter\def\csname suppref@eq:fully-discrete-inverse-bound\endcsname{S7}
\expandafter\def\csname suppref@eq:exact-integral-control-bound\endcsname{S8}
\expandafter\def\csname suppref@eq:exact-integral-coefficient-bound\endcsname{S9}
\expandafter\def\csname suppref@erroreqwave\endcsname{S10}
\expandafter\def\csname suppref@eestimatewave1\endcsname{S11}
\expandafter\def\csname suppref@app:proofs-time-dependent\endcsname{S13.5}
\expandafter\def\csname suppref@app:moved-floats\endcsname{S14}
\expandafter\def\csname suppref@fig:certificate\endcsname{S15}
\expandafter\def\csname suppref@tab:localized-certificate\endcsname{S16}
\expandafter\def\csname suppref@fig:localized-certificate\endcsname{S16}
\expandafter\def\csname suppref@tab:adaptive-stopping\endcsname{S17}
\expandafter\def\csname suppref@fig:adaptive-stopping\endcsname{S17}
\expandafter\def\csname suppref@fig:ex4-errors\endcsname{S18}
\expandafter\def\csname suppref@tab:ablation\endcsname{S18}
\expandafter\def\csname suppref@fig:ablation\endcsname{S19}
\expandafter\def\csname suppref@tab:matched-enrichment\endcsname{S19}
\expandafter\def\csname suppref@fig:matched-enrichment\endcsname{S20}
\expandafter\def\csname suppref@fig:production-multiseed\endcsname{S21}

\title[Residual error control for neural Galerkin wave controls]{Residual-based error control for neural Galerkin reconstruction of minimum-norm wave controls}



\author[N. G. Hien]{Nguyen Gia Hien}
\address{Department of Mathematics, Texas A\&M University, College Station, TX 77843, USA}
\email{giahien-nguyen@tamu.edu} 

\begin{document}
\date{November 3, 2026}

\begin{abstract}
We analyse how errors in independently approximated adjoint waves propagate into the Hilbert Uniqueness Method (HUM) reconstruction of the minimum-$L^2$-norm interior control of a variable-coefficient wave equation. The central result is an identity: the error of the control assembled exactly from any approximate adjoints splits orthogonally into the distance of the minimum-norm control to their observation span and a moment-inconsistency term generated by their interior wave residuals. From it we derive a posteriori bounds for the returned control, including assembly and solve errors, and a continuous accept-or-refine test for the network directions returned by training; under continuous observability, a fully discrete estimate separates truncation, adjoint and assembly errors. The adjoints are computed here by adaptive Galerkin neural networks, but the bounds apply to any computed adjoints. In one-dimensional benchmarks with known minimum-norm controls, all inputs are enclosed by ball arithmetic. Without truncation the residual bounds exceed the true error by factors of 9 to 115; with sixteen adjoints outside the trial class, a truncated terminal space and all inputs computed from the target data, by about 1.08 to 6.50. In two and three dimensions, including a time-dependent coefficient, the neural controls closely reproduce reference controls of the same truncation, and at equal neuron counts trained hidden features reduce the error more than tenfold relative to random features. Classical solvers are much cheaper in all tested cases, and certification is confined to the one-dimensional benchmarks.
\end{abstract}
	
\maketitle
\tableofcontents

\section{Introduction}

Controllability of partial differential equations asks whether a distributed system can be steered from a given initial state to a prescribed target by an input acting on only part of the domain or its boundary. For waves, the answer depends on how energy propagates through the medium and reaches the control region. We consider interior control of a wave equation with a spatially varying coefficient and seek the control of minimum $L^2$-norm among those that reach the target at time $T$.

The Hilbert Uniqueness Method (HUM) of Lions~\cite{lions1988exact} constructs this control from solutions of the homogeneous adjoint equation, with an observability inequality supplying the required coercivity. Observability of waves is governed by geometric optics: the geometric control condition requires every generalized ray to meet the observation region before time~$T$~\cite{rauch1974exponential,blr1992sharp,burq1997condition}. A variable coefficient can change the local wave speed, bend the rays, and thus affect this condition. Throughout the paper, observability of the continuous problem is taken as a hypothesis.

Numerical approximation adds a further difficulty. A scheme may approximate the wave evolution well and still fail to preserve observability: spurious high-frequency waves in semidiscrete schemes can be poorly observed, so convergent forward solvers do not necessarily produce convergent controls~\cite{infante1999boundary,zuazua2005propagation}. Remedies include filtering, regularization, mixed finite elements and two-grid methods~\cite{glowinski1990numerical,asch1998geometrical,munch2005uniformly,castro2008mixed,glowinski2008exact,ignat2009twogrid,ervedoza2013numerical}, as well as space-time variational formulations~\cite{cindea2013numerical,cindea2015mixed}.

Neural networks offer trial functions that can be defined directly in space-time without a mesh or a time-marching scheme for their representation. Residual-based approaches include physics-informed neural networks~\cite{raissi2019physics} and the Deep Galerkin Method~\cite{sirignano2018dgm}, while variational approaches include the Deep Ritz method~\cite{ew2018deep} and deep least-squares methods~\cite{cai2020deep}. Neural approaches have also been studied for PDE control~\cite{garciacervera2023control,mowlavi2023optimal}. The Galerkin neural networks of Ainsworth and Dong~\cite{AinsworthDong2021} combine training with a classical variational structure: trained networks successively enrich a trial space, and a linear Galerkin solve recomputes the coefficients after each enrichment. This construction is the starting point of the present work. Extended Galerkin neural networks carry this framework, with error control, over to singular variational problems~\cite{AinsworthDong2025}. Greedy training of networks neuron by neuron is analysed in~\cite{siegel2023greedy}, and a posteriori error certificates for physics-informed networks are derived in~\cite{hillebrecht2022certified}. The present contribution concerns coupling neural Galerkin adjoint approximation with HUM control reconstruction and the propagation of adjoint errors to that reconstruction.

\paragraph{Relation to established constructions.}
The distinction from earlier methods is the error-controlled reconstruction
from independently inexact adjoints, not the introduction of Galerkin
projection, HUM, Green's identity, interval arithmetic or computer-assisted
certification of steering controls themselves.
Table~\ref{tab:related-work-scope} specifies the relevant division of scope.
It is a comparison of the stated problems and mechanisms, rather than an
exhaustive mathematical-priority claim.
\FloatBarrier
\begin{table}[!htbp]\centering\small
\caption{Established components and the scope of the present analysis.}
\label{tab:related-work-scope}
\begin{tabular}{@{}p{.29\textwidth}p{.66\textwidth}@{}}\toprule
Related construction & Distinction relevant here\\\midrule
Ainsworth--Dong~\cite{AinsworthDong2021,AinsworthDong2025}
& Adaptive neural trial spaces and variational error control are the starting
point. Here independently approximated wave adjoints additionally perturb a
separate HUM reconstruction and its moment equations.\\[3pt]
Lebeau--Nodet~\cite{LebeauNodet2010}
& Spectral HUM reconstruction, truncation and conditioning are established.
Our estimates additionally account for interior residuals of inexact adjoints
and errors in the actual reconstruction integrals and solve.\\[3pt]
C\^indea--M\"unch~\cite{cindea2015mixed}
& Space--time state--adjoint variational approximation of minimum-norm wave
controls is established. We retain separate adjoint enrichment and HUM solves,
and explicitly audit residual-induced moment inconsistency.\\[3pt]
Burman--Feizmohammadi--Oksanen~\cite{burman2020fully}; Burman--Feizmohammadi--M\"unch--Oksanen~\cite{burman2023spacetime}
& Stabilized fully discrete and space--time finite element methods for null controls, with a priori convergence rates in the mesh size under the geometric control condition and regularity of the exact solution, assuming exact evaluation of the discrete forms. The present bounds are a posteriori, concern the computed fields and include assembly and solve errors, but give no rates in a mesh parameter.\\[3pt]
Montaner--M\"unch~\cite{montaner2019first}
& First-order mixed space--time formulation for null controls, extending the approach of~\cite{cindea2015mixed}, with an observability inequality for the first-order system and strong convergence of $H^1$-conforming approximations. Here the state--adjoint coupling is not discretized; independently computed adjoints are combined through HUM moments.\\[3pt]
Nguyen--Tran--Tu~\cite{NguyenTranTu2026} (v2)
& Semilinear waves with constant coefficients on a cube, in a Dirichlet
spectral Galerkin space; controls from shallow networks for the temporal
modal coefficients, from source-frozen HUM iteration or from endpoint least
squares, with Newton--HUM correction. A Newton--Kantorovich-type a posteriori
theorem (Theorem~3.1 there), applied with interval enclosures of the
trajectory and of the linearized controllability Gramian, proves that an
exact steering control of the specified Galerkin system lies near a computed
candidate (reported for an eight-mode one-dimensional problem). Here the equation is linear with variable, possibly
time-dependent, coefficients, and the bounded quantity is the distance of the
returned control to the minimum-norm control of the continuous PDE.\\[3pt]
Greedy neural training~\cite{siegel2023greedy}
& Adding single neurons from a dictionary yields algebraic rates under
variation-space assumptions. Our geometric rate rests on the much stronger
residual-capture hypothesis for whole-network directions
(Remark~\ref{rem:log-count-scope}).\\\bottomrule
\end{tabular}
\end{table}
\FloatBarrier

A HUM method built from learned adjoint waves raises a question of its own. The control is obtained from a linear system assembled from the adjoints on the control region, and errors in the adjoints perturb this system. The central issue is to determine how accurate the learned adjoints must be as more terminal modes are retained, and how this accuracy should be balanced against truncation and numerical assembly errors.

We address this question with an adaptive Galerkin neural-network method. Each adjoint wave, associated with one terminal basis function, is approximated by minimizing a least-squares functional for the wave equation and its terminal data. With consistent continuous terminal data and exact integration, this functional equals one half of the squared adjoint error in the analysis norm; the experiments evaluate a Ritz-based version by quadrature, and Section~\ref{sec:numerics} makes the difference explicit. The adjoints are computed independently and their restrictions to the control region define a finite linear system for the HUM coefficients. The analysis starts from the observability of the continuous equation rather than of a discrete scheme. With an energy-orthonormal terminal basis the exact reconstruction is stable uniformly in the number $n$ of retained modes, and learning and assembly errors enter as perturbations. The projection and matrix-perturbation arguments are standard; what is specific to the wave equation is an identity that expresses the inconsistency of the moment equations through the interior residual of the approximate adjoints.

\paragraph{Central result.} Let $v_1,\dots,v_n$ be \emph{any} fields in the least-squares space, for instance trained networks, with linearly independent restrictions to $Q_\omega$, let $V$ be the span of these restrictions, and let $f_V$ be the control obtained by solving the HUM system with the exact Gram matrix $G$ and the exact moments $L(v_k)$. Then (Proposition~\ref{prop:wave-moment-residual})
\[
 \|f_V-f^*\|_{L^2(Q_\omega)}^2=\operatorname{dist}(f^*,V)^2+r^\top G^{-1}r,\qquad r_k=(u^*,\mathcal Pv_k)_{L^2(Q)},
\]
where $u^*$ is the optimally controlled state. Errors in the terminal data of the $v_k$ only change the space $V$; the moment equations are corrupted only through the interior residuals $\mathcal Pv_k$, tested against $u^*$. When the Gram matrix and the moments are assembled and solved inexactly, Corollary~\ref{cor:actual-gram-certificate} bounds the error of the control actually returned by validated quantities of the returned fields, a bound for $\|u^*\|$ and a bound for $\operatorname{dist}(f^*,V)$, without any smallness condition on the adjoint errors. Two features distinguish this from the perturbation estimates of spectral HUM~\cite{LebeauNodet2010} and from the a priori analysis of convergent space--time discretizations~\cite{burman2020fully,burman2023spacetime,montaner2019first}: the moment defect is identified exactly, rather than bounded through the energy-norm errors of the adjoints, and the bound applies a posteriori to the fields that were actually computed, whatever their origin, including the errors of numerical integration and of the linear solve. The computer-assisted certificates of~\cite{NguyenTranTu2026} answer a different question: they prove that an exact steering control of a fixed semilinear Galerkin system exists near a computed candidate, whereas here the target is the minimum-norm control of the continuous linear equation and the bounded quantity is the distance of the returned control to it. The analysis is therefore independent of the neural approximation; the networks enter only as a mechanism for constructing the trial fields, and their contribution is assessed experimentally.

What the paper contributes is the complete chain from the residuals of computed adjoints to the error of the control that is actually returned. Relative to standard HUM perturbation arguments, it adds three quantitative elements: the moment defect $r_k=(u^*,\mathcal Pv_k)_{L^2(Q)}$ produced by the interior residuals, the errors of the numerical assembly and of the linear solve, and continuous checks of the fields actually returned by training. Together they support a decision that a perturbation estimate alone does not: whether a returned control meets a prescribed tolerance or further enrichment is required. In the three-seed audit of Section~\ref{sec:validated-gate-study}, for example, the smallness condition of the synthesis bound (Corollary~\ref{cor:wave-residual-certificate}) fails for one seed, while the bound of Corollary~\ref{cor:actual-gram-certificate}, computed from the actual Gram matrix, still encloses the error of its returned control. Since that bound exceeds the requested tolerance $10^{-3}$ and the bounds of the other two seeds do not, a certificate-based controller refines only that seed; the exploratory stopping test of Section~\ref{sec:adaptive-stopping} carries this out. The two terms of the identity can also be evaluated separately in the benchmarks (Table~\ref{tab:error-decomposition}): without truncation the moment-inconsistency term carries $61$--$98\%$ of the squared error, so in these stored fields the moment term is substantial and the reconstructed control is not well described by a projection of $f^*$ onto the span of the adjoints. Each link uses standard tools, and coupling neural Galerkin adjoints with HUM is not by itself the contribution. The one-dimensional benchmarks are the setting in which every input of the chain is enclosed. Our main results are:
\begin{enumerate}
    \item a wave-residual identity and an orthogonal decomposition of the control error into an observation-space approximation part and a moment-inconsistency part (Proposition~\ref{prop:wave-moment-residual}), with a posteriori bounds that include numerical assembly and solve errors, either through a synthesis perturbation or through validated spectral bounds of the actual Gram matrix, and a complementary bound from terminal steering and final-field residuals;
    \item a validated acceptance test linking the directions actually returned by training, and their projection defects, to continuous error contraction (Proposition~\ref{prop:validated-acceptance});
    \item a fully discrete a priori estimate separating terminal-space truncation, adjoint approximation and assembly errors, with a worst-case sufficient stability condition, and conditional convergence rates for spectral terminal bases;
    \item one-dimensional verification benchmarks with full and localized observation in which every ingredient of the a posteriori bounds is enclosed by ball arithmetic, including a three-seed audit of the continuous acceptance test and a benchmark with sixteen adjoints outside the trial class, a nonzero truncation error and a non-diagonal Gram matrix, in which every certificate input, including a bound for the truncation error, is computed from the target data;
    \item experiments on a nonconvex domain, a square with a slow lens, a three-dimensional ball and a time-dependent coefficient, with classical reference solvers, a cost comparison and a comparison of trained and random hidden features.
\end{enumerate}

The scope of these results is limited, as detailed in Section~\ref{sec:limitations}. The continuous certificates are evaluated only for explicit one-dimensional benchmarks with a constant coefficient and a known observability constant. In the two- and three-dimensional examples the errors are measured against finite reference models, and the sufficient a priori stability condition is not verified; it would fail at the largest tested $n$ (Section~\ref{sec:example2}). The convergence rates rest on a residual-capture hypothesis that is not proved for any optimizer, and the classical comparators are much cheaper in every tested example (Table~\ref{tab:cost}).

In the two- and three-dimensional examples, the neural reconstruction errors at $n=32$ are two to four orders of magnitude smaller than the truncation errors of the reference models, so that the neural approximation is not the dominant source of error at these dimensions. In the fourth example the lens pulsates in time, the separated modal formula no longer solves the adjoint equation, and the controls obtained by freezing the coefficient differ from those of a Ritz--Galerkin solver with time integration by about $0.16$ in $L^2(Q_\omega)$, whereas the neural controls differ from them by about $10^{-2}$. At equal neuron counts, trained hidden features reduce the distance to the M3 control by factors of about $10$--$30$ in the tested configurations. A separate comparison at similar CPU time suggests an improvement of about $4$--$7$, but uses different seeds, stopping rules and assessment settings.

The paper is organized as follows. Section~\ref{sec:wave} introduces the control problem and the least-squares formulation of the adjoint equations. Section~\ref{sec:method} presents the method and its stability, error, and convergence analysis. Section~\ref{sec:numerics} reports the numerical experiments, including a time-dependent coefficient and a comparison of computational costs, and Section~\ref{sec:conclusion} concludes the paper. Standard proofs, the proof of the density lemma, implementation details, further diagnostics and reproduction instructions are given in the Supplementary Material.

\section{Controlling  Wave Equations}\label{sec:wave}
	
Let $\Omega \subset \mathbb{R}^d$ be a bounded open domain with boundary of class $C^\infty$, and let $\omega \subset \Omega$ be a nonempty open subset. For a given final time $T>0$, we consider the following nonhomogeneous wave equation:
\begin{equation}\label{Wave1}
\begin{cases}
	u_{tt} - \nabla \cdot \big(a(x)\nabla u\big) + f\,\mathbf{1}_{\omega} = 0 & \text{in } (0,T)\times \Omega, \\[2mm]
	u = 0  & \text{on } (0,T)\times \partial\Omega, \\[1mm]
	u(0,\cdot) = u^0,\quad u_t(0,\cdot) = u^1 & \text{in } \Omega .
\end{cases}
\end{equation}
Smoothness up to $\overline\Omega$ is understood in the usual extension sense; in particular, $a\in W^{1,\infty}(\Omega)$. We assume that $a \in C^\infty(\overline\Omega)$ satisfies the uniform ellipticity condition \[0 < a_0 \le a(x) \le a_1 \qquad \text{for all } x \in \overline{\Omega},\] for some positive constants $a_0$ and $a_1$. Here $u = u(t,x)$ denotes the state variable and $f = f(t,x)$ is an interior control acting on the subdomain $\omega$, with $\mathbf{1}_{\omega}$ denoting the characteristic function of $\omega$. The subscript $t$ denotes differentiation with respect to time. The objective is to influence the dynamics of the system through the control input $f$ supported in $\omega$.
	
{\bf  Control Problem \cite{zuazua2006controllability}: } Given the initial condition $u^0,u^1$ and the final condition $u^0_T,u^1_T,$ we want to find $f$ such that 
\begin{equation}\label{Wave2}
	u(T,\cdot) = u^0_T,\quad u_t(T,\cdot) = u^1_T \text{ in } \Omega.
\end{equation}

We first give a characterization of the exact controls for \eqref{Wave2}, following the HUM duality framework \cite{lions1988exact}. Let $\langle \cdot , \cdot \rangle_{-1,1}$ denote the duality pairing between $H^{-1}(\Omega)$ and $H_0^1(\Omega)$. For any $(\varphi_T^0,\varphi_T^1)\in L^2(\Omega)\times H^{-1}(\Omega)$, we consider the following backward homogeneous wave equation:
\begin{equation}\label{Wave3}
\begin{cases}
	\varphi_{tt} - \nabla\cdot \big(a(x)\nabla \varphi\big) = 0 & \text{in } (0,T)\times \Omega, \\[2mm]
	\varphi = 0 & \text{on } (0,T)\times \partial\Omega, \\[1mm]
	\varphi(T,\cdot) = \varphi_T^0,\quad \varphi_t(T,\cdot) = \varphi_T^1 & \text{in } \Omega .
\end{cases}
\end{equation}
	
Let $(\varphi,\varphi_t)\in C\!\left([0,T]; L^2(\Omega)\times H^{-1}(\Omega)\right)$ denote the unique weak solution of \eqref{Wave3}.
\medskip
\noindent
\begin{lemma}\label{Lemma:Wave}
A control $f \in L^2\big((0,T)\times \omega\big)$ steers the solution of system \eqref{Wave1} from the initial state $(u^0,u^1)\in H_0^1(\Omega)\times L^2(\Omega)$ to the target state $(u_T^0,u_T^1)\in H_0^1(\Omega)\times L^2(\Omega)$ at time $T$ if and only if
\begin{equation}\label{Eq:257}
\begin{aligned}
	0 &= -\int_0^T \int_{\omega} \varphi\, f \,dx\,dt - \int_{\Omega} \varphi_T^0 u_T^1\,dx + \langle \varphi_T^1,u_T^0\rangle_{-1,1} + \int_{\Omega} \varphi(0)u^1\,dx - \langle \varphi_t(0),u^0\rangle_{-1,1},
\end{aligned}
\end{equation}
for all $(\varphi_T^0,\varphi_T^1)\in L^2(\Omega)\times H^{-1}(\Omega)$, where $\varphi$ denotes the solution of the adjoint system \eqref{Wave3}.
\end{lemma}
The proof, a standard integration by parts and density argument, is given in Supplementary Section~\sref{app:proofs-wave}.

Suppose that an exact control $f\in L^2((0,T)\times\omega)$ exists. This existence assumption is used in the minimum-norm characterization below; the continuous observability hypothesis \eqref{eq:two-sided-observation} subsequently imposed supplies exact controllability through the standard HUM argument~\cite{lions1988exact}. By Lemma~\ref{Lemma:Wave}, existence of an exact control is equivalent to solvability of the moment identity introduced below. Set \[Q=(0,T)\times\Omega,\qquad Q_\omega=(0,T)\times\omega,\] and define $\mathfrak a[f,\varphi]=(f,\varphi)_{L^2(Q_\omega)}$ and
\[\begin{aligned}
	L(\varphi)={}&-\int_\Omega \varphi(T)u_T^1\,dx+\langle \varphi_t(T),u_T^0\rangle_{-1,1}+\int_\Omega \varphi(0)u^1\,dx-\langle \varphi_t(0),u^0\rangle_{-1,1}.
\end{aligned}\]
Let $\mathcal V=(H^2(\Omega)\cap H_0^1(\Omega))\times H_0^1(\Omega)$, and let $H$ be the space of solutions of \eqref{Wave3} whose terminal data belong to $\mathcal V$. Then $\mathfrak a[f,\varphi]=L(\varphi)\quad(\varphi\in H)$. Define the closed observation space \[ \mathscr H_\omega=\overline{\{\varphi|_{Q_\omega}:\varphi\in H\}}^{\,L^2(Q_\omega)}.\]

\begin{proposition}\label{prop:minimum-norm-exact-control}
The orthogonal projection \[f^*:=P_{\mathscr H_\omega}f\] is independent of the exact control $f$ used in its definition. Moreover, $f^*$ is the unique exact control of minimum $L^2(Q_\omega)$ norm and \[\mathfrak a[f^*,\varphi]=L(\varphi)\] for every weak adjoint solution of \eqref{Wave3} with terminal data in $L^2(\Omega)\times H^{-1}(\Omega)$.
\end{proposition}

This is the standard HUM characterization; the proof is given in Supplementary Section~\sref{app:proofs-wave}.

Introduce the strong trial space \[X_0=C([0,T];H^2(\Omega)\cap H_0^1(\Omega)) \cap C^1([0,T];H_0^1(\Omega)) \cap C^2([0,T];L^2(\Omega)),\] with the homogeneous lateral boundary condition understood, and set $\mathcal P v=v_{tt}-\nabla\cdot(a(x)\nabla v)$. For $\varphi,\psi\in X_0$, define
\[\begin{aligned}
	b_{LS}[\varphi,\psi]={}&(\mathcal P\varphi,\mathcal P\psi)_{L^2(Q)}+(\varphi(T),\psi(T))_{H^2(\Omega)}\\ &+(\varphi(T),\psi(T))_{H^1(\Omega)}+(\varphi_t(T),\psi_t(T))_{H^1(\Omega)},
\end{aligned}\]
and $\|\varphi\|_{LS}^2=b_{LS}[\varphi,\varphi]$. Both terminal displacement terms are kept deliberately: since $\|z\|_{H^2}^2+\|z\|_{H^1}^2=2\|z\|_{H^1}^2+\|D^2z\|_{L^2}^2$, the norm coincides in one space dimension with $a=1$ with the graph norm \eqref{eq:numerical-graph-norm} used in the computations, and it is equivalent to that norm in general by elliptic regularity. Uniqueness for the terminal-value Cauchy problem for the wave equation shows that this is a norm on $X_0$. We define $X$ to be the completion of $X_0$ in this norm. Thus $(X,b_{LS})$ is a Hilbert space by construction; no completeness of the stronger space $X_0$ in the least-squares norm is asserted.

For clarity, the maps \[v\longmapsto \mathcal Pv,\qquad v\longmapsto v(T),\qquad v\longmapsto v_t(T)\] from $X_0$ into $L^2(Q)$, $H^2(\Omega)\cap H_0^1(\Omega)$, and $H_0^1(\Omega)$, respectively, extend continuously to $X$. Throughout the paper, the same notation is used for these canonical extensions. The energy estimate of Lemma~\ref{lem:ls-embedding} below gives a continuous and injective realization of every element of $X$ as an $L^2(Q)$ function; see the remark following that lemma.

Let $A=-\nabla\cdot(a\nabla)$ be the positive Dirichlet realization in $L^2(\Omega)$, and equip $H^{-1}(\Omega)$ with the equivalent norm $\|z\|_{H_A^{-1}(\Omega)}=\|A^{-1/2}z\|_{L^2(\Omega)}$. Set \[\mathcal E=L^2(\Omega)\times H_A^{-1}(\Omega),\qquad \|(z^0,z^1)\|_{\mathcal E}^2=\|z^0\|_{L^2(\Omega)}^2+\|z^1\|_{H_A^{-1}(\Omega)}^2.\] Choose a fixed $\mathcal E$-orthonormal basis $\{(w_k^0,w_k^1)\}_{k\ge1}$ whose elements belong to $\mathcal V$. Such a basis exists because $\mathcal V$ is dense in $\mathcal E$. More explicitly, if $Ae_j=\lambda_je_j$ and $\{e_j\}_{j\ge1}$ is an orthonormal basis of $L^2(\Omega)$, one may take \[(w_{2j-1}^0,w_{2j-1}^1)=(e_j,0),\qquad (w_{2j}^0,w_{2j}^1)=(0,\sqrt{\lambda_j}\,e_j).\] These pairs belong to $\mathcal V$ under the smoothness assumptions imposed above.
	    
Let $\phi_k\in H$ solve
\begin{equation}\label{Waveback3}
\begin{cases}
    \mathcal P\phi_k=0 & \text{in }Q,\\
    \phi_k=0 & \text{on }(0,T)\times\partial\Omega,\\
    \phi_k(T)=w_k^0,\quad (\phi_k)_t(T)=w_k^1 & \text{in }\Omega.
\end{cases}
\end{equation}
The corresponding least-squares functional is 
\[\begin{aligned}
	\mathcal J_k(v)={}&\frac12\|\mathcal Pv\|_{L^2(Q)}^2+\frac12\|v(T)-w_k^0\|_{H^2(\Omega)}^2\\
	&+\frac12\|v(T)-w_k^0\|_{H^1(\Omega)}^2+\frac12\|v_t(T)-w_k^1\|_{H^1(\Omega)}^2.
\end{aligned}\]
Since $\phi_k$ solves \eqref{Waveback3},
\begin{equation}\label{eq:objective-error-identity}
    2\mathcal J_k(v)=\|v-\phi_k\|_{LS}^2.
\end{equation}
First, we show a connection between the LS-norm and the $L^2$-norm on $(0,T)\times\Omega$.
\begin{lemma}\label{lem:ls-embedding}
    There exists a constant $C>0$, depending only on $\Omega$, $T$, $a_0$ and $a_1$, such that, for every $\psi\in X$, \[\|\psi\|_{L^2(Q)}\leq C\|\psi\|_{LS}.\] Consequently, \[\|\psi\|_{L^2(Q_\omega)}\leq C\|\psi\|_{LS}.\]
\end{lemma}

The proof is an energy estimate for $\mathcal P\psi$; it is given in Supplementary Section~\sref{app:proofs-wave}.

\begin{remark}[Realization of $X$ in $L^2(Q)$]\label{rem:X-injective}
Lemma~\ref{lem:ls-embedding} gives a continuous realization of every element of $X$ as an $L^2(Q)$ function and supplies its initial energy traces. This realization is injective (Supplementary Section~\sref{app:X-injective}). Thus all least-squares terms, terminal traces, and restrictions to $Q_\omega$ used below are well defined on $X$.
\end{remark}

\begin{lemma}\label{lem:L-continuity}
    The linear functional \(L\) is continuous with respect to the least-squares norm. More precisely, there exists a constant \(C>0\) such that \[|L(v)|\le C\|v\|_{LS},\qquad \forall v\in X.\]
\end{lemma}

The proof applies the same energy estimate to the initial traces; it is given in Supplementary Section~\sref{app:proofs-wave}.
    
In order to use the Galerkin neural-network construction of Ainsworth and Dong \cite{AinsworthDong2021}, we need the following lemma.

\begin{lemma}\label{Density}
    Assume that \(\Omega\subset\mathbb R^d\) is a bounded domain with \(C^\infty\) boundary and that \(a\in W^{1,\infty}(\Omega)\). Let \(U\) be an open neighborhood of \(\overline\Omega\), and assume that \(\rho\in C^\infty(U)\) is a boundary-defining function for \(\Omega\), namely \[\Omega=\{x\in U:\rho(x)>0\},\qquad \rho=0\ \text{on }\partial\Omega,\qquad \nabla\rho\neq 0\ \text{on }\partial\Omega.\] Assume that \(\sigma\in C^2(\mathbb R)\) has the \(C^2\)-universal approximation property: for every compact set \(K\subset\mathbb R^{d+1}\) and every function \(q\) that is \(C^2\) on a neighborhood of \(K\), finite sums of the form \[z\longmapsto\sum_{j=1}^n c_j\sigma(\xi_j\cdot z+b_j)\] are dense in \(C^2(K)\). This assumption is satisfied, for example, by the smooth bounded nonconstant activations \(\sigma(s)=\tanh s\) and \(\sigma(s)=\sin s\) \cite[Theorem~3]{hornik1991approximation}. For \(n\in\mathbb N\), define \[\mathcal N_n^\rho=\left\{\rho(x)\sum_{j=1}^n c_j\sigma(\alpha_jt+W_j\cdot x+b_j):c_j,\alpha_j,b_j\in\mathbb R,\ W_j\in\mathbb R^d\right\}.\]
    Set \[\mathcal N^\rho=\bigcup_{n=1}^\infty \mathcal N_n^\rho.\] Then \[\overline{\mathcal N^\rho}^{\|\cdot\|_{LS}}=X.\]
\end{lemma}

\begin{proof}
A boundary-preserving smoothing argument and division by the smooth boundary factor reduce the assertion to $C^2$ neural approximation. The complete density argument is given in Supplementary Section~\sref{app:density-proof}.
\end{proof}

\begin{remark}
Lemma~\ref{Density} provides the approximation property needed in the adaptive Galerkin neural-network construction below. At each enrichment step, the neural class is used to approximate the normalized Galerkin error through a residual-maximization problem. The resulting finite-dimensional Galerkin solve is linear and coercive, whereas the feature-generation step involves nonlinear optimization over the network parameters. The density result is an existence statement only: it does not imply that a numerical optimizer finds a global maximizer over the nonlinear class $\mathcal N^\rho$. For the square of Example~2, which does not satisfy the smooth-boundary hypothesis, a square-specific analogue holds (Supplementary Section~\sref{app:square-density}); it concerns density only, not observability.
\end{remark}

\section{{Stable adaptive neural Galerkin--HUM approximation for the wave equation}}\label{sec:method}

Throughout this section, $C>0$ denotes a generic constant whose value may change from one occurrence to the next. Unless stated otherwise, such constants depend only on the fixed problem data and are independent of the discretization parameters $n$, $m$, and $h$. Constants carrying a subscript, such as $C_X$, $C_L$, or $C_{\mathrm{rec}}$, retain the meaning assigned to them. We write $A\lesssim B$ when $A\leq CB$ for a generic constant of this type.
        
\subsection{Adaptive neural Galerkin--HUM approximation of the adjoint solutions}

We apply the adaptive Galerkin neural-network construction to each adjoint variational problem separately. For every $k\geq1$, define the continuous linear functional
\begin{equation*}
    L_k(v):=(w_k^0,v(T))_{H^2(\Omega)}+(w_k^0,v(T))_{H^1(\Omega)}+(w_k^1,v_t(T))_{H^1(\Omega)},\qquad v\in X.
\end{equation*}
Since $\mathcal P\phi_k=0$, $\phi_k(T)=w_k^0$, and $(\phi_k)_t(T)=w_k^1$, the exact adjoint solution $\phi_k\in X$ is equivalently characterized by
\begin{equation}\label{eq:adjoint-ls-variational-problem}
    b_{LS}[\phi_k,v]=L_k(v),\qquad v\in X.
\end{equation}
The bilinear form $b_{LS}$ is symmetric and positive definite on $X$, and the induced energy norm is precisely the least-squares norm $\|v\|_{LS}:=b_{LS}[v,v]^{1/2}$.
For every $k\geq1$, we construct a sequence of finite-dimensional spaces $Y_{k,0}\subset Y_{k,1}\subset Y_{k,2}\subset\cdots\subset X$
adapted to the adjoint solution $\phi_k$. We initialize the procedure by $Y_{k,0}:=\{0\},\quad \phi_k^0:=0$.
Suppose that $Y_{k,m-1}$ and the corresponding Galerkin approximation $\phi_k^{m-1}\in Y_{k,m-1}$ have already been constructed. Define the residual functional
\begin{equation*}
    \langle r_k(\phi_k^{m-1}),v\rangle:=L_k(v)-b_{LS}[\phi_k^{m-1},v],\qquad v\in X.
\end{equation*}
Using \eqref{eq:adjoint-ls-variational-problem}, we obtain
\begin{equation}\label{eq:residual-error-representation}
    \langle r_k(\phi_k^{m-1}),v\rangle=b_{LS}[\phi_k-\phi_k^{m-1},v].
\end{equation}

Let $B_{LS}:=\{v\in X:\|v\|_{LS}\leq1\}$ denote the closed unit ball of $(X,b_{LS})$. If $\phi_k^{m-1}\neq\phi_k$, then \eqref{eq:residual-error-representation} and the Cauchy--Schwarz inequality show that $\max_{v\in B_{LS}}\langle r_k(\phi_k^{m-1}),v\rangle=\|\phi_k-\phi_k^{m-1}\|_{LS}$, attained at the normalized current error
\begin{equation}\label{eq:normalized-current-error}
    \widehat\psi_{k,m}:=\frac{\phi_k-\phi_k^{m-1}}{\|\phi_k-\phi_k^{m-1}\|_{LS}}.
\end{equation}
This observation is the basis of the adaptive enrichment procedure.

The exact maximizer \eqref{eq:normalized-current-error} cannot be evaluated because it depends on the unknown solution $\phi_k$. We therefore approximate it by a shallow neural network. At the $(k,m)$-th enrichment step, let $p_{k,m}\geq1$ be a prescribed width and write
\begin{equation}\label{eq:hidden-neural-features}
    \zeta_j^{\lambda}(t,x):=\rho(x)\sigma(\alpha_jt+W_j\cdot x+b_j),\qquad 1\leq j\leq p_{k,m},
\end{equation}
where $\lambda=\{\alpha_j,W_j,b_j\}_{j=1}^{p_{k,m}}\in\mathbb R^{p_{k,m}(d+2)}$ collects the hidden parameters. The factor $\rho$ imposes the homogeneous Dirichlet boundary condition exactly. The width is allowed to depend on both the terminal mode $k$ and the enrichment index $m$, which is useful because higher-frequency modes may require a richer local feature space.

For every nonzero $v\in X$, define the normalized residual indicator
\begin{equation*}
    \mathcal{I}_{k,m}(v):=\frac{\langle r_k(\phi_k^{m-1}),v\rangle}{\|v\|_{LS}}=\frac{L_k(v)-b_{LS}[\phi_k^{m-1},v]}{\|v\|_{LS}}.
\end{equation*}
Suppose that the partially linear training procedure described below returns a nonzero function $v_{k,m}^{NN}$. We change its sign, if necessary, and set
\begin{equation*}
    \psi_{k,m}^{NN}:=\frac{v_{k,m}^{NN}}{\|v_{k,m}^{NN}\|_{LS}},\qquad\eta_{k,m} :=\langle r_k(\phi_k^{m-1}),\psi_{k,m}^{NN}\rangle\geq0.
\end{equation*}
Thus $\eta_{k,m}=L_k(\psi_{k,m}^{NN})-b_{LS}[\phi_k^{m-1},\psi_{k,m}^{NN}]$.

The quantity $\eta_{k,m}$ measures the component of the current residual captured by the trained direction. A small value of $\eta_{k,m}$ alone is not a certified stopping criterion because it may also result from unsuccessful training. At the analytical level, we stop mode $k$ when
\begin{equation}\label{eq:adjoint-stopping-test}
    \sqrt{2\mathcal J_k(\phi_k^{m-1})}\leq\mathrm{tol}_k.
\end{equation}
By \eqref{eq:objective-error-identity}, this is precisely the condition $\|\phi_k-\phi_k^{m-1}\|_{LS}\leq\mathrm{tol}_k$. If the stopping test is not satisfied and the direction passes the acceptance tests stated below, we define $Y_{k,m}:=Y_{k,m-1}+\operatorname{span}\{\psi_{k,m}^{NN}\}$
and recompute the global Galerkin projection
\begin{equation}\label{eq:adjoint-projection}
    \phi_k^m\in Y_{k,m}:\qquad b_{LS}[\phi_k^m,v]=L_k(v),\qquad v\in Y_{k,m}.
\end{equation}

The relevant approximation-quality quantity is the fraction of the current error detected by the accepted direction:
\begin{equation}\label{eq:residual-correlation}
    \varrho_{k,m}:=\frac{\eta_{k,m}}{\sqrt{2\mathcal J_k(\phi_k^{m-1})}}=b_{LS}[\widehat\psi_{k,m},\psi_{k,m}^{NN}]\in[0,1].
\end{equation}
Unlike the unknown function $\widehat\psi_{k,m}$, the first expression in \eqref{eq:residual-correlation} is computable from the variational functional when the integrals are evaluated exactly.

\begin{proposition}\label{prop:neural-residual-error-reduction}
For every accepted enrichment step, problem \eqref{eq:adjoint-projection} has a unique solution $\phi_k^m$, the $b_{LS}$-orthogonal projection of $\phi_k$ onto $Y_{k,m}$:
\begin{align}
    b_{LS}[\phi_k-\phi_k^m,v]&=0,\qquad v\in Y_{k,m},\label{eq:adjoint-galerkin-orthogonality}\\
    \|\phi_k-\phi_k^m\|_{LS}&=\inf_{v\in Y_{k,m}}\|\phi_k-v\|_{LS}.\label{eq:adjoint-best-approximation}
\end{align}
If $\eta_{k,m}>0$, then $\psi_{k,m}^{NN}\notin Y_{k,m-1}$, so that $Y_{k,m}=Y_{k,m-1}\mathbin{\dotplus}\operatorname{span}\{\psi_{k,m}^{NN}\}$ (algebraic direct sum). If $\phi_k^{m-1}\neq\phi_k$, then
\begin{equation}\label{eq:adjoint-error-contraction}
    \|\phi_k-\phi_k^m\|_{LS}\leq\sqrt{1-\varrho_{k,m}^2}\,\|\phi_k-\phi_k^{m-1}\|_{LS}\leq\|\phi_k-\phi_k^{m-1}\|_{LS},
\end{equation}
and
\begin{equation}\label{eq:correlation-distance-identity}
    \|\widehat\psi_{k,m}-\psi_{k,m}^{NN}\|_{LS}^2=2-2\varrho_{k,m}.
\end{equation}
In particular, the condition $\|\widehat\psi_{k,m}-\psi_{k,m}^{NN}\|_{LS}<1$ is equivalent to $\varrho_{k,m}>1/2$.
\end{proposition}
The proof combines Galerkin orthogonality with the Pythagorean identity in $(X,b_{LS})$; it is given in Supplementary Section~\sref{app:galerkin-algebra}.

Following the partially linear Galerkin strategy of Ainsworth and Dong \cite{AinsworthDong2021}, the output coefficients for fixed hidden parameters $\lambda$ are determined by the local Galerkin problem on the feature space $V_{k,m}(\lambda):=\operatorname{span}\{\zeta_1^\lambda,\ldots,\zeta_{p_{k,m}}^\lambda\}\subset X$,
\begin{equation}\label{eq:local-residual-galerkin-problem}
    v_{k,m}(\lambda)\in V_{k,m}(\lambda): \qquad b_{LS}[v_{k,m}(\lambda),w] = \langle r_k(\phi_k^{m-1}),w\rangle, \qquad w\in V_{k,m}(\lambda).
\end{equation}
With $v_{k,m}(\lambda)=\sum_jc_j(\lambda)\zeta_j^\lambda$, it reads
\begin{equation}\label{eq:local-output-system}
    A_{k,m}(\lambda)c(\lambda) = g_{k,m}(\lambda),\qquad \bigl(A_{k,m}(\lambda)\bigr)_{ij} := b_{LS}[\zeta_j^\lambda,\zeta_i^\lambda], \quad \bigl(g_{k,m}(\lambda)\bigr)_i := \langle r_k(\phi_k^{m-1}),\zeta_i^\lambda\rangle,
\end{equation}
which is consistent even for linearly dependent features. Hidden-parameter training seeks to increase the captured normalized residual $\mathscr Q_{k,m}(\lambda):=\|v_{k,m}(\lambda)\|_{LS}$; no assertion is made that it finds a global maximizer. The returned network $v_{k,m}^{NN}$ may contain $p_{k,m}$ neurons but contributes one accepted direction $\psi_{k,m}^{NN}=v_{k,m}^{NN}/\|v_{k,m}^{NN}\|_{LS}$ to $Y_{k,m}$, so that $\dim Y_{k,m}=m$ after $m$ accepted steps. The consistency argument, the pseudoinverse and its numerical realization, and the joint optimization phase are described in Supplementary Section~\sref{app:galerkin-algebra}.

\begin{remark}
In computation, the quadrature used for training is separated from an independent validation quadrature. For \(q\in\{\mathrm{tr},\mathrm{val}\}\) and a nonzero \(v\in X\), provided that \(\widehat{\mathcal J}_k^q(\phi_k^{m-1})>0\) and \(\|v\|_{\mathrm{LS},q}>0\), define the computable correlation 
\begin{equation}\label{eq:discrete-residual-correlation}
    \widehat{\varrho}_{k,m}^{q}(v):=\frac{\widehat{L}_k^{q}(v) -\widehat{b}_{\mathrm{LS}}^{q} [\phi_k^{m-1},v]}{\sqrt{2\widehat{\mathcal{J}}_k^{q}(\phi_k^{m-1})}\, \|v\|_{\mathrm{LS},q}}.
\end{equation}
The sign of \(v_{k,m}^{\mathrm{NN}}\) is fixed by the training residual. A direction is accepted only if, for prescribed constants \(\varrho_{\mathrm{acc}}>0\) and \(\tau_{\mathrm{rank}}>0\), 
\begin{equation}\label{eq:discrete-direction-gates}
\begin{aligned}
    \widehat{\varrho}_{k,m}^{\mathrm{tr}}(v_{k,m}^{\mathrm{NN}}) &\geq \varrho_{\mathrm{acc}}, &\qquad \widehat{\varrho}_{k,m}^{\mathrm{val}}(v_{k,m}^{\mathrm{NN}}) &\geq \varrho_{\mathrm{acc}}, &\qquad \bigl\| (I-\Pi_{k,m-1}^{\mathrm{tr}})\psi_{k,m}^{\mathrm{NN},tr} \bigr\|_{\mathrm{LS},\mathrm{tr}} \geq \tau_{\mathrm{rank}}
\end{aligned}
\end{equation}
with \[\psi_{k,m}^{\mathrm{NN},tr}:=\dfrac{v_{k,m}^{NN}}{\bigl\|v_{k,m}^{NN}\bigr\|_{LS,tr}}.\] Here \(\Pi_{k,m-1}^{\mathrm{tr}}\) denotes the projection onto \(Y_{k,m-1}\) with respect to the training approximation of \(b_{\mathrm{LS}}\), assuming that this bilinear form is positive definite on \(Y_{k,m-1}\). A rejected direction triggers a larger width, additional random restarts, or a larger training budget.

The conditions above are practical tests based on quadrature. By themselves, they do not imply the exact residual-capture bound \(\varrho_{k,m}\geq\varrho_\star\) in \eqref{eq:AD-uniform-quality}. Hence, the convergence estimates that use \eqref{eq:AD-uniform-quality} remain conditional on that exact bound. Establishing an implication from the discrete tests to \eqref{eq:AD-uniform-quality} would additionally require bounds on the quadrature and algebraic errors in the residual, the least-squares functional, and the norm of the candidate direction. Likewise, implementing the exact stopping criterion \eqref{eq:adjoint-stopping-test} by quadrature requires a reliable upper bound for \(\mathcal{J}_k(\phi_k^{m-1})\). Merely requiring \(\eta_{k,m}>0\) is insufficient for a uniform convergence rate.
\end{remark}

The global coefficients of $\phi_k^m=\sum_{j=1}^ma_{k,j}^{(m)}\psi_{k,j}^{NN}$ solve
\begin{equation}\label{eq:adjoint-linear-system}
    K_{k,m}a_k^{(m)}=g_{k,m},\qquad (K_{k,m})_{ij}:=b_{LS}[\psi_{k,j}^{NN},\psi_{k,i}^{NN}],\quad (g_{k,m})_i:=L_k(\psi_{k,i}^{NN}),
\end{equation}
whose matrix is symmetric positive definite because the accepted directions are linearly independent (Supplementary Section~\sref{app:galerkin-algebra}).

\medskip
\paragraph{From validated numerical residuals to a continuous accepted direction.}
The distinction between a sampled gate and a continuous gate can be made
quantitative. The following sufficient test concerns the direction actually
returned by training; it makes no assumption about global optimizer convergence.
\begin{proposition}[A validated acceptance test]\label{prop:validated-acceptance}
Let $Y\subset X$ be a finite-dimensional accepted space, $y=P_Y^{LS}\phi_k$,
$e=\phi_k-y\ne0$, and let a computed $\widetilde y$ satisfy
$\|\widetilde y-y\|_{LS}\leq\chi$. For a candidate $v\in X\setminus\{0\}$,
suppose validated numbers satisfy
\[
 \|v\|_{LS}\leq V_+,\quad
 \|\phi_k-\widetilde y\|_{LS}\leq E_+,\quad
 \left|\widehat N-
 [L_k(v)-b_{LS}(\widetilde y,v)]\right|\leq\delta_N.
\]
Set
\[
 \underline\varrho=
 \frac{(|\widehat N|-\delta_N-\chi V_+)_+}{(E_++\chi)V_+}.
\]
Then $|b_{LS}(e,v)|/(\|e\|_{LS}\|v\|_{LS})\geq\underline\varrho$.
In particular, $0<\underline\varrho\leq1$ certifies a continuous residual-capture
factor. For the exact projection $y_+$ in $Y+\operatorname{span}\{v\}$ and a
computed update with $\|\widetilde y_+-y_+\|_{LS}\leq\chi_+$,
\[
 \|\phi_k-\widetilde y_+\|_{LS}
 \leq\sqrt{1-\underline\varrho^2}\,\|\phi_k-y\|_{LS}+\chi_+.
\]
\end{proposition}
\begin{proof}
The variational identity for $\phi_k$ gives
\[
 L_k(v)-b_{LS}(\widetilde y,v)
 =b_{LS}(e,v)+b_{LS}(y-\widetilde y,v).
\]
Cauchy--Schwarz bounds the second term by $\chi V_+$, whereas
$\|e\|_{LS}\leq E_++\chi$. These estimates prove the lower bound, with a
change of sign of $v$ if necessary. Projection onto the enlarged space gives
the contraction of Proposition~\ref{prop:neural-residual-error-reduction};
the triangle inequality accounts for $\chi_+$.
\end{proof}
The norm and scalar enclosures must refer to the consistent continuous
least-squares problem. A Ritz terminal-target defect or unbounded quadrature
error cannot be omitted from them. A validated Gram solve on the accepted
space supplies $\chi$, and a validated Gram Schur complement supplies the
continuous rank test when desired. The verification study below evaluates
these continuous quantities for stored binary64 directions. It verifies
finitely many returned directions, rather than the uniform hypothesis for
all future modes, budgets or random seeds.

\subsection{Uniform conditioning of the observation Gram matrices}

We first establish uniform conditioning for the observation Gram matrices generated by the exact adjoint modes. We then show that this property is preserved when each exact mode $\phi_k$ is replaced by its adaptive Galerkin neural-network approximation $\phi_k^m\in Y_{k,m}$. The spaces $Y_{k,m}$ may depend on $k$; the argument below uses only uniform control of the individual approximation errors.

For terminal data $z=(z^0,z^1)\in\mathcal E$, let $\phi_z$ denote the corresponding solution of the homogeneous adjoint equation, and define the observation operator $S:\mathcal E\longrightarrow L^2(Q_\omega),\quad Sz:=\phi_z|_{Q_\omega}$. In particular, for the $\mathcal E$-orthonormal terminal basis $\{w_k\}_{k\geq1}$ introduced above, $\phi_k=\phi_{w_k},\quad Sw_k=\phi_k|_{Q_\omega}$.

\medskip
\noindent\textbf{Observation hypothesis.}
We assume that there exist constants $C_{\mathrm{obs}},C_{\mathrm{dir}}>0$ such that
\begin{equation}\label{eq:two-sided-observation}
    C_{\mathrm{obs}}^{-1}\|z\|_{\mathcal E}^2\leq\|Sz\|_{L^2(Q_\omega)}^2\leq C_{\mathrm{dir}}\|z\|_{\mathcal E}^2, \qquad z\in\mathcal E.
\end{equation}
The upper bound follows directly from conservation of the dual wave energy: \[\|\phi_z(t)\|_{L^2(\Omega)}^2+\|(\phi_z)_t(t)\|_{H_A^{-1}(\Omega)}^2=\|z\|_{\mathcal E}^2.\] Consequently, $\|Sz\|_{L^2(Q_\omega)}^2\leq T\|z\|_{\mathcal E}^2$, so one may take $C_{\mathrm{dir}}=T$. The lower bound is the observability inequality. For the wave equation, it holds under an appropriate geometric control condition on $(\omega,T)$; see, for example, \cite{blr1992sharp,zuazua2006controllability}.

For convenience, set $\alpha:=C_{\mathrm{obs}}^{-1/2}, \quad \beta:=C_{\mathrm{dir}}^{1/2}$.
Then \eqref{eq:two-sided-observation} can be written as
\begin{equation}\label{eq:observation-alpha-beta}
    \alpha\|z\|_{\mathcal E} \leq \|Sz\|_{L^2(Q_\omega)} \leq \beta\|z\|_{\mathcal E}, \qquad z\in\mathcal E.
\end{equation}

For each $n\geq1$, define \[\mathcal E_n:=\operatorname{span}\{w_1,\ldots,w_n\},\qquad \mathscr H_n:=\operatorname{span}\{\phi_1,\ldots,\phi_n\},\] and $\mathscr H_{n,\omega}:=\{\phi|_{Q_\omega}:\phi\in\mathscr H_n\}$. We introduce the operator
\begin{equation}\label{eq:exact-synthesis-operator}
    \Phi_n:\mathbb{R}^n\longrightarrow L^2(Q_\omega), \qquad \Phi_nc := \sum_{k=1}^n c_k\phi_k|_{Q_\omega}.
\end{equation}
The associated observation Gram matrix is
\begin{equation*}
    G_n:=\Phi_n^*\Phi_n = \left((\phi_j,\phi_i)_{L^2(Q_\omega)}\right)_{i,j=1}^n.
\end{equation*}
Here $\Phi_n^*$ denotes the adjoint operator of $\Phi_n$. We use the Loewner ordering for symmetric matrices: \[A\preceq B \quad\Longleftrightarrow\quad c^\top Ac\leq c^\top Bc\quad\text{for every }c\in\mathbb{R}^n.\] Throughout this subsection, $\|A\|_2$ denotes the matrix norm induced by the Euclidean vector norm, and $\kappa_2(A):=\|A\|_2\|A^{-1}\|_2$ for an invertible matrix $A$.

\begin{proposition}\label{prop:exact-gram-stability}
For every $n\geq 1$ and every $c\in\mathbb{R}^n$,
\begin{equation}\label{eq:exact-synthesis-stability}
    \alpha\|c\|_2 \leq\|\Phi_nc\|_{L^2(Q_\omega)} \leq \beta\|c\|_2.
\end{equation}
Consequently,
\begin{equation}\label{eq:exact-gram-spectral-bounds}
    \alpha^2I_n \preceq G_n \preceq \beta^2I_n.
\end{equation}
In particular, $G_n$ is symmetric positive definite and
\begin{equation*}
    \kappa_2(G_n) \leq \frac{\beta^2}{\alpha^2} = C_{\mathrm{obs}}C_{\mathrm{dir}}.
\end{equation*}
The bound is independent of $n$.
\end{proposition}

\begin{proof}
For $z_c=\sum_kc_kw_k$, terminal orthonormality gives
$\|z_c\|_{\mathcal E}=\|c\|_2$, and linearity gives $Sz_c=\Phi_nc$.
The two-sided observation estimate proves
\eqref{eq:exact-synthesis-stability}. Since
$c^\top G_nc=\|\Phi_nc\|_{L^2(Q_\omega)}^2$, its square gives
\eqref{eq:exact-gram-spectral-bounds} and the condition-number bound.
\end{proof}

We now examine the observation Gram matrix obtained from the approximations $\phi_k^m\in Y_{k,m}$ constructed in the previous subsection. We define
\begin{equation}\label{eq:adjoint-uniform-error}
    \varepsilon_{n,m}=\max_{1\leq k\leq n}\|\phi_k-\phi_k^m\|_{LS}.
\end{equation}
Here $m$ counts accepted enrichment steps. If the construction for mode $k$ terminates after $m_k$ steps, we adopt the convention $Y_{k,m}:=Y_{k,m_k}, \quad\phi_k^m:=\phi_k^{m_k},\quad m\geq m_k$, so that a common index $m$ can be used for all modes.

For the quantitative estimates below, set $\mathfrak B_n:=\max_{1\leq k\leq n}\|\phi_k\|_{LS}$.

We shall use the following uniform residual-capture condition: there exists a constant $\varrho_\star\in(0,1)$, independent of $k$, $j$, and $n$, such that
\begin{equation}\label{eq:AD-uniform-quality}
    \varrho_{k,j}\geq\varrho_\star,\qquad 1\leq k\leq n,
\end{equation}
at every accepted enrichment step. Define
\begin{equation*}
    \mu_\star:=\sqrt{1-\varrho_\star^2}\in(0,1),\qquad\mathrm{Tol}_n:=\max_{1\leq k\leq n}\mathrm{tol}_k.
\end{equation*}

\begin{corollary}\label{cor:uniform-adaptive-adjoint-rate}
Suppose that \eqref{eq:AD-uniform-quality} holds. If every mode $1\leq k\leq n$ undergoes $m$ accepted enrichment steps, then
\begin{equation*}
    \varepsilon_{n,m}\leq\mathfrak B_n\mu_\star^m.
\end{equation*}
More generally, if each mode either undergoes $m$ accepted steps or terminates earlier only after satisfying \eqref{eq:adjoint-stopping-test}, then
\begin{equation*}
    \varepsilon_{n,m}\leq\max\{\mathfrak B_n\mu_\star^m,\mathrm{Tol}_n\}\leq\mathfrak B_n\mu_\star^m+\mathrm{Tol}_n.
\end{equation*}
\end{corollary}

\begin{proof}
At every accepted step, Proposition~\ref{prop:neural-residual-error-reduction} and \eqref{eq:AD-uniform-quality} give $\|\phi_k-\phi_k^j\|_{LS}\leq\mu_\star\|\phi_k-\phi_k^{j-1}\|_{LS}$. Since $\phi_k^0=0$, iteration gives $\|\phi_k-\phi_k^m\|_{LS}\leq \mu_\star^m\|\phi_k\|_{LS}$. If mode $k$ terminates early, its error is at most $\mathrm{tol}_k$ by \eqref{eq:adjoint-stopping-test}. Taking the maximum over $1\leq k\leq n$ proves both estimates.
\end{proof}

\begin{remark}
The early-termination part of Corollary~\ref{cor:uniform-adaptive-adjoint-rate} covers only modes that pass the certified stopping test \eqref{eq:adjoint-stopping-test}. Termination caused by stagnation, a restart limit, or a step budget smaller than the common index $m$ is not covered by this corollary. For such a mode, the actual error $\sqrt{2\mathcal J_k(\widetilde\phi_k^{m,h})}$ enters the fully discrete analysis through \eqref{eq:adjoint-error-bound}.
\end{remark}

Let $C_X>0$ be an embedding constant from Lemma~\ref{lem:ls-embedding}, so that
\begin{equation}\label{eq:ls-to-observation}
    \|v\|_{L^2(Q_\omega)}\leq C_X\|v\|_{LS}, \qquad v\in X.
\end{equation}

A conservative explicit choice can be obtained from the proof of Lemma~\ref{lem:ls-embedding}. Let $C_p$ satisfy $\|z\|_{L^2(\Omega)}\leq C_p\|\nabla z\|_{L^2(\Omega)}$ for $z\in H_0^1(\Omega)$. With the standard Sobolev norms, one may choose
\begin{equation}\label{eq:explicit-ls-embedding-constant}
 C_X=\frac{C_p}{\sqrt{a_0}}\left(T\max\{1,a_1\}+\frac{T^2}{2}\right)^{1/2}.
\end{equation}
The derivation is given in Supplementary Section~\sref{app:proofs-perturbation}. A bound for $C_p$, for example from a containing domain, thus gives a conservative computable embedding bound. It does not determine the observability constant or certify the numerical adjoint and assembly errors.

Set
\begin{equation}\label{eq:dnm-definition}
    d_{n,m} := C_X\sqrt{n}\,\varepsilon_{n,m}.
\end{equation}

Define the operator
\begin{equation}\label{eq:approximate-synthesis-operator}
    \Phi_n^m:\mathbb{R}^n\longrightarrow L^2(Q_\omega), \qquad \Phi_n^mc := \sum_{k=1}^n c_k\phi_k^m|_{Q_\omega},
\end{equation}
and the corresponding Gram matrix
\begin{equation*}
    G_n^m := (\Phi_n^m)^*\Phi_n^m = \left((\phi_j^m,\phi_i^m)_{L^2(Q_\omega)}\right)_{i,j=1}^n.
\end{equation*}

\begin{proposition}\label{prop:approximate-gram-stability}
The operators $\Phi_n$ and $\Phi_n^m$ satisfy
\begin{equation}\label{eq:synthesis-operator-perturbation}
    \|\Phi_n^m-\Phi_n\|_{\ell^2\to L^2(Q_\omega)}\leq d_{n,m}.
\end{equation}
Consequently, for every $c\in\mathbb{R}^n$,
\begin{equation}\label{eq:approximate-synthesis-stability}
    (\alpha-d_{n,m})\|c\|_2 \leq \|\Phi_n^mc\|_{L^2(Q_\omega)} \leq (\beta+d_{n,m})\|c\|_2.
\end{equation}
If
\begin{equation*}
    d_{n,m}<\alpha,
\end{equation*}
then $G_n^m$ is symmetric positive definite and satisfies
\begin{equation}\label{eq:approximate-gram-spectral-bounds}
    (\alpha-d_{n,m})^2I_n \preceq G_n^m \preceq (\beta+d_{n,m})^2I_n.
\end{equation}
In particular,
\begin{equation}\label{eq:approximate-gram-condition-number}
    \kappa_2(G_n^m) \leq \left(\frac{\beta+d_{n,m}}{\alpha-d_{n,m}}\right)^2.
\end{equation}
\end{proposition}

The proof is a triangle-inequality estimate; it is given in Supplementary Section~\sref{app:proofs-perturbation}.

The condition $d_{n,m}<\alpha$ ensures invertibility, but it does not provide a condition-number bound that is uniform in $n$ and $m$, because $d_{n,m}$ may approach $\alpha$. A uniform relative perturbation condition gives the following consequence.

\begin{corollary}\label{cor:uniform-approximate-gram-conditioning}
Let $0<\xi<1$, and suppose that
\begin{equation*}
    d_{n,m}\leq\xi\alpha.
\end{equation*}
Then
\begin{equation*}
    (1-\xi)^2\alpha^2I_n \preceq G_n^m \preceq (\beta+\xi\alpha)^2I_n,
\end{equation*}
and
\begin{equation*}
    \kappa_2(G_n^m) \leq \left(\frac{\beta+\xi\alpha}{(1-\xi)\alpha}\right)^2.
\end{equation*}
In particular, if \[d_{n,m}\leq\frac{\alpha}{2},\] then
\begin{equation*}
    \frac{\alpha^2}{4}I_n\preceq G_n^m\preceq\left(\beta+\frac{\alpha}{2}\right)^2I_n,
\end{equation*}
and
\begin{equation*}
    \|(G_n^m)^{-1}\|_2 \leq \frac{4}{\alpha^2}.
\end{equation*}
\end{corollary}

\begin{proof}
The assumption $d_{n,m}\leq\xi\alpha$ gives \[\alpha-d_{n,m} \geq (1-\xi)\alpha, \qquad \beta+d_{n,m} \leq \beta+\xi\alpha.\] The conclusions follow from Proposition~\ref{prop:approximate-gram-stability}. For $\xi=1/2$, we obtain \[\lambda_{\min}(G_n^m) \geq \frac{\alpha^2}{4},\] and hence \[\|(G_n^m)^{-1}\|_2 = \frac{1}{\lambda_{\min}(G_n^m)} \leq \frac{4}{\alpha^2}.\]
\end{proof}

\medskip
\subsection{A residual identity for the wave-control reconstruction}
\label{sec:wave-residual-identity}

The moment equations contain additional information specific to the controlled
wave equation. It is useful to retain this information rather than estimate the
perturbations of their entries separately. In this subsection, all inner products
are exact integrals; numerical assembly is addressed afterwards.

\begin{proposition}[Wave residual and moment consistency]
\label{prop:wave-moment-residual}
Let $f^*$ be the minimum-norm exact control and $u^*$ its controlled state. For
every $v\in X$,
\begin{equation}\label{eq:wave-moment-residual}
 L(v)-(f^*,v)_{L^2(Q_\omega)}=(u^*,\mathcal Pv)_{L^2(Q)}.
\end{equation}
Let $v_1,\ldots,v_n\in X$ have linearly independent restrictions to $Q_\omega$,
let $\Psi c=\sum_k c_kv_k|_{Q_\omega}$, $V=\operatorname{Ran}\Psi$, and set
$G=\Psi^*\Psi$, $b_k=L(v_k)$, $r_k=(u^*,\mathcal Pv_k)_{L^2(Q)}$.
If $Ga=b$ and $f_V=\Psi a$, then
\begin{equation}\label{eq:wave-control-pythagoras}
 \|f_V-f^*\|_{L^2(Q_\omega)}^2
 =\operatorname{dist}(f^*,V)^2+r^\top G^{-1}r.
\end{equation}
In particular,
\begin{equation}\label{eq:wave-control-residual-bound}
 \|f_V-f^*\|_{L^2(Q_\omega)}^2
 \leq\operatorname{dist}(f^*,V)^2+
 \frac{\|u^*\|_{L^2(Q)}^2}{\lambda_{\min}(G)}
 \sum_{k=1}^n\|\mathcal Pv_k\|_{L^2(Q)}^2.
\end{equation}
These identities also hold with the time-dependent coefficient of
Section~\ref{sec:time-dependent}.
\end{proposition}

\begin{proof}
For smooth data and $v\in X_0$, Green's identity and
$\mathcal Pu^*=-f^*\mathbf1_\omega$ give
\[
 - (f^*,v)_{L^2(Q_\omega)}-(u^*,\mathcal Pv)_{L^2(Q)}
 =\int_\Omega[v(u^*)_t-v_tu^*]_{0}^{T}\,dx=-L(v).
\]
Energy continuity, Lemma~\ref{lem:L-continuity}, and density extend the identity
to energy data and $v\in X$. The spatial operator is symmetric at each time,
which also proves the time-dependent assertion. Thus $b=\Psi^*f^*+r$ and
\[
 f_V=P_Vf^*+\Psi G^{-1}r.
\]
The second term belongs to $V$ and is orthogonal to $(I-P_V)f^*$.
Its squared norm is $r^\top G^{-1}\Psi^*\Psi G^{-1}r=r^\top G^{-1}r$,
which proves \eqref{eq:wave-control-pythagoras}.
Finally $\|r\|_2^2\leq\|u^*\|_{L^2(Q)}^2\sum_k\|\mathcal Pv_k\|_{L^2(Q)}^2$.
\end{proof}

The orthogonal-projection step in this proof is a standard Hilbert-space
argument. The wave identity \eqref{eq:wave-moment-residual} identifies its
consistency term with the interior PDE residual. In particular, terminal-data
approximation affects the available space $V$, whereas a nonzero interior
residual also affects the moment equations on that space. This distinction
does not require that the $v_k$ be Galerkin projections or that an optimizer
satisfy the uniform contraction hypothesis.

\begin{corollary}[A residual certificate with numerical assembly]
\label{cor:wave-residual-certificate}
Suppose validated numbers $E_k$ and $P_k$ satisfy
\[
 \|v_k-\phi_k\|_{LS}\leq E_k,\qquad
 \|\mathcal Pv_k\|_{L^2(Q)}\leq P_k.
\]
Write $E=(\sum_kE_k^2)^{1/2}$, $P=(\sum_kP_k^2)^{1/2}$,
$d=C_XE$, and $s=\alpha-d>0$.
Let $D$ bound $\operatorname{dist}(f^*,V)$, and let
$F\geq\|f^*\|_{L^2(Q_\omega)}$, $U\geq\|u^*\|_{L^2(Q)}$ be validated bounds.
Suppose $\widehat G$ is symmetric, $\widehat a$ is the computed coefficient
vector, and
\[
 \|\widehat G-G\|_2\leq\delta_G<s^2,\qquad
 \|\widehat G\widehat a-b\|_2\leq\delta_b.
\]
The second inequality includes the numerical solve residual. Then the actual
field $\widehat f=\Psi\widehat a$ satisfies
\begin{equation}\label{eq:wave-residual-certificate}
 \|\widehat f-f^*\|_{L^2(Q_\omega)}
 \leq \left(D^2+\frac{U^2P^2}{s^2}\right)^{1/2}
 +\frac{\beta+d}{s^2-\delta_G}
 \left(\delta_b+\delta_G\frac{\beta F+C_LE}{s^2}\right).
\end{equation}
One admissible choice is $D=\tau_n+dF/\alpha$. If
$f^*=\sum_{k=1}^nc_k^*\phi_k|_{Q_\omega}$ with known coefficients, the sharper
choice
\begin{equation}\label{eq:manufactured-distance-bound}
 D=C_X\sum_{k=1}^n|c_k^*|E_k
\end{equation}
is available and terminal-space truncation vanishes.
\end{corollary}

\begin{proof}
The synthesis perturbation proof, with the columnwise Cauchy--Schwarz estimate,
gives $\|\Psi-\Phi_n\|\leq C_X(\sum_k E_k^2)^{1/2}=d$.
Consequently $\lambda_{\min}(G)\geq s^2$ and $\|\Psi\|\leq\beta+d$.
Also $\|b\|_2\leq\beta F+C_LE$, so $\|a\|_2\leq(\beta F+C_LE)/s^2$.
For $e=\widehat G\widehat a-b$,
\[
 \widehat a-a=\widehat G^{-1}\bigl[e-(\widehat G-G)a\bigr],
 \qquad \|\widehat G^{-1}\|_2\leq(s^2-\delta_G)^{-1}.
\]
Combining this assembly estimate with
\eqref{eq:wave-control-residual-bound} proves
\eqref{eq:wave-residual-certificate}. Let $a_n$ be the coefficients of the
orthogonal projection of $f^*$ onto $\operatorname{Ran}\Phi_n$; the exact
observation lower bound gives $\|a_n\|_2\leq F/\alpha$. To bound the distance, use
\[
 \operatorname{dist}(f^*,V)\leq\|f^*-\Psi a_n\|
 \leq\tau_n+d\|a_n\|_2,\qquad \|a_n\|_2\leq F/\alpha.
\]
The known-coefficient choice follows by using $\sum_kc_k^*v_k$ as an
approximant and applying the embedding estimate to each column error.
\end{proof}

The corollary is conditional on the stated upper bounds. A finite Gram spectrum
or two agreeing quadratures do not provide them for a general geometry. The
verification benchmark in Section~\ref{sec:certified-benchmark} supplies every
ingredient by analytic formulas and outward ball arithmetic. The certificates
there concern the continuous PDE; the diagnostics of Examples~1--4 retain the
finite-reference interpretation stated in Section~\ref{sec:numerics}.

\begin{corollary}[A certificate from the actual continuous Gram matrix]
\label{cor:actual-gram-certificate}
In the notation of Proposition~\ref{prop:wave-moment-residual}, suppose
validated numbers satisfy
\[
 G\succeq\gamma I>0,\quad \|\Psi\|\leq M,\quad\|b\|_2\leq B,
 \quad\|\widehat G-G\|_2\leq\delta_G<\gamma,
 \quad\|\widehat G\widehat a-b\|_2\leq\delta_b.
\]
Let $D\geq\operatorname{dist}(f^*,V)$,
$U\geq\|u^*\|_{L^2(Q)}$, and
$P^2\geq\sum_k\|\mathcal Pv_k\|_{L^2(Q)}^2$. Then
\begin{equation}\label{eq:actual-gram-certificate}
 \|\Psi\widehat a-f^*\|
 \leq \left(D^2+\frac{U^2P^2}{\gamma}\right)^{1/2}
 +\frac{M}{\gamma-\delta_G}
       \left(\delta_b+\frac{\delta_GB}{\gamma}\right).
\end{equation}
\end{corollary}
\begin{proof}
Proposition~\ref{prop:wave-moment-residual} bounds the exact-integral control
error by the first term. Moreover $\|G^{-1}b\|_2\leq B/\gamma$ and
$\|\widehat G^{-1}\|_2\leq(\gamma-\delta_G)^{-1}$. Subtracting the exact and
computed systems and applying $\|\Psi\|\leq M$ gives the second term.
\end{proof}
This alternative removes the sufficient synthesis-smallness test $d<\alpha$.
It requires validated continuous integrals of the actual returned fields,
and does not infer infinite-dimensional observability from a finite Gram
matrix. For the manufactured controls, $D$ is supplied by
\eqref{eq:manufactured-distance-bound}. The implementation checks Sylvester's
criterion by ball determinants and uses
$\gamma=\|G^{-1}\|_{\mathrm F}^{-1}$ and
$M=\|G\|_{\mathrm F}^{1/2}$ with outward bounds. These conservative spectral
bounds are valid once positive definiteness has been verified.

\begin{remark}[Computable inputs when the control is unknown]
\label{rem:unknown-control-inputs}
Corollaries~\ref{cor:wave-residual-certificate} and~\ref{cor:actual-gram-certificate}
are conditional certification mechanisms; Theorem~\ref{thm:fully-discrete-neural-hum}
below is a structural estimate whose constants are not evaluated in the
production examples. For a problem whose minimum-norm control is unknown, the
inputs of \eqref{eq:actual-gram-certificate} can be obtained as follows.
The quantities $\gamma$, $M$, $B$, $\delta_G$, $\delta_b$ and $P$ concern only the
returned fields and the data and are computable whenever the integrals can be
enclosed, as in Section~\ref{sec:certified-benchmark}. A norm bound for the
control follows from the data and the observability constant: for $z\in\mathcal E$,
$(f^*,Sz)_{L^2(Q_\omega)}=L(\phi_z)$, and the terminal pairing used in the proof of
Proposition~\ref{prop:final-field-certificate} gives $|L(\phi_z)|\leq R_T(0)\|z\|_{\mathcal E}$,
where $R_T(0)$ is the terminal energy residual of the uncontrolled state. With
$f^*=Sz^*$ this yields $\|f^*\|^2\leq R_T(0)\|z^*\|_{\mathcal E}\leq R_T(0)\|f^*\|/\alpha$,
so $F=R_T(0)/\alpha$ is admissible. Splitting $u^*$ into the free evolution of the
initial data and the state driven by $f^*$ from rest, Lemma~\ref{lem:control-to-state-stability}
gives $U=C_p\bigl(2TE_0/a_0\bigr)^{1/2}+C_pTF/\sqrt{2a_0}$, with $E_0$ the initial
energy. The remaining input is $D\geq\operatorname{dist}(f^*,V)$. The admissible
choice $D=\tau_n+dF/\alpha$ of Corollary~\ref{cor:wave-residual-certificate} reduces it
to a bound for the truncation error $\tau_n$, which requires problem-specific
information, for instance a bound for $\|z^*\|_{\mathcal E^{r_\star}}$ in
Theorem~\ref{thm:fd-truncation-rate}. If the exact Gram matrix does not couple retained and discarded
terminal modes, the discarded part of $f^*$ is the minimum-norm control of the discarded part of the
target, and $\tau_n\leq\|b_{\rm tail}\|/\alpha$ with the moments $b_{\rm tail}$ of the target tail;
this bound is used in Section~\ref{sec:truncated-benchmark}. Alternatively,
Proposition~\ref{prop:final-field-certificate} avoids $D$ altogether but requires a
validated enclosure of the forward residual $R_T(\widehat f)$. In both routes the
observability constant $\alpha$ is needed; it is explicit in the benchmarks of
Section~\ref{sec:certified-benchmark} but not for the geometries of Examples~1--4.
\end{remark}

\paragraph{A posteriori control of the final field.}
The preceding certificate separates sources of reconstruction error. A
complementary estimate can instead assess the final field without a prescribed
terminal dimension or known optimal HUM coefficients.
\begin{proposition}[Terminal and adjoint residuals of the final field]
\label{prop:final-field-certificate}
Assume the time-independent model and continuous observation hypothesis
\eqref{eq:observation-alpha-beta}. Let $\widehat v\in X$, set
$\widehat f=\widehat v|_{Q_\omega}$, and let $u_{\widehat f}$ be its exact
controlled state with the prescribed initial data. Define
\[
 R_T(\widehat f)^2=
 \|A^{1/2}(u_{\widehat f}(T)-u_T^0)\|_{L^2}^2+
 \|(u_{\widehat f})_t(T)-u_T^1\|_{L^2}^2,
 \qquad C_0=\frac{C_pT}{\sqrt{2a_0}}.
\]
Then
\begin{equation}\label{eq:final-field-certificate}
 \|\widehat f-f^*\|_{L^2(Q_\omega)}^2
 \leq\alpha^{-2}R_T(\widehat f)^2+
 C_0^2\|\mathcal P\widehat v\|_{L^2(Q)}^2.
\end{equation}
If an approximate forward state has measured terminal residual
$\widehat R_T$ and a validated terminal-energy error at most $\epsilon_F$,
one may replace $R_T(\widehat f)$ by $\widehat R_T+\epsilon_F$.
\end{proposition}
\begin{proof}
Let $\mathscr H_\omega=\operatorname{Ran}S$, which is closed by observability,
and write $\widehat f-f^*=q+r$ with
$q=P_{\mathscr H_\omega}\widehat f-f^*$ and
$r=(I-P_{\mathscr H_\omega})\widehat f$. Green's identity for the exact
forward state gives, for terminal data $z\in\mathcal E$,
\[
 |(q,Sz)_{L^2(Q_\omega)}|
 =|(\widehat f-f^*,Sz)_{L^2(Q_\omega)}|
 \leq R_T(\widehat f)\|z\|_{\mathcal E}.
\]
Here the terminal symplectic pairing has norm at most $R_T$ in the dual of
$L^2\times H_A^{-1}$. Choose $z$ with $Sz=q$. Since
$\alpha\|z\|_{\mathcal E}\leq\|q\|$, it follows that
$\|q\|\leq\alpha^{-1}R_T(\widehat f)$.
Let $\phi$ be the homogeneous adjoint with the terminal traces of
$\widehat v$. The field $w=\widehat v-\phi$ has zero terminal data and
$\mathcal Pw=\mathcal P\widehat v$. The backward forced-wave energy estimate
and Poincar\'e's inequality give $\|w\|_{L^2(Q)}\leq C_0\|\mathcal P\widehat v\|$.
Thus $\|r\|=\operatorname{dist}(\widehat f,\mathscr H_\omega)
\leq\|\widehat v-\phi\|_{L^2(Q_\omega)}\leq C_0\|\mathcal P\widehat v\|$.
Orthogonality of $q$ and $r$ proves the claim. Completion extends the argument
from strong trial fields to $X$. The final statement is the triangle inequality
in terminal energy.
\end{proof}
This estimate separates steering error from failure of the final field to be
an adjoint observation. A terminal residual alone cannot identify a
minimum-norm control, since exact controls may differ by a component in
$\mathscr H_\omega^\perp$. For a computed neural combination
$\widehat v=\sum_k\widehat a_kv_k$, both residuals concern that actual
combination and therefore reflect errors in its coefficients. The estimate
remains conditional on a continuous observability lower bound and validated
forward and residual evaluations; the M3 and finite-difference diagnostics of
Examples~1--4 do not by themselves evaluate its continuous right-hand side.

\subsection{A coupled finite-error estimate}

We now compare two finite-dimensional control reconstructions. The first is defined using the exact adjoint solutions and, after restriction to $Q_\omega$, is the orthogonal projection of the minimal-norm control onto the truncated observation space. The second is obtained by replacing the exact adjoint solutions with their adaptive Galerkin neural-network approximations.

By Proposition~\ref{prop:minimum-norm-exact-control}, the control $f^*\in\mathscr H_\omega$ satisfies the moment identity
\begin{equation}\label{eq:exact-control-moment-identity}
    (f^*,\varphi)_{L^2(Q_\omega)} = L(\varphi)
\end{equation}
for every exact adjoint solution $\varphi$.

For $n\geq 1$, define
\begin{equation*}
    b_n := \begin{pmatrix} L(\phi_1)\\ \vdots\\ L(\phi_n) \end{pmatrix} \in\mathbb{R}^n.
\end{equation*}
Let $a_n\in\mathbb{R}^n$ be the unique solution of
\begin{equation}\label{eq:exact-reconstruction-system}
    G_na_n=b_n,
\end{equation}
and define
\begin{equation*}
    f_n^H := \sum_{k=1}^n(a_n)_k\phi_k.
\end{equation*}
The superscript $H$ indicates that this is the finite-dimensional HUM reconstruction based on exact adjoint modes. By the definition of $\Phi_n$, its restriction to the observation domain is
\begin{equation*}
    f_n^H|_{Q_\omega} = \Phi_na_n \in \mathscr H_{n,\omega}.
\end{equation*}
The system \eqref{eq:exact-reconstruction-system} is uniquely solvable by Proposition~\ref{prop:exact-gram-stability}.

\begin{lemma}\label{lem:exact-reconstruction-projection}
The restriction of $f_n^H$ to $Q_\omega$ is the $L^2(Q_\omega)$-orthogonal projection of $f^*$ onto $\mathscr H_{n,\omega}$:
\begin{equation}\label{eq:exact-reconstruction-projection}
    f_n^H|_{Q_\omega} = P_{\mathscr H_{n,\omega}}f^*.
\end{equation}
Consequently,
\begin{equation}\label{eq:truncation-error-definition}
    \tau_n := \|f_n^H|_{Q_\omega}-f^*\|_{L^2(Q_\omega)} =  \operatorname{dist}_{L^2(Q_\omega)} (f^*,\mathscr H_{n,\omega}).
\end{equation}
Moreover,
\begin{equation}\label{eq:truncation-error-monotone}
    \tau_{n+1}\leq\tau_n,
\end{equation}
and
\begin{equation}\label{eq:truncation-error-convergence}
    \tau_n\longrightarrow 0 \qquad \text{as }n\to\infty.
\end{equation}
\end{lemma}

\begin{proof}
The moment identity gives $b_n=\Phi_n^*f^*$, so the normal equations
$\Phi_n^*(\Phi_na_n-f^*)=0$ characterize the orthogonal projection onto
$\mathscr H_{n,\omega}$. Nested terminal spaces yield monotonicity of the
projection error. Completeness of the terminal basis and continuity of $S$
give $\overline{\bigcup_n S\mathcal E_n}=\overline{S\mathcal E}=\mathscr H_\omega$.
Since $f^*\in\mathscr H_\omega$, the projection errors converge to zero.
\end{proof}

We now define the reconstruction obtained from the approximate adjoint solutions. Let
\begin{equation*}
    b_n^m := \begin{pmatrix} L(\phi_1^m)\\ \vdots\\ L(\phi_n^m) \end{pmatrix} \in\mathbb{R}^n.
\end{equation*}
Each $\phi_k^m$ belongs to its mode-dependent space $Y_{k,m}$; no common trial space is required in the reconstruction argument. Whenever $G_n^m$ is positive definite, let $a_n^m\in\mathbb{R}^n$ solve
\begin{equation*}
    G_n^ma_n^m=b_n^m,
\end{equation*}
and define
\begin{equation*}
    f_n^m := \sum_{k=1}^n(a_n^m)_k\phi_k^m.
\end{equation*}
Its restriction to $Q_\omega$ is
\begin{equation*}
    f_n^m|_{Q_\omega} = \Phi_n^ma_n^m.
\end{equation*}

\begin{remark}
For the exact adjoint solutions, the moment identity gives \[L(\phi_k)=(f^*,\phi_k)_{L^2(Q_\omega)}.\] In general, however, \[L(\phi_k^m)\neq(f^*,\phi_k^m)_{L^2(Q_\omega)},\] because $\phi_k^m$ is not an exact solution of the homogeneous adjoint equation. Therefore, $f_n^m|_{Q_\omega}$ need not be the orthogonal projection of $f^*$ onto \[\operatorname{span} \{\phi_1^m|_{Q_\omega},\ldots,\phi_n^m|_{Q_\omega}\}.\] The perturbations of both the Gram matrix and the right-hand side must therefore be estimated.
\end{remark}

Let
\begin{equation*}
    C_L \geq \|L\|_{X'} = \sup_{v\in X\setminus\{0\}} \frac{|L(v)|}{\|v\|_{LS}}.
\end{equation*}
Lemma~\ref{lem:L-continuity} gives \[|L(v)|\leq C_L\|v\|_{LS},\qquad v\in X.\]

\begin{lemma}\label{lem:moment-gram-perturbation}
Let $\varepsilon_{n,m}$ and $d_{n,m}$ be defined by \eqref{eq:adjoint-uniform-error} and \eqref{eq:dnm-definition}, respectively. Then
\begin{equation}\label{eq:rhs-perturbation-bound}
    \|b_n^m-b_n\|_2 \leq C_L\sqrt{n}\,\varepsilon_{n,m},
\end{equation}
and
\begin{equation}\label{eq:gram-perturbation-bound}
    \|G_n^m-G_n\|_2 \leq d_{n,m}(2\beta+d_{n,m}).
\end{equation}
\end{lemma}

\begin{proof}
By the continuity of $L$, $\|b_n^m-b_n\|_2^2=\sum_{k=1}^n|L(\phi_k^m-\phi_k)|^2\leq C_L^2\sum_{k=1}^n\|\phi_k^m-\phi_k\|_{LS}^2\leq C_L^2n\varepsilon_{n,m}^2$, which is \eqref{eq:rhs-perturbation-bound}. With $D_{n,m}:=\Phi_n^m-\Phi_n$, Proposition~\ref{prop:approximate-gram-stability} gives $\|D_{n,m}\|\leq d_{n,m}$, and $G_n^m-G_n=\Phi_n^*D_{n,m}+D_{n,m}^*\Phi_n+D_{n,m}^*D_{n,m}$ together with $\|\Phi_n\|\leq\beta$ gives \eqref{eq:gram-perturbation-bound}.
\end{proof}

\begin{theorem}\label{thm:coupled-reconstruction-error}
Suppose that
\begin{equation} \label{eq:coupled-stability-condition}
    d_{n,m} = C_X\sqrt{n}\,\varepsilon_{n,m} \leq \frac{\alpha}{2}.
\end{equation}
Then $G_n^m$ is symmetric positive definite, the approximate reconstruction $f_n^m$ is well defined, and
\begin{equation}\label{eq:coupled-reconstruction-error}
    \|f_n^m|_{Q_\omega}-f^*\|_{L^2(Q_\omega)} \leq \tau_n + C_{\mathrm{rec}}\sqrt{n}\,\varepsilon_{n,m},
\end{equation}
where $C_{\mathrm{rec}}>0$ is independent of $n$ and $m$ and is defined explicitly in \seqref{eq:reconstruction-constant}.
\end{theorem}

The proof, a standard perturbation argument for the reconstruction system, is given in Supplementary Section~\sref{app:proofs-perturbation}.

\begin{corollary}
Suppose that $m=m(n)$ is chosen so that
\begin{equation*}
    \sqrt{n}\, \varepsilon_{n,m(n)} \longrightarrow 0 \qquad \text{as }n\to\infty.
\end{equation*}
Then
\begin{equation*}
    f_n^{m(n)}|_{Q_\omega} \longrightarrow f^* \qquad \text{in }L^2(Q_\omega).
\end{equation*}
\end{corollary}

\begin{proof}
By Lemma~\ref{lem:exact-reconstruction-projection}, $\tau_n\to0$, and $d_{n,m(n)}=C_X\sqrt n\,\varepsilon_{n,m(n)}\to0$, so \eqref{eq:coupled-stability-condition} holds for all large $n$ and Theorem~\ref{thm:coupled-reconstruction-error} applies.
\end{proof}

\medskip

\subsection{Fully discrete error estimates and convergence rates}

We now combine the truncation error, the adaptive neural approximation error, and the numerical errors arising from the least-squares and reconstruction systems. The total error has the schematic decomposition 
\[\begin{aligned}
    \text{total error}\leq{}& \text{truncation error}+\text{adjoint approximation error}\\&+\text{assembly and algebraic-solve errors}.
\end{aligned}\]

\subsubsection{Bounds on the discrete adjoint error}\label{sec:discrete-adjoint-error}

Let \[\widetilde\phi_k^{m,h}\in X\] denote the adjoint approximation returned by the numerical pipeline. Here $m$ is the number of accepted neural enrichment steps and $h$ denotes a deterministic space--time discretization or quadrature parameter; for example, one may take $h=\max\{h_x,\Delta t\}$ for spatial mesh size $h_x$ and time step $\Delta t$.

Let $\phi_k^m$ denote the exact Galerkin projection in the accepted space $Y_{k,m}$. We separate the adaptive Galerkin error from the numerical error by setting
\begin{equation*}
    \chi_{n,m,h} := \max_{1\leq k\leq n} \|\widetilde\phi_k^{m,h}-\phi_k^m\|_{LS}.
\end{equation*}
No membership of $\widetilde\phi_k^{m,h}$ in $Y_{k,m}$ is required here. Thus $\chi_{n,m,h}$ may include quadrature and algebraic-solve errors as well as any postprocessing error, including coefficient pruning. The theoretical quantity $\phi_k^m$ remains the exact Galerkin projection in the unmodified accepted space. Throughout the fully discrete analysis, $E_{n,m,h}^{\mathrm{adj}}$ denotes any number satisfying
\begin{equation}\label{eq:adjoint-error-bound}
    \max_{1\leq k\leq n}\|\widetilde\phi_k^{m,h}-\phi_k\|_{LS}
\leq E_{n,m,h}^{\mathrm{adj}}.
\end{equation}
Theorem~\ref{thm:fully-discrete-neural-hum} uses only this property. In particular, $E_{n,m,h}^{\mathrm{adj}}$ bounds absolute errors relative to the continuous adjoints; the quadrature-based Ritz residuals of Section~\ref{sec:numerics} do not supply such a bound without the consistency and integration estimates specified there. Since $\widetilde\phi_k^{m,h}\in X$, identity \eqref{eq:objective-error-identity} gives \[\|\widetilde\phi_k^{m,h}-\phi_k\|_{LS}=\sqrt{2\mathcal J_k(\widetilde\phi_k^{m,h})}.\] Hence the a posteriori choice $E_{n,m,h}^{\mathrm{adj}}=\max_{1\leq k\leq n}\sqrt{2\mathcal J_k(\widetilde\phi_k^{m,h})}$ is admissible when these functionals are evaluated exactly. It covers every mode, including modes that terminate by stagnation, a restart limit, or a step budget. An a priori choice is also available. If every mode either undergoes $m$ accepted enrichment steps or terminates earlier only after satisfying \eqref{eq:adjoint-stopping-test}, and if the accepted residual directions satisfy \eqref{eq:AD-uniform-quality}, then Corollary~\ref{cor:uniform-adaptive-adjoint-rate} and the triangle inequality show that
\begin{equation}\label{eq:apriori-computed-adjoint-error}
    E_{n,m,h}^{\mathrm{adj}}=\mathfrak B_n\mu_\star^m+\mathrm{Tol}_n+\chi_{n,m,h}
\end{equation}
satisfies \eqref{eq:adjoint-error-bound}. Here $\mathrm{Tol}_n$ may be set to zero if every mode undergoes exactly $m$ accepted enrichment steps.

\subsubsection{Perturbation of the reconstruction system}\label{sec:reconstruction-perturbation}

Define the fully discrete synthesis operator
\begin{equation*}
    \widetilde\Phi_n^{m,h}c:=\sum_{k=1}^n c_k\widetilde\phi_k^{m,h}|_{Q_\omega},\qquad c\in\mathbb R^n,
\end{equation*}
and its exact-integral Gram matrix and right-hand side by
\begin{equation*}
    \widetilde G_n^{m,h}:=(\widetilde\Phi_n^{m,h})^*\widetilde\Phi_n^{m,h},\qquad\widetilde b_n^{m,h}:=
    \begin{pmatrix}
        L(\widetilde\phi_1^{m,h})\\\vdots\\L(\widetilde\phi_n^{m,h})
    \end{pmatrix}.
\end{equation*}

The quantities assembled numerically are written as
\begin{equation*}
    \widehat G_n^{m,h}=\widetilde G_n^{m,h}+\Delta G_n^{m,h},\qquad\widehat b_n^{m,h}=\widetilde b_n^{m,h}+\Delta b_n^{m,h},
\end{equation*}
where $\widehat G_n^{m,h}$ is assumed to be symmetric. Set
\begin{equation*}
    \delta_{G,n,m,h}:=\|\Delta G_n^{m,h}\|_2,\qquad\delta_{b,n,m,h}:=\|\Delta b_n^{m,h}\|_2.
\end{equation*}
The assembly errors include any quadrature, finite time-expansion, and floating-point errors relative to these exact-integral quantities. Any residual produced by an inexact solution of the reconstruction system may be absorbed into $\Delta b_n^{m,h}$ and hence into $\delta_{b,n,m,h}$. Let $\widehat a_n^{m,h}$ solve
\begin{equation}\label{eq:fully-discrete-reconstruction-system}
    \widehat G_n^{m,h}\widehat a_n^{m,h}=\widehat b_n^{m,h},
\end{equation}
and define the fully discrete control by
\begin{equation*}
    \widehat f_n^{m,h}:=\sum_{k=1}^n(\widehat a_n^{m,h})_k\widetilde\phi_k^{m,h}.
\end{equation*}

\begin{theorem}\label{thm:fully-discrete-neural-hum}
Assume that
\begin{equation}\label{eq:fully-discrete-stability-conditions}
    C_X\sqrt n\,E_{n,m,h}^{\mathrm{adj}}\leq\frac{\alpha}{2},\qquad\delta_{G,n,m,h}\leq\frac{\alpha^2}{8}.
\end{equation}
Then $\widehat G_n^{m,h}$ is symmetric positive definite and
\begin{equation}\label{eq:fully-discrete-error}
    \|\widehat f_n^{m,h}|_{Q_\omega}-f^*\|_{L^2(Q_\omega)}\leq\tau_n+C_{\mathrm{fd}}\left(\sqrt n\,E_{n,m,h}^{\mathrm{adj}}+\delta_{G,n,m,h}+\delta_{b,n,m,h}\right),
\end{equation}
where $C_{\mathrm{fd}}>0$, defined in the proof (Supplementary Section~\sref{app:proofs-perturbation}), is independent of $n$, $m$, and $h$. The same proof allows $\sqrt n\,E_{n,m,h}^{\mathrm{adj}}$ to be replaced throughout the stability condition and estimate by $(\sum_{k=1}^n E_k^2)^{1/2}$ whenever validated modewise bounds $E_k\geq\|\widetilde\phi_k^{m,h}-\phi_k\|_{LS}$ are available.
\end{theorem}

Theorem~\ref{thm:fully-discrete-neural-hum} is a structural estimate: it shows how truncation, adjoint and assembly errors combine, but its hypothesis \eqref{eq:fully-discrete-stability-conditions} and its constant are not evaluated for the production examples, and in Example~2 the hypothesis is not met at the largest tested $n$ (Section~\ref{sec:example2}). The computable certification mechanisms are Corollaries~\ref{cor:wave-residual-certificate} and~\ref{cor:actual-gram-certificate} and Proposition~\ref{prop:final-field-certificate}, with the inputs discussed in Remark~\ref{rem:unknown-control-inputs}.

If entrywise assembly errors satisfy \[\max_{1\leq i,j\leq n}|(\Delta G_n^{m,h})_{ij}|\leq\eta_{G,n,m,h},\qquad\max_{1\leq i\leq n}|(\Delta b_n^{m,h})_i|\leq\eta_{b,n,m,h},\]
then
\begin{equation}\label{eq:entrywise-quadrature}
    \delta_{G,n,m,h}\leq n\eta_{G,n,m,h},\qquad\delta_{b,n,m,h}\leq\sqrt n\,\eta_{b,n,m,h}.
\end{equation}
Indeed, the first bound follows from $\|A\|_2\leq\|A\|_{\mathrm F}$ and the second from the Euclidean norm.

\subsubsection{Quantitative truncation rate}

From this point onward, we specialize the terminal basis to the eigenfunction basis introduced earlier, namely $w_{2j-1}=(e_j,0)$ and $w_{2j}=(0,\sqrt{\lambda_j}\,e_j)$. Let \[Ae_j=\lambda_je_j,\qquad 0<\lambda_1\leq\lambda_2\leq\cdots,\] where the eigenvalues are repeated according to their multiplicities. Since each spatial eigenvalue generates two terminal modes, define
\begin{equation}\label{eq:fd-repeated-terminal-spectrum}
    \vartheta_{2j-1}=\vartheta_{2j}=\lambda_j,\qquad j\geq1.
\end{equation}
Let \[\mathcal E_n=\operatorname{span}\{w_1,\ldots,w_n\}\] and let $P_n^{\mathcal E}$ denote the $\mathcal E$-orthogonal projection onto $\mathcal E_n$. For $r_\star\geq0$, define
\begin{equation*}
    \mathcal E^{r_\star}:=\left\{ z\in\mathcal E:\sum_{k=1}^{\infty}\vartheta_k^{r_\star}|(z,w_k)_{\mathcal E}|^2<\infty\right\},
\end{equation*}
with norm \[\|z\|_{\mathcal E^{r_\star}}^2:=\sum_{k=1}^{\infty}\vartheta_k^{r_\star}|(z,w_k)_{\mathcal E}|^2.\]
For $s\in\mathbb R$, let $H_A^s$ denote the completion of $\operatorname{span}\{e_j:j\ge1\}$ with respect to the spectral norm \[\left\|\sum_{j\ge1}z_j e_j\right\|_{H_A^s}^2:=\sum_{j\ge1}\lambda_j^s |z_j|^2.\]
With this notation, \[\mathcal E^{r_\star}=H_A^{r_\star}\times H_A^{r_\star-1},\] with equality of the corresponding spectral norms. For $s\ge0$, $H_A^s=D(A^{s/2})$, whereas for $s<0$, $H_A^s$ is the corresponding extrapolation space, equivalently the dual of $H_A^{-s}$ with $L^2(\Omega)$ as pivot space.

\begin{theorem}\label{thm:fd-truncation-rate}
The two-sided observation estimate implies that there exists a unique $z^*\in\mathcal E$ such that \[f^*=Sz^*.\] If $z^*\in\mathcal E^{r_\star}$ for some $r_\star>0$, then
\begin{equation}\label{eq:fd-truncation-rate}
    \tau_n\leq\beta\,\vartheta_{n+1}^{-r_\star/2}\|z^*\|_{\mathcal E^{r_\star}}.
\end{equation}
In particular, for $n=2J$,
\begin{equation*}
    \tau_{2J}\leq\beta\,\lambda_{J+1}^{-r_\star/2}\|z^*\|_{\mathcal E^{r_\star}}.
\end{equation*}
\end{theorem}

The proof, a spectral tail estimate, is given in Supplementary Section~\sref{app:proofs-perturbation}.

\begin{remark}
The membership $z^*\in\mathcal E^{r_\star}$ is an additional source condition on the HUM preimage of the minimum-norm control. Smoothness of the initial and target data alone does not automatically imply it without corresponding regularity information for the inverse observation operator.
\end{remark}

\subsubsection{Coupled convergence rate}

To express the dependence of the neural approximation error on the terminal mode, assume that, for some constants $C_\phi>0$ and $\gamma\geq0$,
\begin{equation}\label{eq:adjoint-LS-growth}
    \mathfrak B_n\leq C_\phi\vartheta_n^\gamma.
\end{equation}

\begin{corollary}
Assume the source condition in Theorem~\ref{thm:fd-truncation-rate}, the adjoint growth condition \eqref{eq:adjoint-LS-growth}, the uniform neural approximation condition \eqref{eq:AD-uniform-quality}, and that every mode either undergoes $m$ accepted enrichment steps or terminates earlier only after satisfying \eqref{eq:adjoint-stopping-test}. Suppose also that

\begin{equation*}
    C_X\sqrt n\left(C_\phi\vartheta_n^\gamma\mu_\star^m+\mathrm{Tol}_n+\chi_{n,m,h}\right)\leq\frac{\alpha}{2},\qquad\delta_{G,n,m,h}\leq\frac{\alpha^2}{8}.
\end{equation*}

Then
\begin{align*}
    \|\widehat f_n^{m,h}|_{Q_\omega}-f^*\|_{L^2(Q_\omega)}\leq{}& C_{\mathrm{tr}} \vartheta_{n+1}^{-r_\star/2}\\
    &+C_{\mathrm{fd}}\left({\sqrt n\,C_\phi\vartheta_n^\gamma\mu_\star^m+\sqrt n\,\mathrm{Tol}_n}+\sqrt n\,\chi_{n,m,h}+\delta_{G,n,m,h}+\delta_{b,n,m,h}\right),
\end{align*}
where \[C_{\mathrm{tr}}:=\beta\|z^*\|_{\mathcal E^{r_\star}}.\]
\end{corollary}

\begin{proof}
With the a priori choice \eqref{eq:apriori-computed-adjoint-error} and \eqref{eq:adjoint-LS-growth}, \[E_{n,m,h}^{\mathrm{adj}}\leq C_\phi\vartheta_n^\gamma\mu_\star^m+\mathrm{Tol}_n+\chi_{n,m,h}.\] The conclusion follows by substituting this estimate and \eqref{eq:fd-truncation-rate} into Theorem~\ref{thm:fully-discrete-neural-hum}.
\end{proof}

Let $\nu_j$ denote the Dirichlet eigenvalues of $-\Delta$ on $\Omega$. The min--max principle and $a_0\|\nabla v\|_{L^2}^2\leq \int_\Omega a|\nabla v|^2\leq a_1\|\nabla v\|_{L^2}^2$ give \[a_0\nu_j\leq\lambda_j\leq a_1\nu_j.\] Weyl's law for the Dirichlet Laplacian on a smooth bounded domain \cite{weyl1912asymptotic} therefore gives \[\lambda_j\asymp j^{2/d}.\] Since every spatial eigenvalue occurs twice in \eqref{eq:fd-repeated-terminal-spectrum}, we also have \[\vartheta_n\asymp n^{2/d}.\] Consequently,
\begin{equation}\label{eq:Weyl-fully-discrete-rate}
    \|\widehat f_n^{m,h}|_{Q_\omega}-f^*\|_{L^2(Q_\omega)}\lesssim n^{-r_\star/d}+{n^{1/2+2\gamma/d}\mu_\star^m+\sqrt n\,\mathrm{Tol}_n}+\sqrt n\,\chi_{n,m,h}+\delta_{G,n,m,h}+\delta_{b,n,m,h}.
\end{equation}

The neural contribution is of the same order as the truncation contribution provided that
\begin{equation}\label{eq:logarithmic-enrichment-choice}
    m(n)\geq\frac{\frac12+\frac{r_\star+2\gamma}{d}}{{|\log\mu_\star|}}\log n+m_0
\end{equation}
for a sufficiently large offset $m_0$. Thus only a logarithmic number of accepted adaptive enrichment steps per mode is required under the uniform contraction hypothesis. For the stopping contribution to have the same order, one also requires $\sqrt n\,\mathrm{Tol}_n=O(n^{-r_\star/d})$. In particular, a fixed nonzero relative stopping threshold does not by itself meet this asymptotic requirement; its relative factor must be decreased with $n$ unless an independent mode-uniform estimate yields the required absolute decay. This statement concerns accepted enrichments and does not, by itself, bound the network width or the nonlinear optimization cost of producing each accepted direction.

If, in addition,
\begin{equation*}
    \sqrt n\,\mathrm{Tol}_n+\sqrt n\,\chi_{n,m(n),h(n)}+\delta_{G,n,m(n),h(n)}+\delta_{b,n,m(n),h(n)}=O(n^{-r_\star/d}),
\end{equation*}
then
\begin{equation*}
    \|\widehat f_n^{m(n),h(n)}|_{Q_\omega}-f^*\|_{L^2(Q_\omega)}=O(n^{-r_\star/d}).
\end{equation*}

For the terminal eigenbasis used above, the value $\gamma=1$ follows under the stated smoothness and elliptic-regularity assumptions. Indeed, $\mathcal P\phi_k=0$ and elliptic regularity gives \[\|\phi_{2j-1}\|_{LS}^2=\|e_j\|_{H^2(\Omega)}^2+\|e_j\|_{H^1(\Omega)}^2\lesssim\lambda_j^2,\] whereas the eigenvalue identity and ellipticity give \[\|\phi_{2j}\|_{LS}^2=\|\sqrt{\lambda_j}e_j\|_{H^1(\Omega)}^2\lesssim\lambda_j^2.\] Thus $\|\phi_k\|_{LS}\lesssim\vartheta_k$, and one may take $\gamma=1$ in \eqref{eq:adjoint-LS-growth}. In this case, \eqref{eq:logarithmic-enrichment-choice} becomes
\begin{equation*}
    m(n)\geq \frac{\frac12+\frac{r_\star+2}{d}}{{|\log\mu_\star|}}\log n+m_0.
\end{equation*}

\begin{remark}[Scope of the logarithmic count]\label{rem:log-count-scope}
The geometric factor $\mu_\star^m$ is a direct consequence of hypothesis \eqref{eq:AD-uniform-quality}, which requires every accepted direction, itself a network of width $p_{k,m}$, to capture a fixed fraction of the current error. For greedy algorithms that add one neuron at a time from a fixed dictionary, only algebraic rates in the number of neurons are available in general~\cite{siegel2023greedy}; the geometric rate here reflects the strength of \eqref{eq:AD-uniform-quality} and is not an optimization guarantee. The constants matter as well. The acceptance gate $\varrho_{\mathrm{acc}}=0.1$ used in Section~\ref{sec:numerics} would, if it held for the exact correlations, only guarantee $\mu_\star=(1-0.01)^{1/2}\approx0.995$, that is, about $460$ accepted steps per decade of error reduction; the smallest sampled correlation $0.112$ of Example~1 corresponds to about $365$ steps per decade. The observed reductions are much faster, typically one to six steps per mode. The bound \eqref{eq:logarithmic-enrichment-choice} therefore describes the asymptotic dependence on $n$ under \eqref{eq:AD-uniform-quality}, not the operating regime of the experiments.
\end{remark}

Algebraic rates for the quadrature and solve errors translate into rates in $h$ through \eqref{eq:entrywise-quadrature}; see Supplementary Section~\sref{app:proofs-perturbation}.

\subsubsection{Propagation of the control error to the state}

 \begin{lemma}\label{lem:control-to-state-stability}
Let $f,f_n\in L^2(Q_\omega)$, and let $u,u_n$ be the solutions of \eqref{Wave1} corresponding to $f,f_n$, respectively, with the same initial data. Then \[\|u-u_n\|_{L^2(Q)} \le \frac{C_pT}{\sqrt{2a_0}} \|f-f_n\|_{L^2(Q_\omega)}, \] and \[ \left( \|(u-u_n)_t(T)\|_{L^2(\Omega)}^2+\int_\Omega a(x)|\nabla(u-u_n)(T,x)|^2\,dx \right)^{1/2} \leq\sqrt T\,\|f-f_n\|_{L^2(Q_\omega)},\] where $C_p$ is the Poincar\'e constant of $\Omega$.
\end{lemma}

The proof is the energy estimate for the difference of the states; it is given in Supplementary Section~\sref{app:proofs-perturbation}.

Combining Theorem~\ref{thm:fully-discrete-neural-hum} with Lemma~\ref{lem:control-to-state-stability} shows that every control-error estimate obtained above immediately yields the same convergence order for the associated state in $L^2(Q)$ and for its terminal energy.

\medskip
\noindent\textbf{Resulting numerical procedure.}

For a prescribed truncation dimension $n$, the method proceeds as follows.
\begin{enumerate}
    \item Construct the first $n$ energy-orthonormal terminal pairs.
    \item For each mode $k$, initialize $Y_{k,0}=\{0\}$ and $\phi_k^0=0$.
	\item At enrichment step $m$, choose hidden features \eqref{eq:hidden-neural-features}. During the Adam phase, solve the output-layer Galerkin system \eqref{eq:local-output-system} after each hidden-parameter update. A subsequent joint refinement of the output and hidden parameters may be used, with the returned candidate tested directly by \eqref{eq:discrete-residual-correlation}.
	\item Test the resulting direction on independent training and validation quadratures and apply the correlation and rank gates \eqref{eq:discrete-direction-gates}. If a gate fails, increase the width, number of restarts, or training budget and repeat the local solve.
	\item After accepting one direction, enlarge $Y_{k,m}$ and recompute the global Galerkin projection \eqref{eq:adjoint-projection}. Continue until \eqref{eq:adjoint-stopping-test} is met.
	\item When certified estimates of the constants and errors are available, refine the accepted-enrichment count and deterministic discretization until \[C_X\sqrt n\,E_{n,m,h}^{\mathrm{adj}}\le\frac{\alpha}{2},\qquad \delta_{G,n,m,h}\le\frac{\alpha^2}{8}.\]
	\item Solve the reconstruction system by a stable factorization and verify the resulting control with an independent forward solve.
\end{enumerate}
Step 6 is an analytical certification target. If the constants or exact assembly errors are unavailable, numerical sensitivity checks cannot be substituted for the two displayed inequalities and should be reported as diagnostics only. In a certification-driven refinement strategy, the dimension $n$ would be increased only after the estimated contributions $\sqrt n\,E_{n,m,h}^{\mathrm{adj}}$, $\delta_{G,n,m,h}$, and $\delta_{b,n,m,h}$ in \eqref{eq:fully-discrete-error} are below the chosen tolerance. Examples~1--3 instead train the first 32 modes and evaluate fixed prefixes $n\in\{4,8,16,32\}$; Example~4 trains the first eight. These examples do not implement Step 6 or a certified adaptive choice of $n$. The separate verification benchmark checks the residual certificate at fixed $n=4$ using explicit continuous constants and ball enclosures.

\begin{remark}
Fix $\Omega$, $a$, $\omega$, $T$, the terminal basis, and the computed adjoints $\widetilde\phi_k^{m,h}$. The adjoint training problems are independent of the initial and target states of the controlled equation. For a family of such data indexed by $\ell$, only the moment functional $L^{(\ell)}$ and the reconstruction right-hand side change. Writing $\widehat L^{(\ell)}$ for its numerical evaluation, the control coefficients satisfy
\[
\widehat G_n^{m,h}\widehat a_{n,\ell}^{m,h}=\widehat b_{n,\ell}^{m,h},
\qquad
(\widehat b_{n,\ell}^{m,h})_k=\widehat L^{(\ell)}(\widetilde\phi_k^{m,h}),
\quad 1\leq k\leq n.
\]
The same adjoint representations, Gram matrix, and matrix factorization can therefore be reused, without retraining, for new initial and target states. A comparison with modal HUM must account for this common reuse property and distinguish the initial construction cost from the cost of each additional control reconstruction.
\end{remark}

\subsection{Time-dependent coefficients}\label{sec:time-dependent}

The construction uses the coefficient through the operator $\mathcal P$, the energy estimates, and the norm of the terminal space. We record the changes needed when the coefficient also depends on time. This case is used in Example~4 of Section~\ref{sec:numerics}, where the separated modal formula for the adjoint fields no longer solves the adjoint equation.

\medskip
\noindent\textbf{Hypothesis (H$_t$).} The coefficient satisfies $a\in C^\infty([0,T]\times\overline\Omega)$ and $0<a_0\leq a(t,x)\leq a_1$ on $[0,T]\times\overline\Omega$. In \eqref{Wave1}, \eqref{Wave3}, and in $\mathcal P$, the coefficient $a(x)$ is replaced by $a(t,x)$. Let $A(t)=-\nabla\cdot(a(t,\cdot)\nabla)$ be the Dirichlet realization, set $A_T:=A(T)$, and define
\[
\kappa(t):=\left\|\frac{\partial_t a(t,\cdot)}{a(t,\cdot)}\right\|_{L^\infty(\Omega)},\qquad K:=\int_0^T\kappa(t)\,dt.
\]
The terminal space is $\mathcal E_T:=L^2(\Omega)\times H_{A_T}^{-1}(\Omega)$, with $\|z^1\|_{H_{A_T}^{-1}(\Omega)}=\|A_T^{-1/2}z^1\|_{L^2(\Omega)}$, and the terminal basis is chosen $\mathcal E_T$-orthonormal in $\mathcal V$; for instance, $(e_j,0)$ and $(0,\sqrt{\lambda_j}e_j)$ with $A_Te_j=\lambda_je_j$. Under (H$_t$), standard regularity theory for second-order hyperbolic equations with smooth time-dependent coefficients (see, e.g., \cite[Section~7.2]{evans2010pde}) gives adjoint solutions in $X_0$ for terminal data in $\mathcal V$, and the least-squares space $X$ is defined as before.

\begin{proposition}\label{prop:time-dependent-coefficients}
Assume (H$_t$).
\begin{enumerate}
\item[(a)] Lemma~\ref{Lemma:Wave}, Proposition~\ref{prop:minimum-norm-exact-control}, and Lemma~\ref{Density} remain valid.
\item[(b)] For $v\in X_0$ with $\mathcal Pv=g$, the energy $E(t)=\frac12\|v_t(t)\|_{L^2(\Omega)}^2+\frac12\int_\Omega a(t,x)|\nabla v(t,x)|^2\,dx$ satisfies
\begin{equation}\label{eq:time-dependent-energy}
\sqrt{E(t)}\leq e^{K/2}\left(\sqrt{E(T)}+\frac{1}{\sqrt2}\int_t^T\|g(s)\|_{L^2(\Omega)}\,ds\right),\qquad t\in[0,T].
\end{equation}
Consequently, Lemmas~\ref{lem:ls-embedding}, \ref{lem:L-continuity}, and~\ref{lem:control-to-state-stability} hold with their constants multiplied by $e^{K/2}$; in particular, $C_X$ in \eqref{eq:explicit-ls-embedding-constant} may be multiplied by $e^{K/2}$.
\item[(c)] For terminal data $z\in\mathcal V$ and the corresponding adjoint solution $\phi$, the dual energy
\[
F(t):=\|\phi(t)\|_{L^2(\Omega)}^2+\|A(t)^{-1/2}\phi_t(t)\|_{L^2(\Omega)}^2
\]
satisfies
\begin{equation}\label{eq:dual-energy-identity}
F'(t)=-\int_\Omega\partial_ta(t,x)\,|\nabla\psi(t,x)|^2\,dx,\qquad \psi(t):=A(t)^{-1}\phi_t(t).
\end{equation}
Hence $e^{-K}\|z\|_{\mathcal E_T}^2\leq F(t)\leq e^{K}\|z\|_{\mathcal E_T}^2$ for $t\in[0,T]$, and the upper bound in \eqref{eq:two-sided-observation}, with $\mathcal E$ replaced by $\mathcal E_T$, holds with $C_{\mathrm{dir}}=Te^{K}$.
\item[(d)] Assume the lower bound in \eqref{eq:two-sided-observation} with $\mathcal E$ replaced by $\mathcal E_T$. Then Propositions~\ref{prop:exact-gram-stability} and~\ref{prop:approximate-gram-stability}, Theorems~\ref{thm:coupled-reconstruction-error} and~\ref{thm:fully-discrete-neural-hum}, and, for the eigenpairs of $A_T$, Theorem~\ref{thm:fd-truncation-rate} hold with $\beta=\sqrt T\,e^{K/2}$ and with $C_X$ multiplied by $e^{K/2}$.
\end{enumerate}
\end{proposition}

The proof modifies the energy arguments by a Gronwall factor; it is given in Supplementary Section~\sref{app:proofs-time-dependent}.

\begin{remark}
Hypothesis (H$_t$) does not imply the lower observation bound. For a time-dependent coefficient, observability depends on the bicharacteristics of $\tau^2-a(t,x)|\xi|^2$ and remains an assumption; geometric control conditions with time dependence are studied, for time-dependent observation domains, in~\cite{lerousseau2017geometric}. The source condition in Theorem~\ref{thm:fd-truncation-rate} and the residual-capture condition \eqref{eq:AD-uniform-quality} also remain separate assumptions. The adjoint energy is no longer conserved; the factor $e^{K}$ bounds its possible growth or decay. Time dependence alone does not destroy separation of variables: for $a(t,x)=m(t)a_1(x)$ the operators $A(t)$ share the eigenfunctions of $-\nabla\cdot(a_1\nabla)$, and the modes decouple into Hill-type equations. When the operators $A(t)$ have no common eigenbasis, as in Example~4, the modal formula \eqref{eq:direct-modal-adjoint-comparator} does not solve \eqref{Wave3}, and a Galerkin representation in a fixed spatial basis leads to a coupled system of ordinary differential equations.
\end{remark}
			
\section{Numerical experiments}\label{sec:numerics}
\subsection{Implementation and interpretation of the errors}

Examples~1--4 use Ritz terminal pairs $w_{2j-1,R}=(e_{j,R},0)$ and $w_{2j,R}=(0,\sqrt{\lambda_{j,R}}e_{j,R})$. Here $(\lambda_{j,R},e_{j,R})$ is obtained from a spectral Rayleigh--Ritz approximation of $A=-\nabla\cdot(a\nabla)$ in each example~\cite{babuska1991eigenvalue}. We suppress the subscript $R$ on eigenpairs in the example descriptions and figure labels. The associated modal reference fields, denoted $\phi_{k,R}$, use the Ritz eigenvalue for their temporal oscillation. They are reference-model fields rather than exact solutions of the continuous adjoint equation.  A fixed spectral cutoff defines a finite reference model and its minimum-norm HUM control $f_{\mathrm{ref}}$.  Within that same finite model, $f_{n,\mathrm{ref}}^H$ denotes the control reconstructed from the first $n$ reference adjoint modes, and $f_n^m$ uses their trained neural Galerkin approximations. We write
\begin{equation*}
    \tau_n^{\mathrm{ref}}=\|f_{n,\mathrm{ref}}^H-f_{\mathrm{ref}}\|_{L^2(Q_\omega)}, \qquad \eta_n=\|f_n^m-f_{n,\mathrm{ref}}^H\|_{L^2(Q_\omega)}.
\end{equation*}
The symbol $\tau_n^{\mathrm{ref}}$ is a measured finite-model truncation error; the $\tau_n$ of Theorem~\ref{thm:fully-discrete-neural-hum} measures the distance to the exact PDE control $f^*$.  We do not identify these two quantities. All norms of controls in the tables are over $Q_\omega$. 
\newcommand{\RESetaminA}{\ensuremath{3.04\cdot10^{-5}}}
\newcommand{\RESetaminB}{\ensuremath{1.64\cdot10^{-5}}}
\newcommand{\RESetaminC}{\ensuremath{3.33\cdot10^{-7}}}
\newcommand{\RESmaxchange}{\ensuremath{6.1\cdot10^{-5}}}
\newcommand{\RESrepro}{\ensuremath{1.7\cdot10^{-10}}}
The distances $\eta_n$ are evaluated directly from the difference fields. Re-evaluating them for the stored fields with $1.5$ times as many Legendre components or Gauss nodes in time, or with a rule on $\omega$ that is $1.5$ times finer in each coordinate, changes them by at most $\RESmaxchange$ relative in Examples~1--3 and by at most $3\cdot10^{-13}$ in Example~4 (Supplementary Tables~\sref{tab:resolution-examples123} and~\sref{tab:resolution}). This sensitivity check supports the accuracy of the reported values on the rules tested; it is not a validated bound for the quadrature error.

\paragraph{Accuracy targets.} In Examples~1--4 the computational target is fixed in advance: to reproduce, at prescribed $n$ ($n\leq32$, and $n\leq8$ in Example~4), the truncated HUM control of the reference model with the accuracy permitted by the adjoint tolerances, relative residuals of $10^{-3}$--$10^{-4}$. The dimension $n$ is not chosen to meet a steering tolerance. A small $\eta_n$ therefore shows agreement with the truncated reference control; it does not mean that the steering problem has been solved, and the terminal residuals $R_T$ reported below remain substantial at these $n$.

We also report the terminal energy residual
\begin{equation*}
    R_T(f)=\left(\|a^{1/2}\nabla(u_f(T)-u_T^0)\|_{L^2(\Omega)}^2+ \|(u_f)_t(T)-u_T^1\|_{L^2(\Omega)}^2 \right)^{1/2}.
\end{equation*}
Unless explicitly called independent, $R_T$ is evaluated by a modal forward solve in the indicated Ritz reference model. 

\begin{remark}
For exact eigenpairs $Ae_j=\lambda_j e_j$ and the spectral terminal pairs used above, the adjoint fields have the explicit form
\begin{equation}\label{eq:direct-modal-adjoint-comparator}
\phi_{2j-1}(t,x)=\cos\!\bigl(\sqrt{\lambda_j}(T-t)\bigr)e_j(x),
\qquad
\phi_{2j}(t,x)=-\sin\!\bigl(\sqrt{\lambda_j}(T-t)\bigr)e_j(x).
\end{equation}
Replacing the exact eigenpairs by their Ritz approximations gives the modal reference fields $\phi_{k,R}$. Thus $f_{n,\mathrm{ref}}^H$ is a deterministic comparator using the same retained terminal modes and no neural training. The Ritz modal fields belong to the finite spectral reference model; their strong continuous residual need not vanish because of the eigenpair defect discussed below. The present comparisons assess how closely neural Galerkin adjoints reproduce the associated truncated HUM control. A computational comparison with direct modal evaluation must be made at the same terminal dimension and include eigenpair preprocessing, neural training, assembly, and reconstruction costs. Table~\ref{tab:cost} reports it for Examples~2 and~4, and Section~\ref{sec:discussion} discusses it.
\end{remark}

In Examples~1--4, the accepted directions are combined by the global discrete Galerkin solve. A candidate enrichment has the form \[v_\theta(t,x)=\rho(x)\sum_{\ell=1}^{p} c_\ell\sin(\alpha_\ell t+W_\ell\cdot x+b_\ell),\] where $\rho$ vanishes on the boundary. The three boundary factors are $P(x,y)/P(0,0)$ in Example~1, $\sin x\sin y$ in Example~2, and $(\pi^2-|x|^2)/\pi^2$ in Example~3. During Adam~\cite{kingma2015adam}, the output vector $c$ is recomputed after every hidden-parameter update using the spectrally truncated local Galerkin pseudoinverse described in Section~\ref{sec:method}. The global Galerkin system on the accepted directions is solved by a dense linear solve. The subsequent Levenberg--Marquardt phase \cite{more1978levenberg} jointly optimizes $(c,\alpha,W,b)$; hence its final output is not assumed to solve the local Galerkin system exactly and is accepted or rejected using \eqref{eq:discrete-direction-gates}. At most six directions are accepted for each adjoint mode.  The width starts near $\max\{200,40\sqrt{\lambda_j}\}$ and is capped at 600; initial temporal frequencies concentrate near $\pm\sqrt{\lambda_j}$. Training and validation quadratures are separate. Further implementation details and the example-dependent settings are collected in Supplementary Section~\sref{app:numerical-diagnostics} (Table~\sref{tab:appendix-configurations}). In every run, $\varrho_{\mathrm{acc}}=0.1$ and $\tau_{\mathrm{rank}}=0.2$. The minima over all accepted directions were
\[\begin{array}{c|ccc}
    &\min\widehat\varrho^{\mathrm{tr}}
    &\min\widehat\varrho^{\mathrm{val}}
    &\min\|(I-\Pi^{\mathrm{tr}})\psi^{NN,\mathrm{tr}}\|_{\mathrm{LS},\mathrm{tr}}\\
    \hline
    \text{Example 1}&0.112016&0.112407&0.975515\\
    \text{Example 2}&0.307207&0.307008&0.986307\\
    \text{Example 3}&0.513899&0.463349&0.986252
\end{array}\]
Thus every reported accepted direction passed both correlation gates and the rank gate on the quadratures used by the implementation. The residual vectors use the Ritz terminal targets in \eqref{eq:ritz-ls-functional}. These are discrete diagnostics and do not establish the exact uniform condition \eqref{eq:AD-uniform-quality}. The continuous interval audit of Section~\ref{sec:validated-gate-study} concerns the separate small verification driver, not these production runs.

The homogeneous least-squares bilinear form used by the code is based on the graph-norm variant
\begin{equation}\label{eq:numerical-graph-norm}
 \|v\|_{\mathrm{LS},A}^2=\|\mathcal Pv\|_{L^2(Q)}^2+2\|v(T)\|_{H^1(\Omega)}^2+\|Av(T)\|_{L^2(\Omega)}^2+\|v_t(T)\|_{H^1(\Omega)}^2.
\end{equation}
On a smooth domain, elliptic regularity makes this norm equivalent to the analysis norm. For the square, the smooth even extension of the coefficient and odd reflection of the Dirichlet solution justify the regularity used in the spectral computation. Norm equivalence changes constants, but it does not remove inconsistencies in the terminal operator data.

For a displacement mode, the code applies the differential operator to the neural field but uses the Ritz eigenvalue relation for the terminal target row. In the sign convention of $A=-\nabla\cdot(a\nabla)$, define \[z_{2j-1,R}=\lambda_{j,R}e_{j,R},\qquad z_{2j,R}=0.\]
The functional underlying the numerical residual is
\begin{equation}\label{eq:ritz-ls-functional}
\begin{aligned}
    \mathcal J_{k,R}(v)={}&\tfrac12\|\mathcal Pv\|_{L^2(Q)}^2+\|v(T)-w_{k,R}^0\|_{H^1(\Omega)}^2\\&+\tfrac12\|Av(T)-z_{k,R}\|_{L^2(\Omega)}^2+\tfrac12\|v_t(T)-w_{k,R}^1\|_{H^1(\Omega)}^2.
\end{aligned}
\end{equation}
In the code, the operator row is written with $\nabla\cdot(a\nabla)$ and target $-\lambda_{j,R}e_{j,R}$, which gives the same squared residual. The relative indicator reported in the numerical tables and figures is
\begin{equation}\label{eq:ritz-relative-indicator}
 e_k^{\mathrm{rel}}:=\frac{\sqrt{2\widehat{\mathcal J}_{k,R}^{\,\mathrm{err}}(\widetilde\phi_k^{m,h})}}{\widehat N_{k,R}^{\mathrm{err}}},\qquad\widehat N_{k,R}^{q}:=\sqrt{2\widehat{\mathcal J}_{k,R}^{q}(0)}.
\end{equation}
Here $q$ identifies the training, validation, or error-assessment quadrature. Thus $e_k^{\mathrm{rel}}$ is a normalized Ritz-based least-squares residual, including the normalization used in the stored results; it is not asserted to equal a relative continuous-adjoint error. The validation stopping test is \[\sqrt{2\widehat{\mathcal J}_{k,R}^{\,\mathrm{val}}(\widetilde\phi_k^{m,h})}\leq\epsilon_k\widehat N_{k,R}^{\mathrm{val}},\] subject to the six-step budget and the stagnation criterion. A mode that meets this test has met the implemented residual target; the normalization in that test may differ slightly from the error-assessment normalization. The absolute error $E_{n,m,h}^{\mathrm{adj}}$ in Theorem~\ref{thm:fully-discrete-neural-hum} is a different quantity.

To make this distinction precise, let $\phi_{k,R}^{\mathrm{PDE}}$ be the exact continuous adjoint with terminal data $w_{k,R}$. Define the operator-data consistency defect
\[
 D_{k,R}:=\|Aw_{k,R}^0-z_{k,R}\|_{L^2(\Omega)}
 =\begin{cases}
 \|Ae_{j,R}-\lambda_{j,R}e_{j,R}\|_{L^2(\Omega)},&k=2j-1,\\
 0,&k=2j.
 \end{cases}
\]
Since $\mathcal P\phi_{k,R}^{\mathrm{PDE}}=0$, the graph-norm error satisfies
\[
\begin{aligned}
 \|v-\phi_{k,R}^{\mathrm{PDE}}\|_{\mathrm{LS},A}^2
 ={}&\|\mathcal Pv\|_{L^2(Q)}^2+2\|v(T)-w_{k,R}^0\|_{H^1(\Omega)}^2\\
 &+\|Av(T)-Aw_{k,R}^0\|_{L^2(\Omega)}^2
 +\|v_t(T)-w_{k,R}^1\|_{H^1(\Omega)}^2.
\end{aligned}
\]
Compared with $2\mathcal J_{k,R}(v)$ in \eqref{eq:ritz-ls-functional}, only the terminal operator target differs. The norm of the difference between the two residual vectors is $D_{k,R}$, independently of $v$. The reverse triangle inequality therefore gives, under exact integration,
\begin{equation}\label{eq:ritz-consistency-comparison}
 \left|\|v-\phi_{k,R}^{\mathrm{PDE}}\|_{\mathrm{LS},A}
       -\sqrt{2\mathcal J_{k,R}(v)}\right|\leq D_{k,R}.
\end{equation}
Thus $\|v-\phi_{k,R}^{\mathrm{PDE}}\|_{\mathrm{LS},A}\leq\sqrt{2\mathcal J_{k,R}(v)}+D_{k,R}$, with equality of the two norms when $D_{k,R}=0$. Norm equivalence alone does not remove $D_{k,R}$. Moreover, the modal reference field $\phi_{k,R}(t,x)=s_{k,R}(t)e_{j,R}(x)$ obeys
\[
 \mathcal P\phi_{k,R}=s_{k,R}(t)(Ae_{j,R}-\lambda_{j,R}e_{j,R}),
\]
which is generally nonzero: the weak Ritz eigenvalue relation does not imply a zero strong residual.

A continuous-error bound can be obtained only after accounting for the additional errors. Suppose that $\delta J_{k,R}^q$ bounds $|\mathcal J_{k,R}(v)-\widehat{\mathcal J}_{k,R}^{q}(v)|$, that $\|\cdot\|_{LS}\leq C_{\mathrm{eq}}\|\cdot\|_{\mathrm{LS},A}$, and that $B_{k,R}$ bounds $\|\phi_{k,R}^{\mathrm{PDE}}-\phi_k\|_{LS}$. Then \eqref{eq:ritz-consistency-comparison} and the triangle inequality yield
\begin{equation}\label{eq:ritz-to-continuous-error-bound}
 \|v-\phi_k\|_{LS}
 \leq C_{\mathrm{eq}}\left(
 \sqrt{2\widehat{\mathcal J}_{k,R}^{q}(v)+2\delta J_{k,R}^q}+D_{k,R}\right)+B_{k,R}.
\end{equation}
Here $B_{k,R}$ accounts for the approximation of the continuous terminal basis, and $\delta J_{k,R}^q$ must bound the error of residual integration. To use this expression in \eqref{eq:adjoint-error-bound}, its ingredients require validated upper bounds, including a bound for $D_{k,R}$. Examples~1--4 do not provide these bounds; the separate benchmark in Section~\ref{sec:certified-benchmark} uses exact terminal modes and analytic ball integration. Their reported residuals are the indicators actually used to train, accept, and stop the networks. The additional quadrature diagnostics in Supplementary Section~\sref{app:ritz-consistency-diagnostics} illustrate the possible size of the terminal-target defect without certifying a continuous-error bound.

For the finite-reference comparisons, $\Phi_{n,R}$ denotes synthesis from the first $n$ modal reference fields, and $\widetilde\Phi_n$ denotes synthesis from their computed neural fields, including pruning where specified. Observation Gram assembly uses a finite Legendre expansion in time, with temporal coefficients computed by Gauss quadrature. The assembly and reference levels retain respectively $100$ and $120$ Legendre components and use respectively $200$ and $320$ time nodes. The intermediate reference-level coefficient array is stored in \texttt{float32} before conversion to \texttt{float64} for the Gram products; the assembly-level array is stored in \texttt{float64}. The two levels also use different spatial quadratures. These details, and the separate evaluation of cross moments, are given in Supplementary Section~\sref{app:assembly-diagnostics}.

We denote the measured differences between these complete assembly procedures by $\Delta_G^{\mathrm{asm}}$ and $\Delta_b^{\mathrm{asm}}$. They are assembly sensitivities, distinct from the exact-integral errors $\delta_{G,n,m,h}$ and $\delta_{b,n,m,h}$ in Theorem~\ref{thm:fully-discrete-neural-hum}. In particular, $\Delta_G^{\mathrm{asm}}$ does not isolate quadrature error: it also reflects changes in time-expansion truncation and intermediate precision. Agreement of the two procedures does not bound either procedure's error relative to exact integration.

Table~\ref{tab:notation} collects the error quantities used in the analysis and in the experiments.
\begin{table}[!htbp]\centering\small
\caption{Error quantities. The first group refers to the continuous problem and the analysis of Section~\ref{sec:method}; the second to quantities measured in the experiments.}
\label{tab:notation}
\begin{tabular}{@{}lp{.62\textwidth}l@{}}\toprule
Symbol & Meaning & Defined in\\\midrule
$\tau_n$ & distance from $f^*$ to the exact $n$-mode HUM reconstruction & \eqref{eq:truncation-error-definition}\\
$\varepsilon_{n,m}$ & largest $\|\cdot\|_{LS}$ error of the exact Galerkin adjoints after $m$ accepted steps & \eqref{eq:adjoint-uniform-error}\\
$\chi_{n,m,h}$ & distance of the computed adjoints from the exact Galerkin adjoints & Sec.~\ref{sec:discrete-adjoint-error}\\
$E^{\mathrm{adj}}_{n,m,h}$ & any bound for the $\|\cdot\|_{LS}$ error of the computed adjoints & \eqref{eq:adjoint-error-bound}\\
$E_k$, $P_k$ & validated bounds for the error and interior residual of the $k$th field & Cor.~\ref{cor:wave-residual-certificate}\\
$\delta_{G,n,m,h}$, $\delta_{b,n,m,h}$ & Gram and moment errors relative to exact integrals, including solve residuals & Sec.~\ref{sec:reconstruction-perturbation}\\\midrule
$\tau_n^{\mathrm{ref}}$ & truncation error measured in a finite reference model & Sec.~\ref{sec:numerics}\\
$\eta_n$ & distance of the neural control from the $n$-mode reference control & Sec.~\ref{sec:numerics}\\
$e_k^{\mathrm{rel}}$ & relative Ritz-based least-squares residual (training indicator) & \eqref{eq:ritz-relative-indicator}\\
$D_{k,R}$ & terminal operator-data consistency defect of a Ritz pair & Sec.~\ref{sec:numerics}\\
$\Delta_G^{\mathrm{asm}}$, $\Delta_b^{\mathrm{asm}}$ & measured differences between two complete assembly procedures & Sec.~\ref{sec:numerics}\\
$R_T(f)$ & terminal energy residual of the state driven by $f$ & Sec.~\ref{sec:numerics}\\\bottomrule
\end{tabular}
\end{table}

\subsection{Verification benchmark: a certified continuous control}
\label{sec:certified-benchmark}

To test residual-to-control propagation without terminal-space truncation, take
$\Omega=\omega=(0,\pi)$, $a=1$, $T=2\pi$, and zero initial data. Let
$e_j(x)=\sqrt{2/\pi}\sin(jx)$. For the terminal pairs $(e_j,0)$ and $(0,je_j)$,
the adjoints are $\phi_{2j-1}=e_j\cos(jt)$ and $\phi_{2j}=e_j\sin(jt)$.
Space and time orthogonality give $G_n=\pi I_n$ for every $n$, so the
continuous observation constants are $\alpha=\beta=\sqrt\pi$.
Prescribe, before training,
\[
 c^*=(2/5,-1/4,3/20,1/10)^\top,\qquad
 f^*=\sum_{k=1}^4c_k^*\phi_k.
\]
The target generated by this control is
\[
 u_T^0=\sum_{j=1}^2\frac{\pi c_{2j}^*}{j}e_j,
 \qquad u_T^1=-\pi\sum_{j=1}^2c_{2j-1}^*e_j.
\]
These formulas follow directly from Duhamel's formula. Since $f^*$ is an
adjoint observation, Proposition~\ref{prop:minimum-norm-exact-control} makes it
the minimum-norm control for these data. At $n=4$, $\tau_4=0$ exactly.

The trial functions are $\sin x\sum_{\ell=1}^p c_\ell
\sin(\alpha_\ell t+w_\ell x+b_\ell)$, with four neurons per enrichment.
For this small verification calculation, each candidate is refined by SciPy's
trust-region least-squares solver with \texttt{max\_nfev}=24 (its evaluation
counter excludes numerical-Jacobian calls), followed by the global linear
Galerkin update. Training uses $40\times32$ Gauss nodes; reconstruction uses
$64\times48$. Correlation and rank gates are $0.1$ and $0.2$.
Modes with relative residual below $10^{-10}$ are frozen, and common budgets
$m=1,2,3$ are compared. The optimizer and widths differ from those of
Examples~1--4; the certificate uses only the returned fields and does not
assume exact optimizer convergence.

This benchmark is deliberately favourable to the trial class. Since
$\sin(jx)=\sin x\,U_{j-1}(\cos x)$ with the Chebyshev polynomial $U_{j-1}$,
product-to-sum formulas write every exact adjoint with $j\le2$ as a combination
of at most two neurons $\sin x\,\sin(jt\pm(j-1)x+b)$; for instance
$\sin(2x)\cos(2t)=\sin x\,[\cos(2t-x)+\cos(2t+x)]$. The initialization draws
temporal frequencies in $\pm j[0.8,1.2]$ and spatial frequencies in
$\pm(j-1)+[-0.4,0.4]$, that is, near these exact parameters, and in several
later enrichment steps the trust-region solver stops after a single function
evaluation. The benchmark therefore tests the propagation of adjoint errors
into the control and the validity of the certificates for the fields actually
returned; it does not test the approximation power of the network class.
Section~\ref{sec:truncated-benchmark} removes these simplifications.

All certificate integrals are analytic integrals of finite exponential sums,
enclosed using Arb/Acb ball arithmetic at $192$ bits~\cite{Johansson2017arb}.
The saved binary64 coefficients are interpreted as exact dyadic rationals.
This accounts for integration and arithmetic errors in the evaluated fields;
the matrix and solve errors are bounded separately. Supplementary Section~\sref{app:certificate}
specifies the enclosure and the constants in
Corollary~\ref{cor:wave-residual-certificate}.
\begin{table}[!htbp]\centering\small
\caption{Continuous verification at $n=4$, with $\tau_4=0$.
The residual and control-bound columns are rounded upward to four significant
digits from validated upper bounds. The direct error is the midpoint of a
separate, analytically integrated ball enclosure. $B_m$ is the bound
\eqref{eq:wave-residual-certificate}, including numerical assembly and solve
errors. Every stability test and every comparison of the bound with the direct
error is verified by ball inequalities. $B_m^{\rm gen}$ is the bound of
Theorem~\ref{thm:fully-discrete-neural-hum}, evaluated using the same validated
inputs and its explicit constant.}
\label{tab:certificate}
\begin{tabular}{@{}rrrrrr@{}}\toprule
$m$ & $\max_k E_k$ & direct error & $B_m$ & $B_m^{\rm gen}$ & $B_m$/error\\
\midrule
1 & $9.454\cdot10^{-3}$ & $1.726\cdot10^{-3}$ & $1.592\cdot10^{-2}$ & $8.588\cdot10^{-1}$ & $9.2$\\
2 & $2.141\cdot10^{-4}$ & $1.386\cdot10^{-5}$ & $4.564\cdot10^{-4}$ & $1.945\cdot10^{-2}$ & $32.9$\\
3 & $1.139\cdot10^{-4}$ & $6.780\cdot10^{-6}$ & $2.341\cdot10^{-4}$ & $1.035\cdot10^{-2}$ & $34.5$\\
\bottomrule

\end{tabular}
\end{table}

Both bounds cover the actual error at all three budgets (Table~\ref{tab:certificate} and Figure~\sref{fig:certificate}). The wave-residual
bound is smaller in this experiment; it uses the known finite HUM coefficients
in \eqref{eq:manufactured-distance-bound} and the interior residual separately.
At $m=3$, the respective upper bounds are $2.341\cdot10^{-4}$ and
$1.035\cdot10^{-2}$ after upward rounding, a factor of about $44$.
The comparison uses the known coefficients and does not assert a uniform
improvement for arbitrary data.
The residual bound decreases as the
adjoints improve. Its overestimation, by factors $9$--$35$, comes from the
conservative state-energy bound $U$ and the embedding constant $C_X$; the
generic bound of Theorem~\ref{thm:fully-discrete-neural-hum} overestimates by
factors of about $500$--$1500$. Since $\tau_4=0$, this experiment isolates
neural and reconstruction errors.

\subsubsection{Localized observation with explicit continuous constants}
\label{sec:localized-certificate}
We next retain $\Omega=(0,\pi)$, $a=1$, $T=2\pi$, and the four terminal modes,
but restrict the observation and control to $\omega=(\pi/3,2\pi/3)$.
Define $M_j=\int_\omega e_j^2\,dx$. Exact integration gives
\[
 M_j=\frac13-\frac{\sin(4j\pi/3)-\sin(2j\pi/3)}{2\pi j}.
\]
Time orthogonality still makes the observation Gram diagonal, with repeated
entries $\pi M_j$. For $j$ congruent to $0,1,2$ modulo $3$, respectively,
$M_j$ equals $1/3$, $1/3+\sqrt3/(2\pi j)$, and
$1/3-\sqrt3/(2\pi j)$. Consequently the sharp constants for the entire
continuous terminal space, rather than only the retained four modes, are
\[
 \alpha^2=\frac\pi3-\frac{\sqrt3}{4},\qquad
 \beta^2=\frac\pi3+\frac{\sqrt3}{2}.
\]
Indeed the minimum occurs at $j=2$ and the maximum at $j=1$.
Prescribe the same coefficients $c^*$ before reconstruction and set
$f^*=(\sum_{k=1}^4c_k^*\phi_k)|_{Q_\omega}$. The target must be changed:
Duhamel's formula and time orthogonality yield
\[
 u_T^0=\pi\sum_{j=1}^2\frac{M_jc_{2j}^*}{j}e_j,
 \qquad u_T^1=-\pi\sum_{j=1}^2M_jc_{2j-1}^*e_j.
\]
The observation of this homogeneous adjoint is the minimum-norm control for
these targets and $\tau_4=0$. Although the forcing is localized in space, all
other terminal Fourier coefficients vanish by the full-period time integrals.
The intermediate controlled state need not have only two spatial modes.

The adjoint problems are independent of $\omega$, so the networks of
Section~\ref{sec:certified-benchmark} are reused without retraining. Endpoint moments are integrated on
$\Omega$, while observation Gram and control-error integrals are integrated
on $\omega$. All are enclosed analytically at $192$ bits. The force norm is
$F^2=\pi\sum_{j=1}^2M_j(|c_{2j-1}^*|^2+|c_{2j}^*|^2)$; the state bound
remains $U=\sqrt2\pi F$. In the expression for $C_L$, each summand acquires
the factor $M_j^2$. The bound $D$ is still
\eqref{eq:manufactured-distance-bound}. This supplies each ingredient of both
continuous certificates in a localized-control problem.
\FloatBarrier

\FloatBarrier
At the third budget (Table~\sref{tab:localized-certificate} and Figure~\sref{fig:localized-certificate}), the direct error is about $6.855\cdot10^{-6}$ and the
synthesis-based certificate is below the upward-rounded value
$3.768\cdot10^{-4}$. The localization changes the observation constants and
the moments, while leaving the trained adjoints unchanged. The exact modes
and the diagonal Gram matrix still make this test far more tractable than
the variable-coefficient geometries.

\FloatBarrier

\subsubsection{Continuous acceptance and projection audit over three seeds}
\label{sec:validated-gate-study}
To quantify the gap between sampled and continuous acceptance, we repeat the
small benchmark with the three prespecified seeds $1729,1730,1731$, retaining
the original four-neuron width, three-step cap and $24$ function-evaluation
budget per trust-region call. No seed is selected or discarded according to
its accuracy. For this study the driver saves each raw normalized direction
before the floating global coefficient update. Its binary64 parameters define an exact
dyadic trial field and a nested sequence of spaces.

Arb-enclosed least-squares inner products and verified small Gram solves give
the exact continuous Galerkin projections in these spaces. They supply
continuous residual-capture and rank lower bounds and the distance $\chi$ of
each returned floating field from its exact projection, as in
Proposition~\ref{prop:validated-acceptance}. The rank is evaluated from the
Gram Schur complement. Frozen modes generate no further acceptance tests.
Across the $24$ accepted directions, every continuous correlation exceeds
$0.4678$ and every normalized rank residual exceeds $0.5946$; the largest
validated projection defect is below $6.256\cdot10^{-15}$.
These are finite-run verifications, not a uniform-in-mode optimizer theorem.
\FloatBarrier
\begin{table}[!htbp]\centering\small
\caption{Three prespecified seeds at budget $m=3$ with localized observation.
$M$ counts all accepted directions; $\underline\varrho$ is a downward-rounded
minimum continuous capture bound. $B_{\rm Gram}$ is rounded upward.
The final column tests $C_X(\sum_kE_k^2)^{1/2}<\alpha$.
All actual-Gram certificates pass positive-definiteness, assembly and coverage
tests, including the seed that fails the sufficient synthesis test.}
\label{tab:validated-gates}
\begin{tabular}{@{}rrrrrl@{}}\toprule
seed & $M$ & $\underline\varrho$ & direct error & $B_{\rm Gram}$ & synthesis test\\
\midrule
1729 & 8 & $0.7354$ & $6.855\cdot10^{-6}$ & $4.548\cdot10^{-4}$ & yes\\
1730 & 6 & $0.5679$ & $1.303\cdot10^{-1}$ & $3.066\cdot10^{0}$ & no\\
1731 & 10 & $0.4678$ & $9.230\cdot10^{-6}$ & $8.960\cdot10^{-4}$ & yes\\
\bottomrule

\end{tabular}
\end{table}
\FloatBarrier
In Table~\ref{tab:validated-gates}, the seed $1730$ is substantially less accurate at the fixed budget despite
passing the continuous direction tests. Its synthesis-smallness test fails,
whereas the directly validated continuous Gram remains positive definite and
Corollary~\ref{cor:actual-gram-certificate} still supplies an upper bound.
Its bound exceeds a control tolerance of $10^{-3}$, while the other two
actual-Gram bounds are below it. A certificate-based stopping decision would
therefore request further refinement for that seed. Accepted directions and
a fixed step count alone do not enforce a requested control accuracy.

\subsubsection{An exploratory certificate-driven stopping test}
\label{sec:adaptive-stopping}
We now use the upper endpoint of the actual-Gram certificate as a stopping
criterion, with requested absolute control tolerance $10^{-3}$. The
separately evaluated direct error is recorded for verification and never
used to choose a stopping step. Starting from the saved prefixes of the
three-seed study, the controller evaluates the continuous gates, projection
defects and control certificate at each budget. A failed tolerance test
requests another enrichment. For budgets beyond three, new directions use
width eight and at most $64$ function evaluations per trust-region call;
the already accepted directions and their history are retained.

The initial refinement cap is eight. Seed $1730$ still fails the requested
bound there. We therefore added one deterministic fresh-space fallback,
with seed $1730+100003$, width $16$, at most $128$ function evaluations per
call, and a cap of three budgets. This fallback was chosen after inspecting
the initial failed run; it is an exploratory, reproducible repair rather
than a prespecified robustness result. A restart replaces the trial space,
so nesting and contraction are asserted only within each phase. All eight
failed tolerance tests and their fields are retained. If the final cap
also failed, the controller would report an unmet tolerance.
\FloatBarrier

\FloatBarrier
The returned bounds meet the requested tolerance for all three recorded
seeds (Table~\sref{tab:adaptive-stopping} and Figure~\sref{fig:adaptive-stopping}). The initial-phase failure remains part of that conclusion: seed
$1730$ reaches direct error $1.289\cdot10^{-3}$ and bound
$9.820\cdot10^{-2}$ at budget eight, before the fresh-space fallback.
For seed $1731$, the direct error is already below tolerance at budget two,
but the bound still exceeds it and the controller correctly continues.
The experiment shows that the actual-Gram certificate can drive refinement
decisions in the explicit benchmark. Because the fallback was fixed after the
first failure and only three seeds were run, it is not evidence of robust
termination; a prespecified restart policy tested on many seeds would be
needed for that.

\FloatBarrier

\subsubsection{A benchmark with truncation, a non-diagonal Gram matrix and adjoints outside the trial class}
\label{sec:truncated-benchmark}
\newcommand{\TBvalidated}{1}
\newcommand{\TBtaufirst}{\ensuremath{6.901\cdot10^{-2}}}
\newcommand{\TBtaulast}{\ensuremath{1.171\cdot10^{-2}}}
\newcommand{\TBtaubarfirst}{\ensuremath{7.485\cdot10^{-2}}}
\newcommand{\TBtaubarlast}{\ensuremath{1.259\cdot10^{-2}}}
\newcommand{\TBtaubarratio}{1.08}
\newcommand{\TBfnorm}{\ensuremath{2.913\cdot10^{-1}}}
\newcommand{\TBFdata}{\ensuremath{3.264\cdot10^{-1}}}
\newcommand{\TBerrfirst}{\ensuremath{6.901\cdot10^{-2}}}
\newcommand{\TBerrlast}{\ensuremath{1.171\cdot10^{-2}}}
\newcommand{\TBgramlast}{\ensuremath{7.616\cdot10^{-2}}}
\newcommand{\TBratiomin}{1.08}
\newcommand{\TBratiomax}{6.50}
\newcommand{\TBmaxElast}{\ensuremath{2.527\cdot10^{-2}}}
\newcommand{\TBneurons}{328}
\newcommand{\TBtrainmax}{\ensuremath{5.781\cdot10^{-4}}}
\newcommand{\TBtrainmin}{\ensuremath{4.107\cdot10^{-7}}}
\newcommand{\TBmaxsteps}{4}
\newcommand{\TBonestep}{14}
\newcommand{\TBsynpass}{4}
\newcommand{\TBgenpass}{4}
\newcommand{\TBgrampass}{4}
\newcommand{\TBrows}{4}
\newcommand{\TBgramratiomax}{6.50}
\newcommand{\TBsynratiomin}{1.08}
\newcommand{\TBsynratiomax}{5.14}
\newcommand{\TBgenratiomax}{197}
\newcommand{\TBsmallexcess}{19}
\newcommand{\TBtaubarexcess}{9}
\newcommand{\TBEsmall}{\ensuremath{5.688\cdot10^{-5}}}
\newcommand{\TBEbig}{\ensuremath{2.527\cdot10^{-2}}}
\newcommand{\TBgammafirst}{1.94}
\newcommand{\TBgammalast}{0.98}
\newcommand{\TBrelagree}{\ensuremath{2\cdot10^{-6}}}
\newcommand{\TBnaccepted}{22}
\newcommand{\TBtraincorrmin}{0.215}
\newcommand{\TBvalcorrmin}{0.215}

The benchmarks above are favourable in four respects: four modes, zero truncation, a diagonal Gram
matrix, and exact adjoints in the span of at most two neurons, with an initialization close to the
exact parameters. The following benchmark removes all four. Take $\Omega=\omega=(0,\pi)$, $a=1$,
zero initial data and $T=5\pi/2$. Since $T$ is not a multiple of $2\pi$, the Gram matrix of the exact
modes on the whole terminal space is block diagonal with the $2\times2$ blocks of
$\cos(j(t-T))e_j$ and $\sin(j(t-T))e_j$, whose eigenvalues are $T/2\pm|\sin(jT)|/(2j)$; hence
\[
 \alpha^2=\frac{5\pi}{4}-\frac12,\qquad \beta^2=\frac{5\pi}{4}+\frac12 .
\]
This block structure concerns the exact adjoints; the Gram matrix of the trained fields, which enters the
certificates, is a full $n\times n$ matrix.
Prescribe $f^*=\sum_{k=1}^{24}c_k^*\phi_k$ with $c_k^*=(-1)^{k+1}/(k+2)^2$; by
Proposition~\ref{prop:minimum-norm-exact-control} it is the minimum-norm control for the target it
generates from rest, which Duhamel's formula gives in closed form. Adjoints are trained for the first
$16$ terminal modes ($j\le8$), and the reconstruction uses the first $n\in\{4,8,12,16\}$ of them, so
that the truncation error $\tau_n=\|\sum_{k>n}c_k^*\phi_k\|$ is nonzero for every $n$, from
$\TBtaufirst$ at $n=4$ to $\TBtaulast$ at $n=16$.

The trial fields are $\rho(x)\sum_\ell c_\ell\sin(\alpha_\ell t+w_\ell x+b_\ell)$ with the polynomial
boundary factor $\rho(x)=x(\pi-x)/\pi^2$, as in Examples~1 and~3. Since $\sin(jx)/\rho(x)$ is not a
trigonometric polynomial, no finite network of this form represents an exact adjoint. For spatial index $j$ an
enrichment has width $\max(10,2j+4)$. This rule is a heuristic. Because $\sin(j\pi)=0$, the quotient
$\sin(jx)/\rho(x)$ extends to an entire function of exponential type $j$ that is square integrable on $\mathbb R$,
so by the Paley--Wiener theorem~\cite[Chapter~19]{rudin1987real} its Fourier transform is supported in
$[-j,j]$; a standing wave also needs both temporal signs. The certificates of the returned fields do not
depend on this choice. Each enrichment is fitted by the trust-region solver from $16$
random starts, the two best after $40$ evaluations are continued to at most $300$ evaluations, the
continued start with the largest training correlation is accepted if it passes the gates
\eqref{eq:discrete-direction-gates} with $\varrho_{\mathrm{acc}}=0.1$ and $\tau_{\mathrm{rank}}=0.2$, and at
most four directions are accepted per mode (relative tolerance $10^{-4}$). The training correlation and
the rank test use the Gauss rule of $96\times64$ nodes on which the networks are fitted, and the
validation correlation an independent rule of $128\times96$ nodes. As in the production runs, temporal frequencies are initialized in $\pm j[0.8,1.2]$;
spatial frequencies are drawn uniformly in $[-(j+2),j+2]$ and phases in $[-\pi,\pi]$, so no start is
centred at an exact representation, since none exists. All certificate integrals are integrals of polynomials times exponentials;
they are enclosed with Arb at $192$ bits by an integration-by-parts recursion or, for small
frequencies, by a power series with a rigorous remainder (Supplementary Section~\sref{app:truncated}).

No certificate input uses $f^*$. Following Remark~\ref{rem:unknown-control-inputs}, the norm bound
$F=R_T(0)/\alpha=\TBFdata$ and the state bound $U=TF/\sqrt2$ are computed from the target data, and
Corollary~\ref{cor:wave-residual-certificate} gives $D=\bar\tau_n+C_X(\sum_{k\le n}E_k^2)^{1/2}F/\alpha$.
Here the truncation error is replaced by a bound computed from the data as well: the exact Gram matrix
is block diagonal in $j$, so the tail $\sum_{k>n}c_k^*\phi_k$ is the minimum-norm control of the tail of
the target, and
\[
 \tau_n\le\bar\tau_n:=\frac{1}{\alpha}\Bigl(\sum_{j>n/2}\bigl(j^2|\langle u_T^0,e_j\rangle|^2+|\langle u_T^1,e_j\rangle|^2\bigr)\Bigr)^{1/2}.
\]
This step uses the spectral structure of the example; for a variable coefficient or a partial
observation region the coupling between retained and discarded modes would also have to be bounded.
The exact $\tau_n$ and the direct error are evaluated separately, for comparison only.
\FloatBarrier
\begin{table}[!htbp]\centering\footnotesize
\caption{Benchmark with truncation, a non-diagonal Gram matrix and adjoints outside the trial class: $n$ of
the $24$ modes of $f^*$ are reconstructed. ``neurons'' is the total number of neurons of the $n$ trained
fields, $\bar\tau_n$ the data bound of $\tau_n$, $\max_kE_k$ the largest validated adjoint error in $\|\cdot\|_{LS}$, $B_{\rm Gram}$ the bound
\eqref{eq:actual-gram-certificate}, $B_{\rm syn}$ the bound \eqref{eq:wave-residual-certificate} and
$B^{\rm gen}$ the bound of Theorem~\ref{thm:fully-discrete-neural-hum}; ``fails'' means that the
corresponding sufficient condition is not verified. Bounds are rounded upward.}
\label{tab:truncated-benchmark}
\setlength{\tabcolsep}{3pt}
\begin{tabular}{@{}rrrrrrrrrr@{}}\toprule
$n$ & neurons & $\tau_n$ & $\bar\tau_n$ & $\max_kE_k$ & direct error & $B_{\rm Gram}$ & $B_{\rm syn}$ & $B^{\rm gen}$ & $B_{\rm Gram}$/error\\
\midrule
4 & 40 & $6.90\cdot10^{-2}$ & $7.49\cdot10^{-2}$ & $5.7\cdot10^{-5}$ & $6.90\cdot10^{-2}$ & $7.50\cdot10^{-2}$ & $7.50\cdot10^{-2}$ & $7.75\cdot10^{-2}$ & $1.09$\\
8 & 84 & $3.27\cdot10^{-2}$ & $3.53\cdot10^{-2}$ & $5.7\cdot10^{-5}$ & $3.27\cdot10^{-2}$ & $3.54\cdot10^{-2}$ & $3.54\cdot10^{-2}$ & $3.90\cdot10^{-2}$ & $1.08$\\
12 & 192 & $1.90\cdot10^{-2}$ & $2.05\cdot10^{-2}$ & $2.2\cdot10^{-2}$ & $1.90\cdot10^{-2}$ & $5.60\cdot10^{-2}$ & $4.89\cdot10^{-2}$ & $1.71\cdot10^{0}$ & $2.94$\\
16 & 328 & $1.17\cdot10^{-2}$ & $1.26\cdot10^{-2}$ & $2.6\cdot10^{-2}$ & $1.17\cdot10^{-2}$ & $7.62\cdot10^{-2}$ & $6.02\cdot10^{-2}$ & $2.31\cdot10^{0}$ & $6.50$\\
\bottomrule

\end{tabular}
\end{table}
\FloatBarrier
Fourteen of the sixteen modes reach the tolerance with one accepted direction; the modes $k=11$
($j=6$) and $k=16$ ($j=8$) stop at the budget of four directions, with training residuals
$5.8\cdot10^{-4}$ and $3.9\cdot10^{-4}$. All \TBnaccepted{} accepted directions pass both correlation gates;
the smallest training and validation correlations are $\TBtraincorrmin$ and $\TBvalcorrmin$.
Table~\ref{tab:truncated-benchmark} and Supplementary Figure~\sref{fig:truncated-benchmark} show two regimes. For
$n\le8$ all validated adjoint errors are at most \TBEsmall, the direct error agrees with $\tau_n$, and the
three bounds exceed it by at most $\TBsmallexcess\%$; most of the excess comes from replacing $\tau_n$ by
$\bar\tau_n$ (at most $\TBtaubarexcess\%$). For $n=12$ and $16$ the two modes that missed the
tolerance enter, with $E_k$ up to \TBEbig. The direct error still agrees with $\tau_n$ to a relative
\TBrelagree, so these adjoint errors hardly change the control. The term
$C_X(\sum_kE_k^2)^{1/2}F/\alpha$ in $D$ and the moment part of the certificates are worst-case bounds
and do not see this (Table~\ref{tab:error-decomposition}): the actual-Gram and synthesis certificates exceed the error by factors up to $\TBgramratiomax$ and $\TBsynratiomax$, and the
generic bound of Theorem~\ref{thm:fully-discrete-neural-hum}, whose stability condition still holds, by
up to $\TBgenratiomax$. The lower bound $\gamma$ for the actual Gram matrix decreases from
$\TBgammafirst$ at $n=4$ to $\TBgammalast$ at $n=16$, and positive definiteness is verified in every
case. All bounds thus remain valid when the terminal space is truncated, the Gram matrix is not diagonal and no
input is taken from the solution. The residual-based certificates stay within a small factor of the
error, but they are driven by the worst trained adjoint rather than by its effect on the control, so a
single poorly trained mode makes them pessimistic.

\paragraph{The two terms of the identity.}
Table~\ref{tab:error-decomposition} evaluates, for the stored fields of both benchmarks, the two terms
$D_{\rm true}=\operatorname{dist}(f^*,V)$ and $R_{\rm true}=(r^\top G^{-1}r)^{1/2}$ of
\eqref{eq:wave-control-pythagoras}, together with the corresponding parts of the certificates. These are
float64 Gauss--Legendre diagnostics, not enclosures. Two rules with $160\times128$ and $256\times192$
nodes agree to $2.4\cdot10^{-10}$ relative in Section~\ref{sec:certified-benchmark} and to
$1.4\cdot10^{-8}$ here, the largest difference occurring for $R_{\rm true}\approx10^{-8}$; the two terms
reproduce $\|f_V-f^*\|^2$ to $2\cdot10^{-11}$ relative, and the returned coefficients reproduce the enclosed
direct errors to $3\cdot10^{-12}$. Without truncation the moment-inconsistency term dominates: it carries
$98\%$ of the squared error at $m=1$ and $61\%$ at $m=3$, and $85$--$99\%$ with observation on the middle
third. In these stored fields the moment term is therefore substantial, and the reconstructed control is
not well described by the projection of $f^*$ onto the span of the computed adjoints. In the truncated benchmark the distance term dominates and
$R_{\rm true}\leq1.5\cdot10^{-5}$, so the ratios near one at $n=4,8$ reflect a sharp tail bound
$\bar\tau_n$ rather than a sharp estimate of the moment term. The table also locates the pessimism of the
certificates. In the rows of Section~\ref{sec:certified-benchmark}, the distance bound $D$ exceeds $D_{\rm true}$
by factors of $21$--$28$ and the moment part exceeds $R_{\rm true}$ by factors of $8$--$41$; here the
moment part exceeds $R_{\rm true}$ by more than a factor of $2000$,
because $|r_k|\le\|u^*\|_{L^2(Q)}\|\mathcal Pv_k\|_{L^2(Q)}$ is evaluated with a data bound $U$ for
$\|u^*\|$; at $n=12,16$ the two adjoints that missed their tolerance enlarge both parts. The assembly and
solve contributions are below $10^{-12}$.
\FloatBarrier
\begin{table}[!htbp]\centering\small
\caption{Error decomposition for the stored fields of the verification benchmarks.
$D_{\rm true}=\operatorname{dist}(f^*,V)$ and $R_{\rm true}=(r^\top G^{-1}r)^{1/2}$ are the two terms of
\eqref{eq:wave-control-pythagoras}, evaluated by float64 Gauss quadrature (diagnostics, not
enclosures). $D$ and the moment bound are the corresponding parts of the reported certificate:
$D$ from \eqref{eq:manufactured-distance-bound} and $UP/s$ in \eqref{eq:wave-residual-certificate} for
Section~\ref{sec:certified-benchmark}, and $D$ and $UP/\gamma^{1/2}$ in \eqref{eq:actual-gram-certificate}
for Section~\ref{sec:truncated-benchmark}. The last two columns are the enclosed error of the returned
control and the certificate, as in Tables~\ref{tab:certificate} and~\ref{tab:truncated-benchmark}. Bounds
are rounded upward.}
\label{tab:error-decomposition}
\setlength{\tabcolsep}{4pt}
\begin{tabular}{@{}lrrrrrr@{}}\toprule
& $D_{\rm true}$ & $D$ & $R_{\rm true}$ & moment bound & $\|\widehat f-f^*\|$ & bound\\
\midrule
\multicolumn{7}{@{}l}{\emph{Section~\ref{sec:certified-benchmark}: $n=4$, $\tau_4=0$, whole observation, Corollary~\ref{cor:wave-residual-certificate}}}\\
$m=1$ & $2.60\cdot10^{-4}$ & $7.24\cdot10^{-3}$ & $1.71\cdot10^{-3}$ & $1.42\cdot10^{-2}$ & $1.726\cdot10^{-3}$ & $1.592\cdot10^{-2}$\\
$m=2$ & $6.47\cdot10^{-6}$ & $1.65\cdot10^{-4}$ & $1.23\cdot10^{-5}$ & $4.26\cdot10^{-4}$ & $1.386\cdot10^{-5}$ & $4.564\cdot10^{-4}$\\
$m=3$ & $4.25\cdot10^{-6}$ & $8.74\cdot10^{-5}$ & $5.28\cdot10^{-6}$ & $2.18\cdot10^{-4}$ & $6.780\cdot10^{-6}$ & $2.341\cdot10^{-4}$\\
\midrule
\multicolumn{7}{@{}l}{\emph{Section~\ref{sec:truncated-benchmark}: $\tau_n>0$, Corollary~\ref{cor:actual-gram-certificate}}}\\
$n=4$ & $6.90\cdot10^{-2}$ & $7.50\cdot10^{-2}$ & $5.58\cdot10^{-9}$ & $8.03\cdot10^{-5}$ & $6.901\cdot10^{-2}$ & $7.492\cdot10^{-2}$\\
$n=8$ & $3.27\cdot10^{-2}$ & $3.54\cdot10^{-2}$ & $1.95\cdot10^{-8}$ & $1.31\cdot10^{-4}$ & $3.272\cdot10^{-2}$ & $3.538\cdot10^{-2}$\\
$n=12$ & $1.90\cdot10^{-2}$ & $4.40\cdot10^{-2}$ & $1.46\cdot10^{-5}$ & $3.48\cdot10^{-2}$ & $1.902\cdot10^{-2}$ & $5.599\cdot10^{-2}$\\
$n=16$ & $1.17\cdot10^{-2}$ & $4.90\cdot10^{-2}$ & $1.46\cdot10^{-5}$ & $5.84\cdot10^{-2}$ & $1.171\cdot10^{-2}$ & $7.616\cdot10^{-2}$\\
\bottomrule

\end{tabular}
\end{table}
\FloatBarrier

\subsection{Example 1: a nonconvex Cassini oval}

Let $a(x)\equiv1$, $T=3.5\pi$, and \[\Omega=\{(x,y):P(x,y)>0\},\qquad P=(2.3^4-2^4)-(x^2+y^2)^2+8(x^2-y^2).\] This analytic domain has two lobes joined by a concave waist. The connected C-shaped control region is a boundary band $\omega=\{P/|\nabla P|<\delta(x,y)\}\cap\Omega$, where \[\delta(x,y)=0.38\left(1+\frac{0.45x}{\sqrt{x^2+y^2}}\right)\chi(x,y),\quad\chi(x,y)=1-s(-y/0.3)\left[1-s\bigl((|x|-0.55)/0.4\bigr)\right].\] The smooth step $s$ is zero on $(-\infty,0]$, one on $[1,\infty)$, and $s(z)=e^{-1/z}/(e^{-1/z}+e^{-1/(1-z)})$ for $0<z<1$. At points where $\nabla P=0$ the ratio $P/|\nabla P|$ is set to $+\infty$; set $\delta(0,0)=0$.  These conventions affect no point of the boundary band.  The lower central waist is uncovered. Both initial and target states have nonzero displacement and velocity.  They are smooth spectral bumps, with $g_p=\sum_j e^{-0.5\lambda_j}e_j(p)e_j$, $u^0=g_{(-1.9,0.3)}$ and $u_T^0=-g_{(1.9,-0.3)}$; the velocity coefficients are respectively $0.5\sqrt{\lambda_j}e^{-0.5\lambda_j}e_j(p)$ and their negatives.

The reference uses a Legendre Rayleigh--Ritz basis of degrees $(88,48)$, retaining 311 eigenpairs with $\lambda\leq300$, or 622 terminal modes.  The relative residual tolerance is $\epsilon_k=10^{-3}$ for $k\leq16$ and $1.5\cdot10^{-3}$ thereafter. Figure~\sref{fig:peanut-geometry} displays the region and a slow ray found numerically.  One million random rays and a local search found entry times below $1.97\pi<T$; the smallest observed eigenvalue of the 622-mode observation Gram matrix is about $0.0726$. Neither finite ray sampling nor this finite spectrum proves GCC or a uniform observability inequality.

\begin{table}[!htbp]\centering
\small
\caption{Example 1: control errors and terminal residual in the $\lambda\leq300$ reference model.  The maximum Ritz-based least-squares residual is normalized as in \eqref{eq:ritz-relative-indicator}.}
\label{tab:peanut}
\setlength{\tabcolsep}{4pt}
\begin{tabular}{@{}rrrrrrr@{}}
    \toprule
    \multicolumn{1}{@{}c}{$n$} &
    \multicolumn{1}{c}{$\max e_k^{\mathrm{rel}}$} &
    \multicolumn{1}{c}{$\|f_n^m-f_{\rm ref}\|$} &
    \multicolumn{1}{c}{$\tau_n^{\rm ref}$} &
    \multicolumn{1}{c}{$\eta_n$} &
    \multicolumn{1}{c}{$R_T(f_n^m)$} &
    \multicolumn{1}{c@{}}{$\lambda_{\min}(\widehat G_n)$}\\
    \midrule
4 & $1.69{\cdot}10^{-3}$ & $3.738{\cdot}10^{-1}$ & $3.738{\cdot}10^{-1}$ & $3.04{\cdot}10^{-5}$ & $1.737{\cdot}10^{-1}$ & $0.0799$\\
8 & $1.70{\cdot}10^{-3}$ & $1.384{\cdot}10^{-1}$ & $1.384{\cdot}10^{-1}$ & $3.88{\cdot}10^{-5}$ & $8.286{\cdot}10^{-2}$ & $0.0749$\\
16 & $2.19{\cdot}10^{-3}$ & $9.422{\cdot}10^{-2}$ & $9.422{\cdot}10^{-2}$ & $4.05{\cdot}10^{-5}$ & $6.750{\cdot}10^{-2}$ & $0.0743$\\
32 & $3.99{\cdot}10^{-3}$ & $6.697{\cdot}10^{-2}$ & $6.697{\cdot}10^{-2}$ & $4.77{\cdot}10^{-5}$ & $5.445{\cdot}10^{-2}$ & $0.0734$\\
\bottomrule

\end{tabular}
\end{table}

The total error in Table~\ref{tab:peanut} is governed by $\tau_n^{\rm ref}$: at $n=32$, $\eta_n$ is about $1.4\cdot10^3$ times smaller.  The residual $R_T(f_{32}^m)=0.05445$ is $18.5\%$ of the target energy norm and about one quarter of the uncontrolled miss.  Figure~\sref{fig:peanut-errors} separates the truncation and neural contributions.  Twenty of the 32 modes stop through stagnation or the step budget (Table~\sref{tab:peanut-modewise} in Supplementary Section~\sref{app:example1-diagnostics}). The adjoint-training problem depends on the domain, coefficient, terminal data, and final time; the control region enters the subsequent HUM reconstruction. The present experiments do not isolate the causes of the different training behavior across examples.

A finer Legendre basis of degrees $(104,56)$ with cutoff $\lambda\leq500$ changes the reference control by $0.482\%$ relative to the finer reference norm when the two controls are compared directly as functions on $Q_\omega$.  This is a measured distance between two Ritz controls, not an estimate of $\|f_{\rm ref}-f^*\|$.  The finer Ritz modal forward solve changes $R_T(f_{32,\mathrm{ref}}^H)$ by $0.12\%$.  There is no forward solve for this example independent of the Ritz construction; thus the displayed terminal residual is a finite-model quantity.

\subsection{Example 2: a slow lens in a square}\label{sec:example2}

Let $\Omega=(0,\pi)^2$, $\omega=\{x>3\pi/4\}\cup\{y>3\pi/4\}$, and $T=3\pi$. The coefficient \[a(x,y)=1-0.6G(x)G(y),\qquad G(s)=\exp[-2(\cos s-\cos(\pi/3))^2]\] has minimum $0.4$ in the uncontrolled corner.  Because $G$ is smooth, even at the coordinate walls, and $2\pi$-periodic, the Dirichlet problem has a smooth odd-reflection extension to a torus. This addresses the corner regularity required by the sine spectral reference, although the physical domain does not have the smooth boundary assumed in the general theorem.  The initial displacement and velocity have sine coefficients $0.6^{j+\ell}\mu_{j\ell}^{-1}$ and $0.6^{j+\ell}\mu_{j\ell}^{-1/2}$, respectively; the target displacement has coefficients $(-1)^{j+\ell}0.6^{j+\ell}\mu_{j\ell}^{-1}$ and the target velocity is zero, with $\mu_{j\ell}=j^2+\ell^2$.

The reference comprises 222 Ritz eigenpairs with $\lambda\leq250$ (444 terminal modes), using sine indices at most 56.  Here $\epsilon_k=10^{-3}$ for $k\leq16$ and $1.5\cdot10^{-3}$ afterward; all 32 trained modes meet these relative targets.  The coefficient bends the rays toward the slow part of the domain.  Of $1{,}841{,}440$ rays examined by random sampling, a deterministic sweep, and local search, the longest entry time found was $2.597\pi<T$.  The smallest finite Gram eigenvalue is approximately $0.5813$ for 444 modes and $0.58123$ for an enlarged 1844-mode Ritz system.  Figure~\ref{fig:lens-geometry} shows these diagnostics, which remain numerical checks.

\begin{figure}[!htbp]\centering
\includegraphics[width=\textwidth]{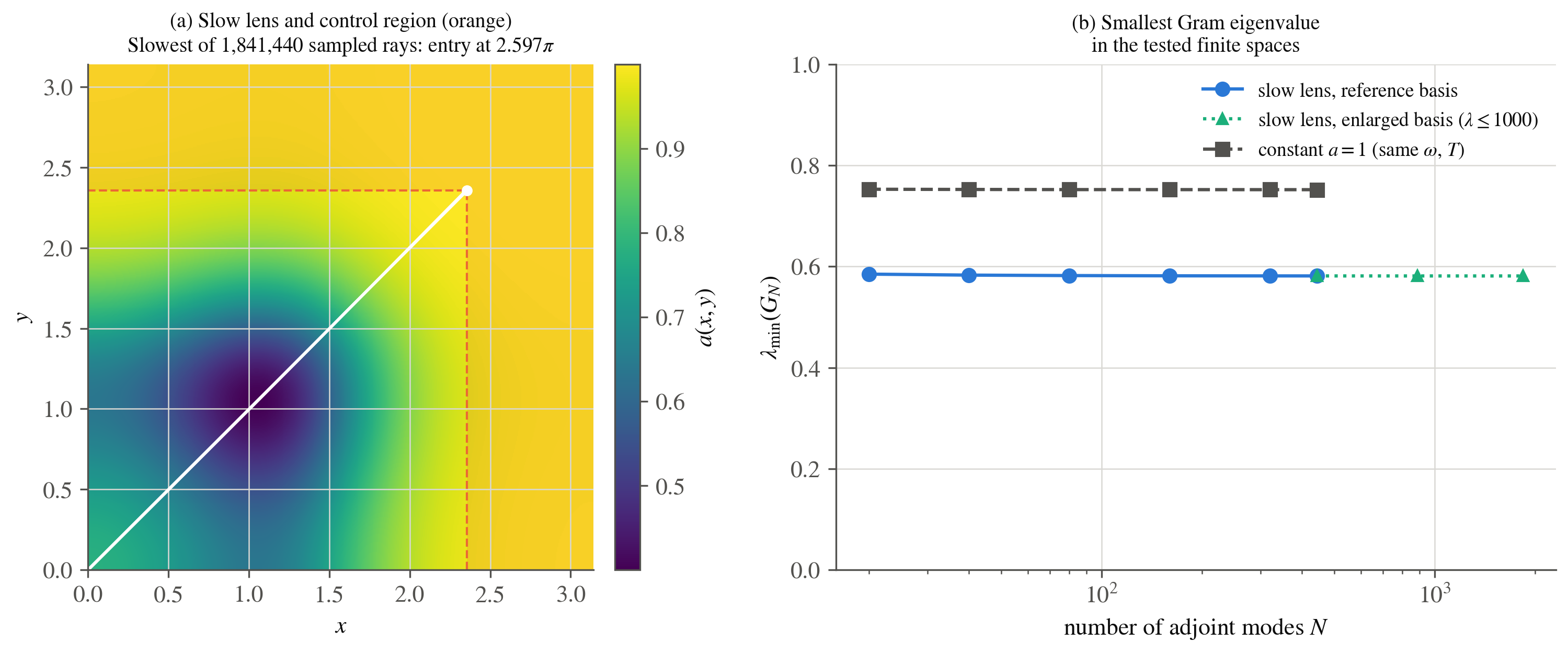}
\caption{Example 2.  Slow-lens coefficient and sampled ray (left); smallest observation Gram eigenvalue in finite reference systems (right). The right panel is a diagnostic over the displayed finite dimensions; its apparent plateau does not certify a uniform infinite-dimensional lower bound.}
\label{fig:lens-geometry}
\end{figure}

\begin{table}[!htbp]\centering
\small
\caption{Example 2: results in the 444-mode reference model. The maximum Ritz-based least-squares residual is normalized as in \eqref{eq:ritz-relative-indicator}.}
\label{tab:lens-results}

\setlength{\tabcolsep}{4pt}
\begin{tabular}{@{}rrrrrrr@{}}
    \toprule
    \multicolumn{1}{@{}c}{$n$} &
    \multicolumn{1}{c}{$\max e_k^{\mathrm{rel}}$} &
    \multicolumn{1}{c}{$\|f_n^m-f_{\rm ref}\|$} &
    \multicolumn{1}{c}{$\tau_n^{\rm ref}$} &
    \multicolumn{1}{c}{$\eta_n$} &
    \multicolumn{1}{c}{$R_T(f_n^m)$} &
    \multicolumn{1}{c@{}}{$\lambda_{\min}(\widehat G_n)$}\\
    \midrule
4 & $9.10{\cdot}10^{-4}$ & $1.5308{\cdot}10^{-1}$ & $1.5307{\cdot}10^{-1}$ & $1.64{\cdot}10^{-5}$ & $1.9171{\cdot}10^{-1}$ & $0.6059$\\
8 & $9.10{\cdot}10^{-4}$ & $1.1129{\cdot}10^{-1}$ & $1.1129{\cdot}10^{-1}$ & $1.71{\cdot}10^{-5}$ & $1.3008{\cdot}10^{-1}$ & $0.6043$\\
16 & $9.97{\cdot}10^{-4}$ & $7.0579{\cdot}10^{-2}$ & $7.0579{\cdot}10^{-2}$ & $1.76{\cdot}10^{-5}$ & $8.6015{\cdot}10^{-2}$ & $0.5912$\\
32 & $1.46{\cdot}10^{-3}$ & $2.5040{\cdot}10^{-2}$ & $2.5040{\cdot}10^{-2}$ & $3.86{\cdot}10^{-5}$ & $2.8605{\cdot}10^{-2}$ & $0.5833$\\
\bottomrule

\end{tabular}
\end{table}

At $n=32$, the neural reconstruction error is $3.86\cdot10^{-5}$, while the finite-model truncation error is $2.504\cdot10^{-2}$ (Table~\ref{tab:lens-results}). Figure~\sref{fig:lens-errors} illustrates this separation. The maximum absolute Ritz-based least-squares residual increases from $4.55\cdot10^{-3}$ at $n=4$ to $3.40\cdot10^{-2}$ at $n=32$, despite the relative residuals remaining near $10^{-3}$.

Only Examples~2 and~4 use post-training pruning. In Example~2, neurons satisfying $\max\{|\alpha|,|W_1|,|W_2|\}>24$ were removed before postprocessing: 375 of 16,025 neurons in total. Their coefficients obeyed $|c|\leq4.2\cdot10^{-6}\max|c|$, and the removed part was bounded by $1.1\cdot10^{-6}$ in $L^2(Q_\omega)$ per mode. The reported least-squares residual indicators were recomputed for the pruned fields on $240\times120^2$ and $320\times160^2$ Gauss rules. The global Galerkin coefficients were not recomputed after pruning, so the postprocessed fields are not asserted to be exact Galerkin projections; this postprocessing introduces an additional discrepancy from the discrete Galerkin minimizer. Applying the fully discrete analysis would require a bound on that discrepancy together with the Ritz-data consistency and quadrature errors.

\paragraph{The a priori stability condition.} Within the finite model, the directly measured synthesis difference $\|\widetilde\Phi_{32}-\Phi_{32,R}\|$ is $6.23\cdot10^{-4}$; the values for all tested $n$ are listed in Table~\sref{tab:lens-finite-diagnostics}. Let $\alpha_{444}:=\sqrt{\lambda_{\min}(G_{444})}\approx0.7624$ be the \emph{finite-model} lower constant, so that $\alpha_{444}/2\approx0.381$, and let the measured Gram assembly sensitivity $\Delta_G^{\mathrm{asm}}=1.6\cdot10^{-9}$ stand in for $\delta_{G,n,m,h}$, compared with $\alpha_{444}^2/8\approx0.073$. The embedding constant \eqref{eq:explicit-ls-embedding-constant} with $C_p=2^{-1/2}$, $a_0=0.4$, $a_1=1$ and $T=3\pi$ is $C_X\approx8.20$. Even if the absolute Ritz-based residuals of Table~\sref{tab:lens-finite-diagnostics} were taken as the adjoint errors $E^{\mathrm{adj}}_{n,m,h}$, and the constant of equivalence between \eqref{eq:numerical-graph-norm} and $\|\cdot\|_{LS}$ were ignored, the quantity $C_X\sqrt n\,E^{\mathrm{adj}}_{n,m,h}$ would be $0.075$, $0.124$, $0.316$ and $1.58$ for $n=4,8,16,32$. The sufficient condition \eqref{eq:fully-discrete-stability-conditions} would therefore hold for $n\leq16$, nearly saturated at $n=16$, and fail at $n=32$ by a factor of about $4$; the corresponding a priori bound for the synthesis perturbation exceeds its measured value by a factor of about $2.5\cdot10^3$. The worst-case factor $\sqrt n$ in Proposition~\ref{prop:approximate-gram-stability} and the strong terminal norms in $\|\cdot\|_{LS}$, which make $\|\phi_k\|_{LS}$ grow like $\lambda_k$ while $\|\phi_k\|_{L^2(Q_\omega)}$ stays bounded, are the main sources of this pessimism. The a priori estimates therefore explain the structure of the error but do not predict its size in these examples, and the stability of the reconstructions is established here only by the finite-model measurements. Neither a continuous observability constant nor a certified $E^{\mathrm{adj}}_{n,m,h}$ or $\delta_{G,n,m,h}$ is supplied for this example.

The corresponding state fields are displayed in Figure~\ref{fig:lens-states}. The controlled terminal displacement captures the target profile concentrated near the upper-right part of the square, while the terminal velocity remains small relative to the initial velocity, consistently with the zero target velocity. The discrepancy panels display the pointwise terminal discrepancies and use substantially narrower color scales than the state panels. These plots provide a qualitative physical-space illustration; the terminal energy residual $R_T$ and its independent forward verification are reported separately below.

\begin{figure}[!htbp]\centering
\includegraphics[width=\textwidth]{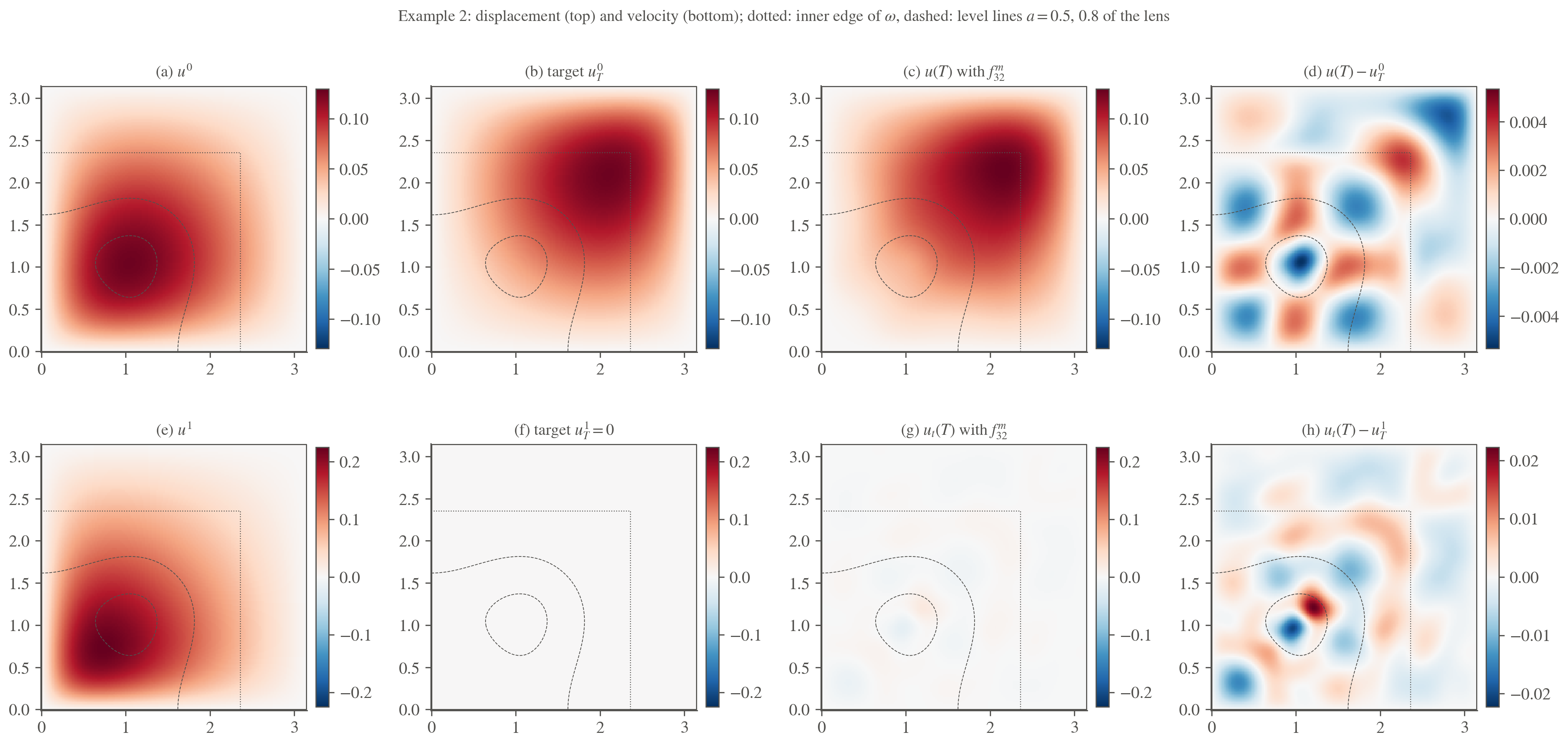}
\caption{Example 2: initial, target, and controlled terminal states for \(n=32\). The top and bottom rows show displacement and velocity, respectively; the columns show the initial state, prescribed target, terminal state generated by \(f_{32}^{m}\), and pointwise terminal discrepancy. The dotted line marks the inner edge of the L-shaped control region \(\omega\), and the dashed curves are the level sets \(a=0.5\) and \(a=0.8\) of the slow lens. The discrepancy panels use separate, narrower color scales. Panel (f) is uniformly zero because the target velocity vanishes.}
\label{fig:lens-states}
\end{figure}

An independent flux-form finite-difference forward solver gives $R_T(f_{32}^m)=0.028677$ on a $512^2$ grid and $0.028687$ after extrapolation, against $0.028605$ in the original Ritz model. An enlarged Ritz solve gives $0.028670$; the complete forward checks are listed in Table~\sref{tab:lens-forward-checks}. For the 444-mode reference control itself, the extrapolated independent residual is about $0.002012$, whereas it is essentially zero inside the 444-mode model.  These values show sensitivity to the forward discretization and reference-space enlargement; they do not isolate the contributions of omitted modes, quadrature, and forward-solver error.  They support the reported terminal residual for $f_{32}^m$ at the tested resolutions without proving exact controllability for the continuous square problem.

\subsection{Example 3: a graded medium in a ball}

In three dimensions take $\Omega=\{x\in\mathbb R^3:|x|<\pi\}$, $T=3.5\pi$, and \[a(x)=0.8-\frac{0.2x_1}{\pi},\qquad\omega=\left\{x:\pi-\frac{\pi}{4}\bigl(1+0.6\widehat x_1\bigr)<|x|<\pi\right\}.\] Here $\widehat x=x/|x|$ and the shell width is $\delta(\widehat x)=\frac{\pi}{4}(1+0.6\widehat x_1)$. The data are smooth heat-kernel bumps built from Dirichlet Laplacian eigenfunctions with multiplier $e^{-\mu_j}$; their initial and target centers lie at radius $0.55\pi$ in the directions $(110^\circ,200^\circ)$ and $(60^\circ,40^\circ)$ in spherical angles.  Displacement and velocity components are both nonzero, with velocity multipliers $\pm0.5\sqrt{\mu_j}e^{-\mu_j}$.

The reference uses polynomial Rayleigh--Ritz degree 28 and $\lambda\leq40$, giving 657 eigenpairs and 1314 terminal modes. The relative residual targets are $10^{-4}$ for $k\leq16$ and $2\cdot10^{-4}$ afterward.  At $n=32$, 31 of 32 modes meet their targets; mode 21 reaches the six-step budget with relative residual $2.36\cdot10^{-4}$.  We retain this result and report the unmet tolerance explicitly.  One million sampled rays and a local search found all entry times below $1.69\pi<T$; the smallest Gram eigenvalue falls from $0.490$ at 20 modes to about $0.329$ at 1314 modes (Figure~\sref{fig:graded-geometry}).  This finite trend cannot establish a uniform lower bound.

\begin{table}[!htbp]\centering
\small
\caption{Example 3: results in the $\lambda\leq40$ Ritz model. The maximum Ritz-based least-squares residual is normalized as in \eqref{eq:ritz-relative-indicator}; $N_{\rm tol}$ is the number of modes meeting their targets.}
\label{tab:graded}
\setlength{\tabcolsep}{3pt}
\begin{tabular}{@{}rrrrrrrr@{}}
    \toprule
    \multicolumn{1}{@{}c}{$n$} &
    \multicolumn{1}{c}{$N_{\rm tol}$} &
    \multicolumn{1}{c}{$\max e_k^{\mathrm{rel}}$} &
    \multicolumn{1}{c}{$\|f_n^m-f_{\rm ref}\|$} &
    \multicolumn{1}{c}{$\tau_n^{\rm ref}$} &
    \multicolumn{1}{c}{$\eta_n$} &
    \multicolumn{1}{c}{$R_T(f_n^m)$} &
    \multicolumn{1}{c@{}}{$\lambda_{\min}(\widehat G_n)$}\\
    \midrule
4 & 4 & $7.14{\cdot}10^{-5}$ & $5.5437{\cdot}10^{-2}$ & $5.5437{\cdot}10^{-2}$ & $3.33{\cdot}10^{-7}$ & $5.8390{\cdot}10^{-2}$ & $0.6340$\\
8 & 8 & $9.28{\cdot}10^{-5}$ & $4.4856{\cdot}10^{-2}$ & $4.4856{\cdot}10^{-2}$ & $4.94{\cdot}10^{-7}$ & $4.8835{\cdot}10^{-2}$ & $0.6340$\\
16 & 16 & $9.28{\cdot}10^{-5}$ & $3.2018{\cdot}10^{-2}$ & $3.2018{\cdot}10^{-2}$ & $5.93{\cdot}10^{-7}$ & $3.1206{\cdot}10^{-2}$ & $0.5872$\\
32 & 31 & $2.36{\cdot}10^{-4}$ & $2.7045{\cdot}10^{-2}$ & $2.7045{\cdot}10^{-2}$ & $3.00{\cdot}10^{-6}$ & $2.6102{\cdot}10^{-2}$ & $0.4864$\\
\bottomrule

\end{tabular}
\end{table}

The neural error is between $3.3\cdot10^{-7}$ and $3.0\cdot10^{-6}$, much smaller than the corresponding finite-model truncation errors (Table~\ref{tab:graded} and Figure~\sref{fig:graded-errors}).  The Ritz degree check gives $\|f_{\rm ref}^{(28)}-f_{\rm ref}^{(32)}\| =4.6\cdot10^{-7}$ at $\lambda\leq40$ (Table~\sref{tab:graded-degree-check}). The cutoff check, by contrast, gives $\|f_{\rm ref}(40)-f_{\rm ref}(60)\| =2.108\cdot10^{-3}$ (Table~\sref{tab:graded-cutoff}). This is a measured change between finite Ritz controls, not the unknown distance from $f_{\rm ref}(40)$ to $f^*$.

An independent forward solver using spherical Bessel functions, spherical harmonics, and the exact stiffness matrix of $a$ gives $R_T(f_{32}^m)=0.026139$ with its largest Laplacian cutoff $\mu\leq200$ (Table~\sref{tab:graded-forward-checks}), against $0.026102$ in the original Ritz model.  It gives $R_T(f_{\rm ref}(40))=0.002706$ at that cutoff, instead of the nearly zero internal Ritz residual. The latter independent value still changes with forward cutoff, from $0.002629$ at $\mu\leq150$ to $0.002706$ at $\mu\leq200$; it is not presented as a converged continuous-PDE residual.  By comparison, the network controls' residuals change by at most $4\cdot10^{-6}$ between these two cutoffs.

\subsection{Example 4: a pulsating lens}\label{sec:example4}

Example~4 tests the method on a problem for which the modal formula \eqref{eq:direct-modal-adjoint-comparator} is not available. We keep the square, the control region, the final time, and the initial and target states of Example~2, and let the lens pulsate:
\[
a(t,x,y)=1-0.6\,m(t)\,G(x)G(y),\qquad m(t)=1-\delta\bigl(1-\cos(\nu t)\bigr),\qquad \nu=\tfrac43,\quad\delta=0.25 .
\]
Since $\nu T=4\pi$, we have $m(0)=m(T)=1$, so $a(T,\cdot)$ is the coefficient of Example~2. The terminal space $\mathcal E_T$ of Section~\ref{sec:time-dependent}, its Ritz pairs, and the boundary factor are therefore those of Example~2. The coefficient satisfies $0.4\leq a\leq1$, and the constant of Proposition~\ref{prop:time-dependent-coefficients} is $K=4\log(1+3\delta)=4\log(7/4)\approx2.23846$. Indeed the spatial maximum of $|a_t/a|$ occurs at $G(x)G(y)=1$, and each of the two temporal periods contributes $2\log((0.4+1.2\delta)/0.4)$. Writing $a=a_T+\mu(t)a_1$ with $a_1=-0.6\,G(x)G(y)$ and $\mu=m-1$, let $B=-\nabla\cdot(a_1\nabla)$, so $A_T=-\Delta+B$ and $A(t)=A_T+\mu(t)B$. These operators do not commute: the principal symbols of $-\Delta$ and $B$ are $|\xi|^2$ and $a_1(x)|\xi|^2$, whose Poisson bracket is $2(\xi\cdot\nabla a_1)|\xi|^2$, nonzero for this nonconstant $a_1$. Since $\mu$ is nonconstant, $[A(t),A(s)]=(\mu(s)-\mu(t))[A_T,B]$ is nonzero for some $s,t$, precluding a common eigenbasis. The separated formula \eqref{eq:direct-modal-adjoint-comparator} therefore does not solve the adjoint equation, while a Galerkin representation in a fixed spatial basis leads to coupled ordinary differential equations.

We compare three constructions that use the same terminal space and the same retained modes.
\begin{itemize}
\item \emph{M1, frozen modal HUM.} These are the Example~2 controls $f_{n,\mathrm{ref}}^H$, obtained by freezing the coefficient at $a(T,\cdot)$. They are built for a different equation; M1 only measures the effect of ignoring the time dependence.
\item \emph{M2, the neural Galerkin--HUM method.} The driver and the settings of Example~2 are used, with the interior residual $\mathcal Pv=v_{tt}-\nabla\cdot(a(t,x)\nabla v)$ and one change in the initialization described below. For a neuron, the part of the residual involving $a_1$ carries the additional time factor $\mu(t)$, so the separable evaluation of the least-squares integrals in time and space is retained. The terminal operator target is $A_Te_{j,R}$, computed by differentiating the sine expansion of the Ritz eigenfunction, instead of $\lambda_{j,R}e_{j,R}$. The operator-data consistency defect $D_{k,R}$ of Section~\ref{sec:numerics} is thus zero by construction. For this sine basis the two targets differ only slightly: $\|A_Te_{j,R}-\lambda_{j,R}e_{j,R}\|_{L^2(\Omega)}/\|\lambda_{j,R}e_{j,R}\|_{L^2(\Omega)}\leq3\cdot10^{-12}$ for the tested modes.
\item \emph{M3, Ritz--Galerkin in space with time integration.} The adjoint fields are written as $\phi(t,x)=\sum_ic_i(t)e_{i,R}(x)$ over the $222$ Ritz eigenfunctions of Example~2, and
\[
c''(t)+\bigl(\Lambda+\mu(t)K_1\bigr)c(t)=0,\qquad (K_1)_{ij}=\int_\Omega a_1\nabla e_{i,R}\cdot\nabla e_{j,R}\,dx,
\]
with $\Lambda=\operatorname{diag}(\lambda_{i,R})$, is integrated backward from the terminal data by an explicit eighth-order Runge--Kutta method (DOP853, relative tolerance $10^{-11}$). The Gram matrix and the moments are assembled in this basis, with $\int_\omega e_{i,R}e_{j,R}$ computed exactly in the sine basis and the time integrals by composite Gauss rules. The full HUM control of this finite model, using all $444$ terminal modes, is denoted $f_{\mathrm{ref}}$, and $f_n^{\mathrm{M3}}$ uses the first $n$ modes. For $\delta=0$, M3 reproduces the Gram matrix of Example~2 to a relative difference of $4.6\cdot10^{-10}$, and in both cases $f_{\mathrm{ref}}$ leaves a terminal residual of about $4\cdot10^{-11}$ in the finite model.
\end{itemize}
M3 is the main comparator: M2 and M3 solve the same problem. Because the terminal target of M2 is consistent, M2 and M3 approximate the same adjoint solutions of the time-dependent equation, and
\[
\eta_n=\|f_n^m-f_n^{\mathrm{M3}}\|_{L^2(Q_\omega)},\qquad\tau_n^{\mathrm{ref}}=\|f_n^{\mathrm{M3}}-f_{\mathrm{ref}}\|_{L^2(Q_\omega)}
\]
replace the quantities of Section~\ref{sec:numerics}. The quantity $\eta_n$ now also contains the spatial Galerkin error of M3. We estimate this error by recomputing $f_n^{\mathrm{M3}}$ with the same terminal data in the enlarged Ritz basis of Example~2 ($922$ eigenpairs, $\lambda\leq1000$). We regard $\eta_n$ as a measurement of the neural error only when this difference is at most $\eta_n/10$.

\paragraph{Initialization of the temporal frequencies.} For a time-independent coefficient, the adjoint modes oscillate in time with the single frequency $\sqrt{\lambda_i}$. In Examples~1--3, about $80\%$ of the temporal frequencies are therefore initialized in $\pm\sqrt{\lambda_i}[0.95,1.05]$ (Supplementary Section~\sref{app:numerical-diagnostics}). For the coefficient of Example~4, the product $\mu(t)A_1$ shifts temporal frequencies by $\pm\nu$, and the adjoint modes acquire side bands at $\pm\nu$, $\pm2\nu$, \dots\ around the frequencies of the frozen operators. A pilot run with the Example~2 initialization reduced the Ritz residuals of the modes $k=7,8$ only to $4.3\cdot10^{-2}$ and $7.4\cdot10^{-2}$ after two accepted directions. In Example~2 the same modes are below $2\cdot10^{-3}$ after one direction. For Example~4, the temporal frequencies of all neurons outside the $20\%$ dispersion-shell group are drawn from $U(-S_t,S_t)$ with $S_t=1.2\sqrt{\lambda_i}+2\nu$, which covers the first two side bands. A diagnostic using the M3 adjoints, which would not be available in a practical computation, confirms this width. Best approximations of the M3 adjoints of the first $16$ modes by time harmonics, measured in the energy norm, leave the following relative errors: between $2.6\%$ and $6.2\%$ for frequencies in $\pm\sqrt{\lambda_i}[0.95,1.05]$, between $0.6\%$ and $1.4\%$ for frequencies in $[-1.2\sqrt{\lambda_i},1.2\sqrt{\lambda_i}]$, and at most $0.07\%$ for frequencies in $[-S_t,S_t]$. On mode $k=8$, this rule was compared with the variant that keeps half of the neurons near $\pm\sqrt{\lambda_i}$. After two directions their relative residuals were $3.4\cdot10^{-2}$ and $3.9\cdot10^{-2}$, and the first rule was used for all modes. This selection uses only the training indicator of the method, not the comparison with M3. The tuning runs took about $4.5$ accumulated wall-clock hours and are not included in Table~\ref{tab:cost}. All other settings are those of Example~2.

\paragraph{Observability diagnostics.} For the time-dependent coefficient, the rays are the projections of the bicharacteristics of $\tau^2-a(t,x)|\xi|^2$, and every ray on $(0,T)$ is determined by its state at $t=0$. Rays were started at $t=0$ outside $\omega$ by random sampling, a deterministic sweep, and local search. All $1{,}841{,}440$ of them entered $\omega$, and the longest entry time found was $2.422\pi<T$ (Figure~\ref{fig:ex4-geometry}); in Example~2 it was $2.597\pi$. Since $m(t)\leq1$, the pulsating lens is on average weaker than the lens of Example~2. The smallest Gram eigenvalue of the full M3 model is $0.274$, but $94\%$ of the corresponding eigenvector lies on the terminal modes with $\lambda>200$, near the cutoff $\lambda\leq250$ of the spatial basis; for the terminal modes with $\lambda\leq200$ alone, the smallest eigenvalue is $0.6163$. With the $453$ Ritz eigenfunctions with $\lambda\leq500$ as the spatial basis of M3, the smallest eigenvalue of the Gram matrix of the terminal modes with $\lambda\leq\Lambda$ lies between $0.6162$ and $0.6164$ for $\Lambda=125,250,300,375$. It drops to $0.292$ and $0.149$ for $\Lambda=450$ and $500$, when terminal modes within $25\%$ of the basis cutoff are included; more than $91\%$ of the corresponding eigenvector then lies on the modes with $\lambda>0.8\Lambda$. We therefore attribute the smaller values to the truncation of the spatial basis. The control $f_{\mathrm{ref}}$ is insensitive to this truncation: it differs from the HUM control of the enlarged model with $906$ terminal modes by $1.33\cdot10^{-3}$, that is, $0.28\%$ of its norm. As in Examples~1--3, the rays and Gram spectra are numerical checks; they do not establish the observability hypothesis.

\begin{figure}[!htbp]\centering
\includegraphics[width=\textwidth]{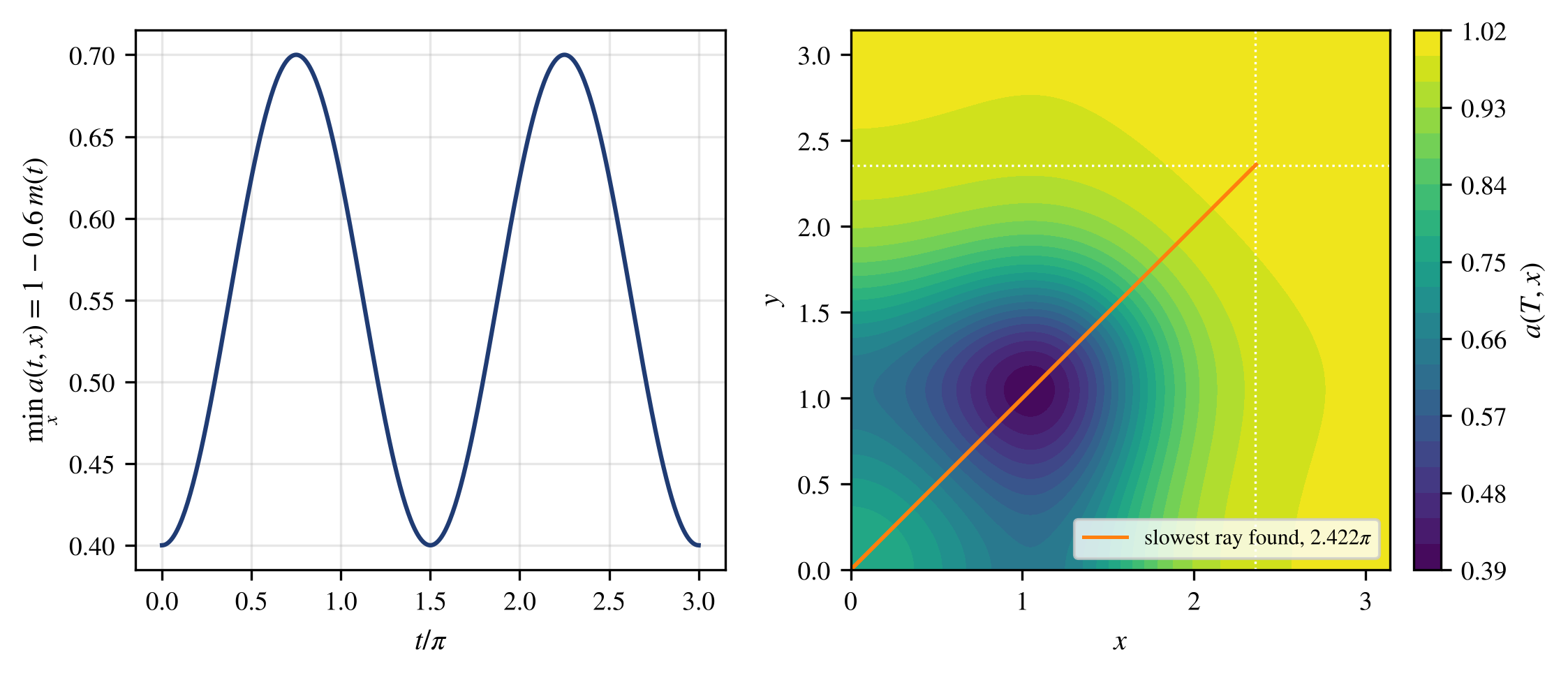}
\caption{Example 4. Left: minimum of the coefficient over $\Omega$ as a function of time. Right: terminal coefficient $a(T,\cdot)$, which equals the coefficient of Example~2, with the inner edge of $\omega$ (dotted) and the projection of the slowest ray found, started at $t=0$.}
\label{fig:ex4-geometry}
\end{figure}

\paragraph{Training.} The modes $k=1,\dots,8$ were trained. None of them reached its relative tolerance $10^{-3}$: five stopped by stagnation after five accepted directions and three at the budget of six directions (Table~\sref{tab:ex4-modewise} in Supplementary Section~\sref{app:example4}). The final relative Ritz-based residuals are between $9.9\cdot10^{-3}$ and $3.0\cdot10^{-2}$, whereas all trained modes of Example~2 reach their tolerances near $10^{-3}$, most low-frequency modes with one or two directions. The networks were pruned as in Example~2: $632$ of $10{,}734$ neurons had a temporal or spatial frequency above $24$. Their coefficients obeyed $|c|\leq8.2\cdot10^{-6}\max|c|$, and the removed part is bounded by $5.0\cdot10^{-5}$ in $L^2(Q_\omega)$ per mode. Recomputed after pruning on the two refined Gauss rules of Example~2, the residuals are between $0.1\%$ and $17.5\%$ smaller than the training values, and the two rules agree to a relative difference of $6\cdot10^{-13}$. All quantities below use the pruned networks.

\paragraph{Controls.} Table~\ref{tab:ex4-results} and Figure~\sref{fig:ex4-errors} report the errors. The terminal modes $k=5,6$ belong to the third eigenfunction, which is antisymmetric under the interchange of the two coordinates. The coefficient and the data are symmetric, so the moments of these modes vanish and $f_6^{\mathrm{M1}}=f_4^{\mathrm{M1}}$, $f_6^{\mathrm{M3}}=f_4^{\mathrm{M3}}$; for the network controls, $\eta_n$, $\|f_n^m-f_{\mathrm{ref}}\|$, and $R_T$ change by less than $10^{-5}$, and the rows for $n=6$ are omitted. Four checks support the interpretation of $\eta_n$. First, the frozen controls are far from the M3 controls, and the neural controls are close to them: $\|f_n^{\mathrm{M1}}-f_n^{\mathrm{M3}}\|$ is $31.5$, $17.3$, and $14.5$ times $\eta_n$ for $n=2,4,8$. Second, recomputing $f_n^{\mathrm{M3}}$ in the $922$-function basis changes it by at most $2.7\cdot10^{-9}$, which is below $3.1\cdot10^{-7}\eta_n$; the M3 controls of the first eight modes are therefore resolved, and $\eta_n$ measures the neural error. Third, the distances are insensitive to refinement of the evaluation rule: re-evaluating $\eta_2,\eta_4,\eta_8$ directly from the difference fields in float64 with $200$ instead of $140$ Legendre components, $450$ instead of $300$ time nodes, or a spatial rule $1.5$ times finer changes them by at most $3\cdot10^{-13}$ relative (Table~\sref{tab:resolution}). Fourth, the independent forward solves below agree with the M3 model. The neural error $\eta_8=1.08\cdot10^{-2}$ is $2.4\%$ of $\|f_8^m\|$ and of the same relative size as the adjoint residuals. It is $6.6\%$ of the truncation error $\tau_8^{\mathrm{ref}}=0.164$; in Examples~1--3 the corresponding ratios at $n=32$ are between $1.1\cdot10^{-4}$ and $1.5\cdot10^{-3}$.

\begin{table}[!htbp]\centering
\small
\caption{Example 4: errors of the controls in $L^2(Q_\omega)$ with respect to the M3 model with $444$ terminal modes. The residual $\max e_k^{\mathrm{rel}}$ is the largest resolved relative Ritz-based residual over $k\leq n$; the M3 refinement column is $\|f_n^{\mathrm{M3}}-f_n^{\mathrm{M3},922}\|$, where $f_n^{\mathrm{M3},922}$ is computed in the enlarged $922$-function basis.}
\label{tab:ex4-results}
\setlength{\tabcolsep}{4pt}
\begin{tabular}{@{}rrrrrrrr@{}}
\toprule
\multicolumn{1}{@{}c}{$n$} &
\multicolumn{1}{c}{$\max e_k^{\mathrm{rel}}$} &
\multicolumn{1}{c}{$\|f_n^m-f_{\rm ref}\|$} &
\multicolumn{1}{c}{$\tau_n^{\rm ref}$} &
\multicolumn{1}{c}{$\eta_n$} &
\multicolumn{1}{c}{M3 refinement} &
\multicolumn{1}{c}{$\|f_n^{\mathrm{M1}}-f_n^{\mathrm{M3}}\|$} &
\multicolumn{1}{c@{}}{$\lambda_{\min}(\widehat G_n)$}\\
\midrule
2 & $1.06{\cdot}10^{-2}$ & $3.517{\cdot}10^{-1}$ & $3.513{\cdot}10^{-1}$ & $3.83{\cdot}10^{-3}$ & $5.72{\cdot}10^{-10}$ & $1.208{\cdot}10^{-1}$ & $0.6489$\\
4 & $2.52{\cdot}10^{-2}$ & $1.983{\cdot}10^{-1}$ & $1.993{\cdot}10^{-1}$ & $8.71{\cdot}10^{-3}$ & $2.65{\cdot}10^{-9}$ & $1.504{\cdot}10^{-1}$ & $0.6380$\\
8 & $2.75{\cdot}10^{-2}$ & $1.633{\cdot}10^{-1}$ & $1.644{\cdot}10^{-1}$ & $1.08{\cdot}10^{-2}$ & $2.37{\cdot}10^{-9}$ & $1.567{\cdot}10^{-1}$ & $0.6368$\\
\bottomrule

\end{tabular}
\end{table}

\paragraph{Terminal residuals.} Table~\ref{tab:ex4-residuals} compares $R_T$ in the M3 model with the independent finite-difference solver of Example~2, extended to the time-dependent coefficient (Supplementary Section~\sref{app:example4}). For the controls with $n\leq8$, the extrapolated finite-difference value agrees with the M3 value to a relative difference of at most $1.3\cdot10^{-4}$. The neural and M3 controls leave nearly the same residual, for example $0.19149$ and $0.19187$ at $n=8$; the residual of the frozen control is $0.24799$. The full control $f_{\mathrm{ref}}$ leaves $1.9\cdot10^{-11}$ in the M3 model and $1.41\cdot10^{-3}$ in the finite-difference solve (Example~2: $2.0\cdot10^{-3}$). The full frozen control $f_{\mathrm{ref}}^{\mathrm{M1}}$, which steers the frozen model exactly, leaves $0.199$; without control the residual is $0.432$. The full frozen control thus achieves only about half of the reduction of the terminal residual that $f_{\mathrm{ref}}$ achieves.

\begin{table}[!htbp]\centering
\small
\caption{Example 4: terminal energy residuals $R_T$ in the M3 model and from the independent finite-difference solver (FD, Richardson extrapolation of the $128^2$, $256^2$, and $512^2$ grids). The last row gives the full control $f_{\rm ref}$ of the M3 model and, in the M1 columns, the full frozen control of Example~2.}
\label{tab:ex4-residuals}
\setlength{\tabcolsep}{4pt}
\begin{tabular}{@{}lrrrrrr@{}}
\toprule
& \multicolumn{2}{c}{neural (M2)} & \multicolumn{2}{c}{M3} & \multicolumn{2}{c@{}}{frozen (M1)}\\
\cmidrule(lr){2-3}\cmidrule(lr){4-5}\cmidrule(l){6-7}
\multicolumn{1}{@{}c}{$n$} & \multicolumn{1}{c}{model} & \multicolumn{1}{c}{FD} & \multicolumn{1}{c}{model} & \multicolumn{1}{c}{FD} & \multicolumn{1}{c}{model} & \multicolumn{1}{c@{}}{FD}\\
\midrule
2 & $3.7660{\cdot}10^{-1}$ & $3.7661{\cdot}10^{-1}$ & $3.7596{\cdot}10^{-1}$ & $3.7598{\cdot}10^{-1}$ & $4.0851{\cdot}10^{-1}$ & $4.0853{\cdot}10^{-1}$\\
4 & $2.4715{\cdot}10^{-1}$ & $2.4716{\cdot}10^{-1}$ & $2.4737{\cdot}10^{-1}$ & $2.4738{\cdot}10^{-1}$ & $2.5718{\cdot}10^{-1}$ & $2.5720{\cdot}10^{-1}$\\
8 & $1.9149{\cdot}10^{-1}$ & $1.9151{\cdot}10^{-1}$ & $1.9187{\cdot}10^{-1}$ & $1.9189{\cdot}10^{-1}$ & $2.4799{\cdot}10^{-1}$ & $2.4802{\cdot}10^{-1}$\\
$f_{\rm ref}$ & -- & -- & $1.9{\cdot}10^{-11}$ & $1.411{\cdot}10^{-3}$ & $1.9928{\cdot}10^{-1}$ & --\\
\bottomrule

\end{tabular}
\end{table}

\paragraph{Dependence on the modulation.} Figure~\ref{fig:ex4-sweep} and Table~\sref{tab:ex4-sweep} in Supplementary Section~\sref{app:example4} repeat the comparison of M1 and M3 for $\delta=0,0.1,0.25$. For $\delta=0$ the two constructions coincide up to $6\cdot10^{-14}$ at $n\le 8$ and $3.4\cdot10^{-12}$ for the full control. The distance $\|f_n^{\mathrm{M1}}-f_n^{\mathrm{M3}}\|$ grows with $\delta$, by a factor between $2.1$ and $2.3$ from $\delta=0.1$ to $\delta=0.25$. For the full controls, the frozen construction leaves $R_T=0.094$ at $\delta=0.1$ and $0.199$ at $\delta=0.25$ in the M3 model, whereas the M3 control leaves less than $3\cdot10^{-11}$.

\begin{figure}[!htbp]\centering
\includegraphics[width=\textwidth]{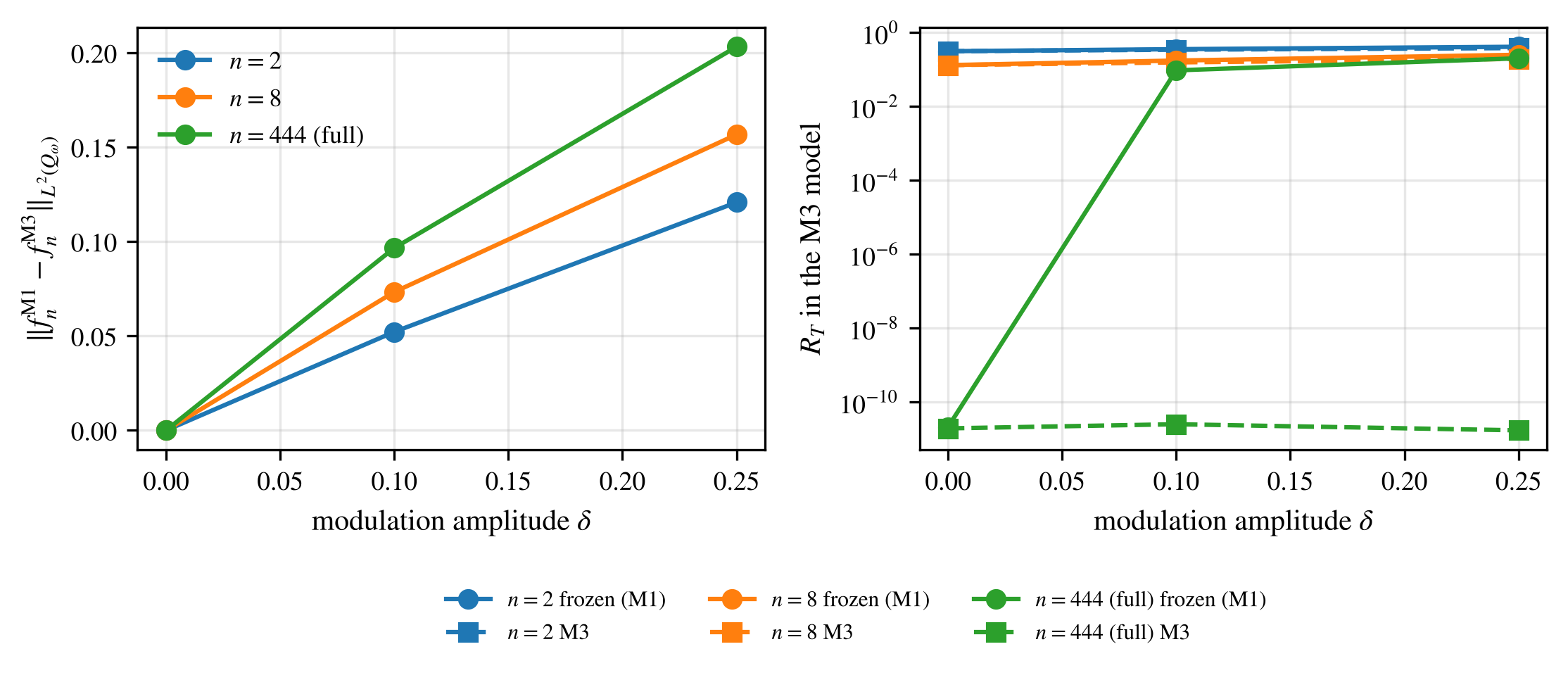}
\caption{Example 4. Dependence on the modulation amplitude $\delta$. Left: distance of the frozen controls (M1) to the M3 controls. Right: terminal residual of the M1 and M3 controls in the M3 model. For $n=2$ and $n=8$ the residual is dominated by truncation; for the full controls, the frozen construction leaves a residual of order $10^{-1}$ once $\delta>0$.}
\label{fig:ex4-sweep}
\end{figure}

In summary, at $n\leq8$ the neural Galerkin--HUM controls approximate the controls of the time-dependent problem, which the frozen modal construction does not. The adjoint residuals of this example are, however, ten to thirty times the tolerance reached in Example~2, after five or six accepted directions per mode, and the M3 comparator computes the same controls more accurately at a small fraction of the cost (Table~\ref{tab:cost}).

\paragraph{Effect of hidden-parameter training.}
\label{sec:ablation-results}
We additionally compare the production neural method with a fixed-hidden-feature
arm on the same time-dependent problem. The latter keeps each random
initialization and solves only for output and global Galerkin coefficients.
Both arms use the same terminal modes, initialization distribution, seeds,
width schedule, six-step cap, correlation/rank gates and stopping tests.
The accepted counts can differ because the stopping tests are retained.
Controls are reconstructed again on a common double-precision assessment rule,
using unpruned fields in both arms; Supplementary Section~\sref{app:ablation} gives its
settings. Thus these controls are distinct from the pruned production fields
reported above.

At $n=8$ (Table~\sref{tab:ablation} and Figure~\sref{fig:ablation}), the fixed-hidden arm has $\eta_8=0.355$, compared with
$0.0108$ for the trained-hidden arm, an improvement by a factor of about $33$
in this finite-reference comparison. The maximum final relative validation
residual decreases from $0.546$ to $0.0305$. The accepted totals are $35$ and
$43$ directions, respectively, so the final comparison is not at an equal
accepted count; equal counts are considered below. At the first enrichment,
when each arm has one direction per mode and the same width, the maximum
validation residual over the eight modes is already $0.718$ versus $0.0659$.
This is a single-seed comparison.

The fixed-hidden construction took $1292.5$ process CPU seconds ($1294.9$ wall
seconds) on an AMD EPYC 9V74 processor with one BLAS thread; the production
trained-hidden run used the Intel machines described in Table~\ref{tab:cost},
so these two timings are not comparable.

\paragraph{Controls at equal enrichment counts.}
\label{sec:matched-enrichment}
To compare the two arms at equal representation size, we reconstruct controls
from the stored prefixes with exactly $m=1,2,3$ accepted directions for each of
the first eight adjoints. Each prefix uses its own stored global Galerkin
coefficients, rather than a truncation of the final coefficient vector. Both
arms have the same total neuron count at each budget, and neither is retrained.
Assessment uses observation-node density $64$, $160$ Legendre components and
$400$ time nodes, in double precision and without frequency pruning; endpoint
moments use the $144\times144$ domain rule of the previous comparison.
\FloatBarrier

\FloatBarrier
At $m=1,2,3$, feature training reduces the distance to M3
(Table~\sref{tab:matched-enrichment} and Figure~\sref{fig:matched-enrichment}) by factors of
about $30.1$, $21.1$ and $18.5$, respectively. Terminal residuals need not
decrease monotonically with enrichment: the reconstruction approaches the
truncated M3 control, whose own terminal steering error is nonzero.

\FloatBarrier

\paragraph{Three seeds at a reduced budget.}
\label{sec:production-multiseed}
To assess seed dependence, we fixed the seeds $0,1,2$ before running a
reduced-budget protocol on the same pulsating-lens PDE, with terminal modes
$k=1,\ldots,8$ and the residual and acceptance gates of Example~4. Both arms
use two accepted directions per adjoint and constant width $64$; tolerance and
stagnation stopping are disabled. The trained arm uses at most $100$ Adam
epochs (patience $40$), followed by at most $40$ Levenberg--Marquardt
iterations at the first enrichment and $20$ at the second. The fixed arm keeps
its initialized hidden parameters and fits output coefficients. The
correlation and rank thresholds $0.1$ and $0.2$ and the single-retry rule are
those of the production runs. All seed--arm outcomes are reported; a run
failing to reach the common count would have been marked ineligible rather
than removed.

Assessment uses the rule of the equal-count comparison: observation density
$64$, $160$ Legendre components, $400$ time nodes and $144\times144$ endpoint
quadrature, without pruning. The CPU column sums per-mode process CPU time,
including all attempts in that mode, on the AMD EPYC 9V74 processor of the
fixed-feature run above; four single-thread workers ran concurrently, and
common reference construction and assessment are excluded. This protocol has
its own, much smaller budget than the $23$-hour production training.
\FloatBarrier
\begin{table}[!htbp]\centering\small
\caption{All prespecified outcomes of the reduced-budget Example~4 study,
$n=8$. The residual is the maximal relative validation residual over the eight
modes and $\lambda_{\min}$ the smallest eigenvalue of the assembled Gram
matrix $\widehat G_8$. Distances and terminal residuals refer to M3.}
\label{tab:production-multiseed}
\setlength{\tabcolsep}{4pt}
\begin{tabular}{@{}rlrrrrrrr@{}}\toprule
seed & features & directions & neurons & max. residual & $\lambda_{\min}$ & $\eta_8$ & $R_T$ & CPU (s)\\
\midrule
0 & fixed & 16 & 1024 & $0.843$ & $0.0078$ & $9.721\cdot10^{-1}$ & $5.600\cdot10^{-1}$ & $50.9$\\
0 & trained & 16 & 1024 & $0.203$ & $0.5194$ & $4.785\cdot10^{-2}$ & $2.042\cdot10^{-1}$ & $1587.1$\\
1 & fixed & 16 & 1024 & $0.837$ & $0.0151$ & $9.212\cdot10^{-1}$ & $7.333\cdot10^{-1}$ & $50.5$\\
1 & trained & 16 & 1024 & $0.209$ & $0.4635$ & $8.724\cdot10^{-2}$ & $2.187\cdot10^{-1}$ & $1523.6$\\
2 & fixed & 16 & 1024 & $0.867$ & $0.0074$ & $9.810\cdot10^{-1}$ & $6.869\cdot10^{-1}$ & $50.3$\\
2 & trained & 16 & 1024 & $0.200$ & $0.5106$ & $6.024\cdot10^{-2}$ & $2.128\cdot10^{-1}$ & $1557.5$\\
\bottomrule

\end{tabular}
\end{table}
\FloatBarrier
All six arms reach the requested $16$ accepted directions and $1024$
neurons. Feature training reduces the control distance to M3 for each seed
(Figure~\sref{fig:production-multiseed}), by factors between $10.6$ and $20.3$, and the trained distances, $0.048$ to
$0.087$, vary noticeably with the seed. At this width the fixed arm is
severely under-resolved: its maximal relative residual is $0.84$--$0.87$ and
its assembled Gram matrix is nearly singular, with $\lambda_{\min}$ between
$0.007$ and $0.015$, compared with residuals near $0.20$ and $\lambda_{\min}$
near $0.5$ for the trained arm. Training costs $30$--$31$ times the
fixed-arm CPU time, so Table~\ref{tab:production-multiseed} compares
representation quality at equal size, not methods at equal cost.

\paragraph{Comparison at similar CPU time.}
The fixed-feature run of Table~\sref{tab:ablation}, at the production widths,
and the trained arm of Table~\ref{tab:production-multiseed} ran on the same
processor type and used similar CPU time: $1292.5$\,s for the fixed features
with $35$ accepted directions and $8415$ neurons, against $1524$--$1587$\,s
for the trained features with $16$ directions and $1024$ neurons. The trained
arm then reaches $\eta_8=0.048$--$0.087$, against $0.355$ for the fixed
features, that is, a smaller distance by a factor of $4$--$7$ with eight times
fewer neurons. The two runs differ in seeds, stopping rules and assessment
rules (Section~\ref{sec:ablation-results}), so this is an indicative rather
than a controlled equal-cost comparison. Random-feature spaces of much larger
width, and the classical M3 comparator, which computes more accurate controls
in under one second, remain the relevant benchmarks for efficiency.

\FloatBarrier

\subsection{Discussion and limits of the numerical evidence}\label{sec:discussion}

Let $\varepsilon_n^{\mathrm{ref}}=\|f_n^m-f_{\mathrm{ref}}\|_{L^2(Q_\omega)}$. The definitions of $\eta_n$ and $\tau_n^{\mathrm{ref}}$ and the reverse triangle inequality give
\begin{equation}\label{eq:finite-reference-error-comparison}
\bigl|\varepsilon_n^{\mathrm{ref}}-\tau_n^{\mathrm{ref}}\bigr|\leq\eta_n.
\end{equation}
The displayed results for Examples~1--3 give the following approximate ratios at $n=32$:
\[
\begin{array}{c|ccc}
 & \text{Example 1} & \text{Example 2} & \text{Example 3}\\
\hline
\eta_{32}/\tau_{32}^{\mathrm{ref}}
 & 7.12\cdot10^{-4} & 1.54\cdot10^{-3} & 1.11\cdot10^{-4}.
\end{array}
\]
In Example~4 the corresponding ratio $\eta_8/\tau_8^{\mathrm{ref}}$ is $6.6\cdot10^{-2}$: the neural error is still smaller than the truncation error, but no longer negligible. The neural reconstructions thus reproduce the truncated HUM controls of the reference models closely, and at the tested dimensions the total error is governed by truncation. The truncation is substantial: at $n=32$ the terminal residual $R_T$ of the neural control is about one quarter of the uncontrolled residual in Example~1 and about $7\%$ in Example~2, and at $n=8$ in Example~4 it is about $44\%$. Approaching the reference controls would require many more trained modes, up to the several hundred retained in the reference models.

\paragraph{Computational cost.} Table~\ref{tab:cost} compares the costs at the largest tested terminal dimension. When the eigenpairs of a time-independent operator are available, as in Examples~1--3, the modal formula \eqref{eq:direct-modal-adjoint-comparator} gives the adjoint fields in closed form. In Example~2 the truncated modal HUM control then costs about ten seconds, almost all of it for the Ritz pairs, whereas the neural adjoints took $26.8$ hours of accumulated per-mode training time and reproduce that control to $\eta_{32}=3.86\cdot10^{-5}$. The neural method therefore has no computational advantage in these examples; Examples~1--3 measure its accuracy, not its efficiency. In Example~4 the fixed modal formula fails: at $n=8$ the frozen controls differ from the M3 controls by $0.157$, while the neural controls differ from them by $\eta_8=1.08\cdot10^{-2}$. The method thus remains applicable where the fixed modal formula fails. This applicability does not by itself give a computational benefit. With the same Ritz pairs, the Ritz--Galerkin solver M3 computes these controls more accurately in less than one second, against $23.0$ hours of accumulated per-mode neural training time. A computational benefit of the neural construction must therefore be assessed against an appropriate classical solver at matched accuracy. Since every terminal mode is trained separately, the training cost grows linearly with $n$, at about $0.8$ hours per mode in Example~2 and $2.9$ hours per mode in Example~4; the several hundred modes of the reference models would thus require several hundred hours of training. Settings without a spatial basis in which the stiffness and observation integrals are inexpensive, such as complex three-dimensional geometries or rough coefficients, might change this comparison; they have not been tested here.

\begin{table}[!htbp]\centering
\small
\caption{Computational cost at the largest tested terminal dimension: process CPU time for the deterministic stages and accumulated per-mode wall time for the single-thread neural production runs. The last column is the distance in $L^2(Q_\omega)$ to the comparator, modal HUM in Example~2 and M3 in Example~4. Training times come from different machines from the other entries, so the table supports only an order-of-magnitude comparison; the measurement details, the hardware and the excluded interrupted and tuning runs are listed in Supplementary Section~\sref{app:timing}.}
\label{tab:cost}
\setlength{\tabcolsep}{4pt}
\begin{tabular}{@{}clrrrrr@{}}
\toprule
\multicolumn{1}{@{}c}{Ex.} &
\multicolumn{1}{c}{method} &
\multicolumn{1}{c}{$n$} &
\multicolumn{1}{c}{eigenpairs} &
\multicolumn{1}{c}{adjoint fields} &
\multicolumn{1}{c}{assembly, solve} &
\multicolumn{1}{c@{}}{$\|f_n-f_n^{\mathrm{cmp}}\|$}\\
\midrule
2 & modal HUM & 32 & $10.7$\,s & closed form & $2.3$\,ms & --\\
2 & neural & 32 & $10.7$\,s & $26.8$\,h$^{\ast}$ & $164$\,s & $3.86{\cdot}10^{-5}$\\
4 & M3 & 8 & $10.7$\,s & $0.67$\,s & included & --\\
4 & neural & 8 & $10.7$\,s & $23.0$\,h$^{\dagger}$ & $106$\,s & $1.08{\cdot}10^{-2}$\\
\bottomrule

\end{tabular}
\end{table}

\subsection{Limitations and scope}\label{sec:limitations}
We collect here the limitations that qualify the results above.

\paragraph{Certification.} The continuous certificates are evaluated only for
the one-dimensional benchmarks of Section~\ref{sec:certified-benchmark}. The
benchmark of Section~\ref{sec:truncated-benchmark} removes exact
representability, the diagonal Gram matrix and the zero truncation error, but
it keeps a constant coefficient, an explicit observability constant and an
exact Gram matrix that does not couple retained and discarded modes, which
makes the truncation error computable from the data. All of these benchmarks
are therefore far more tractable than the production examples. For Examples~1--4 the
training residuals, ray calculations and finite Gram spectra are diagnostics:
they do not enclose the continuous observability constant, the strong
eigenpair defects (which can exceed the Ritz residuals by nearly an order of
magnitude, Table~\sref{tab:ritz-consistency-diagnostics}), quadrature and
temporal-tail errors, pruning defects, or the forward error required by
Proposition~\ref{prop:final-field-certificate}. Examples~1 and~3 also use the
stored error-assessment quadratures without the frequency pruning and the two
resolved error rules of Examples~2 and~4.

\paragraph{A priori theory.} The sufficient stability condition of
Theorem~\ref{thm:fully-discrete-neural-hum} is not verified for any production
example and would fail at the largest tested $n$ in Example~2 even with the
training residuals taken as exact errors (Section~\ref{sec:example2}). The
convergence rates rest on the residual-capture hypothesis
\eqref{eq:AD-uniform-quality}, which is not proved for any optimizer, and the
logarithmic count of accepted directions does not bound network width,
optimizer iterations, failed attempts or runtime
(Remark~\ref{rem:log-count-scope}).

\paragraph{Optimization and stopping.} The certificate-driven stopping test of
Section~\ref{sec:adaptive-stopping} used a restart policy fixed after a first
failure and three seeds; it is not evidence of robust termination. The
production stopping indicators are not certificates.

\paragraph{Feature training and cost.} The comparisons of trained and random
features use one seed at full budget and three seeds at a reduced budget. The
classical modal and Ritz--Galerkin comparators are much cheaper in all tested
settings, and no regime in which the neural construction is computationally
competitive has been identified. Reusing the adjoints for several target
states is equally available to classical HUM.

\section{Conclusion}
\label{sec:conclusion}

We have analysed the reconstruction of minimum-norm wave controls by the Hilbert Uniqueness Method from adjoint waves that are approximated independently, here by adaptive Galerkin neural networks. Under continuous observability, the reconstruction error splits into terminal-space truncation, adjoint approximation and assembly errors. A Green identity for the wave operator identifies the inconsistency of the moment equations with the interior residual of the approximate adjoints; this yields an orthogonal decomposition of the control error and a posteriori bounds that remain valid when the Gram matrix is assembled and solved inexactly. In one-dimensional benchmarks with known minimum-norm controls, all inputs of these bounds were enclosed by ball arithmetic, and every bound enclosed the true error. Without truncation the residual bounds exceeded it by factors between $9$ and $115$. In a benchmark with sixteen adjoints outside the trial class, a truncated terminal space and a non-diagonal Gram matrix, in which every input was computed from the target data, they exceeded it by factors between about $1.08$ and $6.50$, the larger factors being caused by two adjoints that missed their training tolerance.

In two and three space dimensions the neural controls reproduce the truncated HUM controls of spectral and Ritz--Galerkin reference models closely, and in an example with a time-dependent coefficient they capture an effect that a frozen-coefficient modal construction misses. Trained hidden features reduce the control error substantially relative to random features with the same number of neurons; a comparison at similar CPU time points the same way but is not a controlled experiment. The limitations collected in Section~\ref{sec:limitations} remain: the certificates have been evaluated only in one space dimension with an explicit observability constant, the a priori stability condition is not met at the largest tested dimension, and classical solvers are much cheaper in every tested case. Certified eigenpair, quadrature and assembly bounds for complex geometries, and problems in which no efficient classical spatial basis is available, are the natural next steps.

\section*{Acknowledgments}
The author would like to thank Professor Minh-Binh Tran for suggesting this problem and for many helpful discussions.

\section*{Declarations}
\paragraph{Funding.} The author received no financial support for this research.

\paragraph{Competing interests.} The author has no competing interests to declare.

\paragraph{Data and code availability.} \ifcodereleased The code, trained networks and stored results of Examples~1--4, the verification studies of Section~\ref{sec:certified-benchmark} (\path{example4_pulsating_lens.zip}) and the source of this manuscript are openly available in version~2 of the Zenodo record, \url{https://doi.org/\codeDOIvtwo}. Version~1 (\url{https://doi.org/\codeDOI}) contains Examples~1--3 and a superseded manuscript source (Supplementary Section~\sref{app:code-data}). \else The code, trained networks and stored results of Examples~1--3 are openly available at \url{https://doi.org/\codeDOI}; the manuscript source in that archive is a superseded earlier version (Supplementary Section~\sref{app:code-data}). The package for Example~4 and for the verification studies of Section~\ref{sec:certified-benchmark} (\path{example4_pulsating_lens.zip}) accompanies this submission as supplementary software and will be deposited as a new version of the same record. \fi Reproduction instructions are given in the Supplementary Material.

\paragraph{Use of AI tools.} AI tools were used for language editing, code review and manuscript checking. The code was developed with the assistance of Claude (Anthropic) and ChatGPT (OpenAI). The author takes full responsibility for all content.

\bibliographystyle{plain}
\bibliography{References}
\end{document}